\documentclass[11pt,reqno]{amsart}

\usepackage{amssymb}
\usepackage{eucal}
\usepackage{epstopdf}
\usepackage{mathrsfs}
\usepackage{amsmath}
\usepackage{hyperref}
\usepackage{enumerate}
 \usepackage{fullpage} 
\usepackage{setspace}
\usepackage{color}
  \usepackage{dsfont}
 \usepackage{bbold} 
 \usepackage{todonotes}

\newcommand{\oh}{\mbox{$\frac{1}{2}$}}
\newcommand{\Z}{\mathcal{Z}}

\newcommand{\I}{\mathcal{I}}
\newcommand{\J}{\mathcal{J}}
\newcommand{\K}{\mathcal{K}}
\newcommand{\Ll}{\mathcal{L}}
\newcommand{\s}{\sigma}
\newcommand{\m}{{\bf m}}

\newcommand{\e}{\varepsilon}

\newcommand{\C}{{\mathbb C}}

\newcommand{\N}{{\mathbb N}}
\newcommand{\R}{{\mathbb R}}
\renewcommand{\S}{{\mathcal S}}

\newtheorem{prop}{Proposition}[section]
\newtheorem{thm}[prop]{Theorem}

\newtheorem{lem}[prop]{Lemma}
\newtheorem{defn}{Definition}

\newtheorem*{defn*}{Definition}
\newtheorem{conj}{Conjecture}

\numberwithin{equation}{section}

\begin{document}

\title[Mean values of long Dirichlet polynomials]{Mean values of long Dirichlet polynomials with higher divisor coefficients}

\thanks{This research was supported by NSERC Discovery grants RGPIN-2018-06313 of Alia Hamieh and RGPIN-2020-06032 of Nathan Ng. For part of this research, the first author was supported by the Pacific Institute for the Mathematical Sciences (PIMS) postdoctoral fellowship at the University of Lethbridge.}




\author[Alia Hamieh and Nathan Ng]{Alia Hamieh and Nathan Ng}
\address{University of Northern British Columbia\\ Department of Mathematics and Statistics \\ 3333 University Way\\ Prince George, BC\ V2N4Z9\\ Canada}
\email{alia.hamieh@unbc.ca}
\address{University of Lethbridge \\ Department of Mathematics and Computer Science \\ 4401 University Drive \\ Lethbridge, AB \ T1K 3M4 \\ Canada}
\email{nathan.ng@uleth.ca}

\date{\today}
\begin{abstract}

 In this article, we 
prove an asymptotic formula for  mean values
of long Dirichlet polynomials with higher order shifted divisor functions, assuming a smoothed additive divisor conjecture
for higher order shifted divisor functions. 
As a consequence of this work, we prove special
cases of conjectures of Conrey-Keating \cite{CK3} 
on mean values of long Dirichlet polynomials
with higher order shifted divisor functions as coefficients.
 
\end{abstract}


\maketitle

\section{Introduction}\label{sec:intro}
An important field of research with a long history in analytic number theory is the study of the $2k$-th moments of the Riemann zeta
function, $\zeta(s)$.   These moments are given by 
\begin{equation}
  \label{IkT}
 I_k(T) = \int_{0}^{T} |\zeta(\tfrac{1}{2}+it)|^{2k} \, dt \text{ where } k>0 \text{ and}\; T \ge 1. 
\end{equation}
A  driving force in this field is the conjectural asymptotic
\begin{equation}
  \label{IkTasymp}
 I_k(T) \sim \frac{c_k g_k}{(k^2)!} T (\log T)^{k^2},
\end{equation}
where 
\begin{equation}
 \label{ckgk}
 c_k=\prod_{p}\Big(\Big(1-\frac{1}{p}\Big)^{k^2}\sum_{\alpha=0}^{\infty}\frac{\tau_{k}(p^{\alpha})^2 }{p^{\alpha}}\Big)
 \text{ and }
 g_k = (k^2)! \prod_{j=0}^{k-1} \frac{j!}{(j+k)!}. 
\end{equation}
This asymptotic is only known to be true in the cases  $k=1$ and  $k=2$ as established by Hardy-Littlewood \cite{HL}
and Ingham \cite{In} respectively. 
The conjecture in the form \eqref{IkTasymp} with $c_k$ given in \eqref{ckgk} is folklore.  For a long time, 
the values of $g_k$ were unknown until Keating and Snaith \cite{KS} famously  computed $g_k$ via a random
matrix model and then announced their result at a conference in Vienna in 1998.  
At the same conference Conrey and Gonek announced the conjecture $g_4=24024$, based on mean values of long Dirichlet polynomials. 
Previously, Conrey and Ghosh \cite{CGh} conjectured $g_3 =42$  by studying various mean value formulae for 
$\zeta(s)$ and Dirichlet polynomials. 

The Keating-Snaith conjecture \eqref{IkTasymp} is intimately related to the size of the Riemann zeta function and the Riemann hypothesis.   Littlewood showed that the Riemann hypothesis implies that 
\begin{equation}
  \label{littlewood}
  |\zeta(\tfrac{1}{2}+it)| \ll \exp \Big( 
  \frac{C \log t}{\log \log t} \Big), 
\end{equation}
for some positive constant $C$.
Note that the Lindel\"{o}f hypothesis (LH) is the assertion
\begin{equation}
    \label{lindelof}
    \text{ for all } \e >0, \,  \, 
     |\zeta(\tfrac{1}{2}+it)| \ll_{\e} |t|^{\e}.
\end{equation}
From  Littlewood's bound \eqref{littlewood}
it follows that the Riemann hypothesis implies LH. 
The first non-trivial subconvexity bound was established by Hardy-Littlewood, 
using a method of Weyl (see \cite[Chapter 5]{Ti}):
\begin{equation}
     \label{subconvex}
    \text{ for all } \e >0, \,  \, 
     |\zeta(\tfrac{1}{2}+it)| \ll_{\e} |t|^{\vartheta+\e}
\end{equation}
where $\vartheta= \frac{1}{6}=0.1666 \ldots$. The current record due to Bourgain
is \eqref{subconvex} with  $\vartheta= \frac{13}{84} = 0.1547 \ldots $.   
Hardy and Littlewood introduced the moments $I_k(T)$ in an attempt to solve LH. 
This is since it is known that the bound $I_k(T) \ll_{k,\e} T^{1+\e}$ for all $k \ge 1$ implies \eqref{lindelof}.  Therefore, the Keating-Snaith conjecture \eqref{IkTasymp} implies 
LH.   The Riemann hypothesis has a number of arithmetic consequences. In some instances these arithmetic consequences can be proved only assuming LH. 
For instance,  Ingham \cite{In} showed that LH implies that the gaps between consecutive primes satisfies $p_{n+1}-p_n \ll p_{n}^{\tfrac{1}{2}+\e}.$
One application of a uniform version of the Keating-Snaith conjecture 
is to the maximal size of the Riemann zeta function on the critical 
line.  Farmer, Gonek, and Hughes \cite{FGH} (see also \cite{HR})  conjectured that 
\begin{equation}
    \label{maxzeta} 
    \max_{0 \le t \le T} |\zeta(\tfrac{1}{2}+it)|
    = \exp \Big(  (1+o(1)) \sqrt{\tfrac{1}{2} \log T \log \log T}
    \Big).
\end{equation}
Another reason for studying the moments \eqref{IkT} is that the techniques, tools, and ideas  used in evaluating them can often be useful in the context of moment problems
for other families of $L$-functions. In the last thirty years, there has been a flurry of activity in the study of the distribution and moments of $L$-functions and the distribution of values of $L$-functions. 
For a comprehensive overview of many of the recent advances in the theory see \cite{S}.

In studying the moments \eqref{IkT} it is useful to consider
smoothed and shifted versions of them  given by 
  \begin{equation}
   \label{IIJT}
    \int_{-\infty}^{\infty} \omega(t)
   \Big( \prod_{j=1}^{k}
     \zeta(\oh+a_j +it) 
   \prod_{j=1}^{\ell}
    \zeta(\oh+b_j -it) \Big)
  \, dt,
\end{equation}
where $\{ a_1, \ldots, a_k \}$ and $\{ b_1, \ldots, b_{\ell} \}$ are multisets of complex numbers and   $\omega$ is a suitable 
real or complex valued function.  The idea of  using complex shifting parameters was introduced by Ingham \cite{In}, and
the idea of introducing smoothing weights has long been known and was used by Atkinson \cite{At} and Titchmarsh \cite{Ti}. 
In \cite{HB} a smooth weight $\omega(t)$ was removed and replaced by an indicator function. 
Generalized moments such as   \eqref{IIJT} are known to  reveal the combinatorial structure of the moments \eqref{IkT} (see \cite{CFKRS}).

The mean values \eqref{IkT} and \eqref{IIJT} can be modelled by mean values of long Dirichlet polynomials.

This approach had previously been introduced by 
Conrey-Gonek \cite{CG} and Ivi\'{c} \cite{Iv1}. 
We set 
\begin{equation}
  \label{ABphi} 
  \mathds{A}_{a,\varphi}(s) = \sum_{n =1}^{\infty} \frac{a(n)}{n^s} \varphi \Big( \frac{n}{K} \Big),\quad
  \mathds{B}_{b,\varphi}(s) = \sum_{n =1}^{\infty} \frac{b(n)}{n^s} \varphi \Big( \frac{n}{K} \Big),
\end{equation}
where $K=T^{\theta}$, $\{a(n)\}$ and $\{b(n)\}$ 
are arbitrary sequences, and $\varphi$ is a real valued function.  
Attached to these polynomials is the mean value
\begin{equation}
  \label{DabwK}
     \mathscr{D}_{a,b;  \omega}(K) = \int_{\mathbb{R}} \omega(t)  \mathds{A}_{a,\varphi}(\tfrac{1}{2}+it) 
     \mathds{B}_{b,\varphi}(\tfrac{1}{2}-it)\;dt.
\end{equation}
Such mean values are simple to evaluate when $0 < \theta \le 1$. 
In the case that $\theta >1$, they become harder to evaluate and they are called
mean values of a long Dirichlet polynomial. Goldston and Gonek \cite{GG} 
considered such mean values and provided certain formulae   for $\mathscr{D}_{a,b;  \omega}(K)$
based on correlation sum estimates for $a(n)$ and $b(n)$.
In this article we shall evaluate $\mathscr{D}_{a,b; \omega}(K)$ in the cases that $a(n)$ and $b(n)$ are 
generalized divisor functions, $\theta \in (1,2)$,  and $\omega$ and $\varphi$ are suitably chosen smooth functions (see Section \ref{section:weights} for details). 
 Throughout this article $k$ and $\ell$ denote natural numbers
and  $\I$ and $\J$ denote  multisets of complex numbers given by 
\begin{equation*}
  \label{IJ}
  \I = \{ a_1, \ldots, a_k \} \text{ and }
  \J = \{b_1, \ldots, b_{\ell} \}. 
\end{equation*}
It will be convenient to use the notation 
\begin{equation}
 \label{KL}
 \K =  \{ 1, \ldots, k \} \text{ and }
 \Ll =  \{ 1, \ldots, \ell \}.
\end{equation} 
We shall assume throughout the article that we have the following size condition:
there exists a positive absolute constant $ \delta $ such that 
\begin{equation}
    \label{sizerestrictiondelta}
   |a_i|, |b_j|   \le \delta \text{ for }  i \in \K \text{ and } j \in \Ll.  
\end{equation} 
At times we shall require the more restrictive size conditions
\begin{equation}
  \label{sizerestriction}
 |a_i|, |b_j|   \ll \frac{1}{\log T} \text{ for }  i \in \K \text{ and } j \in \Ll, 
\end{equation}
and \begin{equation}
  \label{sizerestriction-} 
  |a_{i_1} - a_{i_2}| \gg \frac{1}{\log T} 
  \text{ and }
  |b_{j_1}-b_{j_2}| \gg \frac{1}{\log T}
  \text{ for } i_1 \ne i_2 \in \K \text{ and }
  j_1 \ne j_2 \in \Ll,
\end{equation}

where $T \ge 2$ is a parameter.  Note that if $T$ is taken sufficiently large, and  if $\I$ and $\J$ satisfy
\eqref{sizerestriction}, then they will also satisfy \eqref{sizerestrictiondelta}. 
For $k \in \mathbb{N}$, we define the $k$-th divisor function to be 
\[
  \tau_k(n) = \# \{ (n_1, \ldots, n_k) \in \mathbb{N}^k \ | \ n_1 \cdots n_k = n \}
  \text{ for } n \in \mathbb{N}
\]
and for a multiset $\I= \{ a_1, \ldots, a_k \} \subset \mathbb{C}$, 
the shifted divisor function 
\footnote{In the  articles \cite{CK1}, \cite{CK2}, \cite{CK3}, \cite{CK4}, and \cite{CK5},  the authors use the 
notation $\tau_{\I}(n)$ instead of our $\sigma_{\I}(n)$. }
is given by 
\begin{equation}
  \label{shiftdiv}
  \sigma_{\I}(n) = \sum_{d_1 \cdots d_k=n} d_{1}^{-a_1} \cdots d_{k}^{-a_k}. 
\end{equation}
Observe that if $\I=\{0, \ldots, 0\}$, then $\sigma_{\I}(n)=\tau_{k}(n)$ where $k=|\I|$. 
We shall evaluate $ \mathscr{D}_{a,b;\omega}(K)$ in the cases
\begin{equation*}
  a(n) =\tau_k(n) \text{ and }
  b(n) =\tau_{\ell}(n),
\end{equation*}
and
\begin{equation*}
  a(n) = \sigma_{\I}(n) \text{ and }
  b(n) =\sigma_{\J}(n).
\end{equation*}
We shall use the short hand notation
\begin{equation*}
    \mathscr{D}_{k,\ell;\omega}(K) := \mathscr{D}_{\tau_k,\tau_\ell;\omega}(K)
    \text{ and }
      \mathscr{D}_{\I,\J;\omega}(K) := \mathscr{D}_{\sigma_{\I},\sigma_{\J};\omega}(K).
\end{equation*}
In \cite{CG}, Conrey and Gonek gave heuristic arguments which showed how to model the sixth and eighth moments
in terms of $   \mathscr{D}_{k,k;\omega}(K)$ for $k=3,4$.  
We have the following conjecture for   $\mathscr{D}_{k,\ell;\omega}(K)$.  In the case $k=\ell$ this is due 
to Conrey-Gonek \cite[Conjecture 4, p.583]{CG}.
\begin{conj}
  \label{conjCG}
Let $k,\ell \in \mathbb{N}$.  Let $T$ be sufficiently large, $K=T^{1+\eta}$ where $\eta \in (0,1)$.  Then we have
\begin{equation*}
    \mathscr{D}_{k,\ell;\omega}(K)  \sim \frac{a_{k,\ell}}{\Gamma(k\ell+1)} w_{k,\ell} \left(1+\eta \right)T (\log T)^{k\ell},
\end{equation*}
where
\begin{equation*}
a_{k,\ell}=\prod_{p}\left(\left(1-\frac{1}{p}\right)^{k\ell}\sum_{\alpha=0}^{\infty}\frac{\tau_{k}(p^{\alpha})\tau_{\ell}(p^{\alpha})}{p^{\alpha}}\right),
\end{equation*}
\begin{equation}
  \label{wklx}
 w_{k,\ell}(x) = x^{k\ell}  \Big( 1- \sum_{n=0}^{k\ell-1} \binom{k\ell}{n+1}\gamma_{k,\ell}(n) (-1)^{n+\ell+k} (1-x^{-n-1})  \Big),
\end{equation}
\begin{equation}
  \label{gammakln}
  \gamma_{k,\ell}(n) =  \sum_{i=1}^{\ell} \sum_{j=1}^{k}  \binom{\ell}{i} \binom{k}{j}
   \binom{n-1}{i+j-2 }\binom{i+j-2}{i+k-\ell-1},
\end{equation}
for $n\in\mathbb{N}$, and \[\gamma_{k,\ell}(0)=\sum_{i=1}^{\ell} \sum_{j=1}^{k} (-1)^{k+\ell+i+j} \binom{\ell}{i} \binom{k}{j}
   \binom{i+j-2}{i+k-\ell-1}.\]
\end{conj}
\label{w33}
Let   $\eta \in (0,1)$  and $\omega=\mathbb{1}_{[T,2T]}$.  In \cite{CG} an argument with the classical 
approximate functional equation for $\zeta^{k}(s)$ is given which suggests the asymptotics 
\begin{align}
 \label{zeta6dp}
\int_{T}^{2T} |\zeta(\tfrac{1}{2}+it)|^{6} \, dt & \sim 
  \mathscr{D}_{3,3;\omega}(T^{1+\eta}) + \mathscr{D}_{3,3;\omega}(T^{2-\eta}) 
  \sim 
 ( w_{3,3} (1+ \eta)
  + w_{3,3} ( 2-\eta))   \frac{c_3}{9!}  T (\log T)^{9} 
\end{align}
and 
\begin{equation}
 \label{zeta8dp}
  \int_{T}^{2T} |\zeta(\tfrac{1}{2}+it)|^{8} \, dt   \sim 
  2\mathscr{D}_{4,4;\omega}(T^{2})  
   \sim 2 w_{4,4} ( 2)   \frac{c_4}{16!}  T (\log T)^{16} .
\end{equation}
In fact, one can verify by straightforward computations that
\[
    w_{3,3} (1+ \eta)
  + w_{3,3} ( 2-\eta) = 42=g_3 \text{ and }
  w_{4,4}(2) = 24024=g_4,
  \footnote{Note that $w_{3,3}(x) = -2x^9+27x^8-324x^7+2268x^6-8694x^5+19278x^4-25452x^3+19764x^2-8343x+1479$  and \\
  $w_{4,4}(x) = 
 -3x^{16}+16x^{15}-1320x^{14}+14000x^{13}-78260x^{12}+179088x^{11} -152152x^{10}+11440x^{9}+12870x^{8}+11440x^{7}+8008x^6+4368x^5+1820x^4+560x^3+120x^2+16x+1$.}
\]
where $g_k$ is as defined in \eqref{ckgk}. Thus  \eqref{zeta6dp} and \eqref{zeta8dp} agree with the Keating-Snaith conjecture. 
This was one of the main motivations for evaluating mean values of long Dirichlet polynomials with divisor coefficients.  In \cite{Ng}, the heuristic \eqref{zeta6dp} is made precise assuming the $3$-$3$ additive divisor conjecture  (see Conjecture \ref{conj:shifted-sum} below) with error term $O(P^{C} X^{\frac{2}{3}-\delta})$, for some $\delta >0$, uniformly for $|r| \le \sqrt{X}$. In \cite{NSW}, \eqref{zeta8dp} is made precise in the same way where the $4$-$4$ additive divisor conjecture is required with an 
error term $O(P^{C} X^{\frac{1}{2}+\e_1})$, uniformly for $|r| \le X^{1-\e_2}$
where $\e_1, \e_2$ are arbitrarily small positive constants.  Note that the argument of Conrey-Gonek using the additive divisor sums \eqref{DfIJr} does not seem to extend to the moments $I_k(T)$ with $k >4$. 
In order to further understand $I_k(T)$, Conrey and Keating undertook an extensive study
\cite{CK1}, \cite{CK2}, \cite{CK3}, \cite{CK4}, \cite{CK5}
of $\mathscr{D}_{\I,\J,\omega}(K)$, the mean values of Dirichlet polynomials with shifted divisor functions.  This work has led to the consideration of more complicated additive divisor 
sums.  Furthermore, they have formulated a conjecture on the asymptotic size of 
$\mathscr{D}_{\I,\J,\omega}(K)$.

In order to state their conjecture we must introduce some notation and definitions. 
\begin{defn} 
Let $\I,\J$ be finite multisets of complex numbers. We define $\mathcal{B}(\I,\J)$ as the series 
\begin{align}\label{BIJ}\mathcal{B}(\I,\J)&=\sum_{n=1}^{\infty}\frac{\sigma_{\I}(n)\sigma_{\J}(n)}{n},\end{align} if the series converges (for example, when $\Re(a),\Re(b)>0$ for all $a\in\I$ and $b\in \J$), and by analytic continuation otherwise.
\end{defn}
Observe that when the series \eqref{BIJ} converges, we use the multiplicativity of $\sigma_{\I}\sigma_{\J}$ to write
\begin{align}\label{BIJeuler}
  \mathcal{B}(\I,\J)&=\prod_{p}\sum_{u=0}^{\infty}\frac{\sigma_{\I}(p^u)\sigma_{\J}(p^u)}{p^{u}}.\end{align}

Upon factoring out $\prod_{p}  \prod_{\substack{ i \in \K \\ j \in \Ll}  }  (1 - p^{-1-a_i -b_j})^{-1}$  from the right hand side of \eqref{BIJeuler}, we obtain 
\begin{equation*}
\begin{split}
  \mathcal{B}(\I,\J)&=\prod_p\prod_{\substack{i \in \K \\ j \in \Ll}}(1 - p^{-1-a_i -b_j})^{-1}\prod_{p}\prod_{\substack{i \in \K \\ j \in \Ll}}(1 - p^{-1-a_i -b_j})\sum_{u=0}^{\infty}\frac{\sigma_{\I}(p^u)\sigma_{\J}(p^u)}{p^{u}}.
\end{split}\end{equation*}

\begin{defn}\label{ZA}
For $p$ prime and $s \in \C$, we set $z_p(s) = (1-p^{-s})^{-1}$. 
 Attached to the local factors $z_{p}(s)$, we define 
\begin{align}
   \label{ZIJ}
     \Z(\I,\J) & =  \prod_{p}\prod_{\substack{i\in \K\\j\in\Ll}} z_{p}(1+a_i+b_j), \\
     \label{AIJ}
      \mathcal{A}(\I,\J) & =  \prod_{p} \prod_{\substack{i\in \K\\j\in\Ll}}  z_{p}^{-1}(1+a_i+b_j)\sum_{u=0}^{\infty}\frac{\sigma_{\I}(p^u)\sigma_{\J}(p^u)}{p^{u}}. 
 \end{align}

\end{defn}
Observe that we have 
\begin{equation}
  \label{ZIJid}
   \mathcal{Z}(\I,\J) = \prod_{i \in \K, j \in \Ll} \zeta(1+a_i+b_j),
\end{equation}
and 
\begin{equation}
  \label{BIJid}
   \mathcal{B}(\I,\J)  =   \mathcal{A}(\I,\J) \mathcal{Z}(\I,\J).
\end{equation}
We now introduce some useful notation on set operations. 
\begin{defn}\label{setops}
Given a multiset $U= \{\alpha_1, \ldots, \alpha_n \}$ and $w \in \mathbb{C}$, we 
define 
$U_{w} := U+\{ w \}= \{\alpha_1+w, \ldots, \alpha_n+w \}$. We also set $-U=\{-\alpha_1, \ldots, -\alpha_n \}$.With this notation, observe that we have the identity
\begin{equation}
    \label{sigmaidentity}
    \sigma_{U_w}(n) = n^{-w} \sigma_{U}(n).
\end{equation}
\end{defn} 

We can now state the Conrey-Keating conjectures for the mean values $ \mathscr{D}_{\I,\J;\omega}(K)$ (see \cite[pages~739-740]{CK3}). 
\begin{conj}[Conrey-Keating] \label{ck1}
Let $K=T^{\theta}$ with $\theta >0$.  Then for $T$ sufficiently large
\begin{align}\label{ck1eq}
       & \mathscr{D}_{\I,\J;\omega}(K)
       = \int_{0}^{\infty} \omega(t) 
       \frac{1}{(2 \pi i)^2} \int_{(c_1)} \int_{(c_2)} \Phi(s_1)\Phi(s_2)K^{s_1+s_2}
        \times \nonumber \\
      &  \sum_{ \substack{U \subset \I, V \subset \J \\ |U|=|V|   } }
       \Big( \frac{t}{2 \pi} \Big)^{-\sum_{x\in U, y \in V}(x+s_1+y+s_2)}  
       \mathcal{B}((\I_{s_2} \backslash U_{s_2}) \cup (-V_{s_1}),(\J_{s_1} 
       \backslash V_{s_1}) \cup (-U_{s_2})) \;ds_1 ds_2 dt
       +o(T),
\end{align}
where $c_1, c_2 >0$ and $\Phi$ is the Mellin tranform of $\varphi$.  
\end{conj}

In the above conjecture, we used the terminology of \cite{CK3}; for example, by $(\I_{s_2} \backslash U_{s_2}) \cup (-V_{s_1})$ we mean the following: remove the elements of $U_{s_2}$ from $\I_{s_2}$ and then include the negatives of the elements of $V_{s_1}$. The notation $(\J_{s_1} 
\backslash V_{s_1}) \cup (-U_{s_2})$ can be explained similarly. Since the subsets $U$ and $V$ have equal cardinalities, this process is referred to as swapping equal numbers of elements between $\I_{s_2}$ and $\J_{s_1}$. The cardinality $|U|$ is referred to as the number of swaps in the associated term.

To give more insight on the terms $\mathcal{B}((\I_{s_2} \backslash U_{s_2}) \cup (-V_{s_1}),(\J_{s_1} 
       \backslash V_{s_1}) \cup (-U_{s_2}))$ appearing in  \eqref{ck1eq}, we give precise formulae when $|\I|=|\J|=2$. Suppose that $\I=\{a_1,a_2\}$ and $\J=\{b_1,b_2\}$. The term $\mathcal{B}((\I_{s_2} \backslash U_{s_2}) \cup (-V_{s_1}),(\J_{s_1} 
       \backslash V_{s_1}) \cup (-U_{s_2}))$ corresponding to $|U|=|V|=0$ simplifies to \begin{align}\label{BIJcase2}
          & \mathcal{B}(\{a_1+s_2,a_2+s_2\},\{b_1+s_1,b_2+s_1\})\nonumber\\&=\frac{\zeta(1+s_1+s_2+a_1+b_1)\zeta(1+s_1+s_2+a_1+b_2)\zeta(1+s_1+s_2+a_2+b_1)\zeta(1+s_1+s_2+a_2+b_2)}{\zeta(2+2s_1+2s_2+a_1+a_2+b_1+b_2)},\end{align}
       which can be derived from a formula of Ramanujan (see \cite[Eq~1.3.3]{Ti}). To describe the terms corresponding to $|U|=|V|=1$,  we consider the case $U=\{a_1\}$ and $V=\{b_1\}$ as an example. In this case, the term $\mathcal{B}((\I_{s_2} \backslash U_{s_2}) \cup (-V_{s_1}),(\J_{s_1} 
       \backslash V_{s_1}) \cup (-U_{s_2}))$ simplifies to 
       \begin{align}
           &\mathcal{B}(\{a_2+s_2,-b_1-s_1\},\{-a_1-s_2,b_2+s_1\})\nonumber\\&=\frac{\zeta(1+a_2-a_1)\zeta(1+b_2-b_1)\zeta(1+s_1+s_2+a_2+b_2)\zeta(1-s_1-s_2-a_1-b_1)}{\zeta(2+a_2-a_1+b_2-b_1)}.
       \end{align}
       
       The remaining three terms corresponding to $|U|=|V|=1$ can be computed similarly. When $|U|=|V|=2$, we see that the term  $\mathcal{B}((\I_{s_2} \backslash U_{s_2}) \cup (-V_{s_1}),(\J_{s_1} 
       \backslash V_{s_1}) \cup (-U_{s_2}))$ simplifies to 
       \begin{align}
          &\mathcal{B}(\{-b_1-s_1,-b_2-s_1\},\{-a_1-s_2,-a_2-s_2\})\nonumber \\&=\frac{\zeta(1-s_1-s_2-a_1-b_1)\zeta(1-s_2-s_2-a_2-b_1)\zeta(1-s_1-s_2-a_1-b_2)\zeta(1-s_1-s_2-a_2-b_2)}{\zeta(2-2s_1-2s_2-a_1-a_1-b_1-b_2)}
       \end{align}
  Going back to \eqref{ck1eq} and using \eqref{sizerestriction}, we note that the inner double integrand is 
of size  \[|\Phi(s_1)\Phi(s_2)|(K/T^{|U|})^{\Re(s_1)+\Re(s_2)}.\] Thus, the size of $K$ determines which swaps contribute to our main term.
In the special case when $K=o(T^2)$,  we only get a contribution
from $U,V$ such that $|U|=|V| \le 1$.   In particular, if $K=o(T^2)$ and $|U|=|V|\ge 2$, then 
 these terms do not contribute to the main term. 
The contribution of the terms with $|U|=|V|=0$ to the integral is 
 \begin{equation}
 \begin{split}
    \label{M0swaps}
  \mathcal{M}_{0, \I, \J; \omega}(K) &= \int_{0}^{\infty} \omega(t) 
       \frac{1}{(2 \pi i)^2} \int_{(c_1)} \int_{(c_2)} \Phi(s_1)\Phi(s_2)K^{s_1+s_2}
       \mathcal{B}(\I_{s_2} ,\J_{s_1} ) \;ds_1 ds_2 dt
      \\&=\frac{\widehat{\omega}(0)}{2\pi i}\int_{(c_1+c_2)}K^{s}\Phi_{2}(s)\mathcal{B}\left(\I_{s},\J\right)\;ds,
      \end{split}
 \end{equation}
where $\Phi_2$ is defined in \eqref{phi2}  in Section \ref{section:weights}. The second equality in \eqref{M0swaps} is obtained by applying the change of variables $s=s_2+s_2$ and observing that $\mathcal{B}(\I_{s_2},\J_{s_1})=\mathcal{B}(\I_{s_1+s_2},\J)$, which follows from
two applications of \eqref{sigmaidentity}.
The contribution of the terms with $|U|=|V|=1$ is 
     \begin{equation*}
     \begin{split}
     \mathcal{M}_{1,\I, \J; \omega}(K)  & :=\int_{0}^{\infty} \omega(t) 
       \frac{1}{(2 \pi i)^2} \int_{(c_1)} \int_{(c_2)} \Phi(s_1)\Phi(s_2)K^{s_1+s_2}
          \sum_{i \in \K,j \in \Ll  }
       \Big( \frac{2\pi}{t} \Big)^{a_i+b_j+s_1+s_2}  
      \nonumber\\&\hspace{2em}\times \mathcal{B}((\I_{s_2} \backslash \{a_i+s_2\}) \cup \{-b_j-s_1\},(\J_{s_1} 
       \backslash \{b_j+s_1\}) \cup \{-a_i-s_2\}) \;ds_1 ds_2 dt.
     \end{split}
     \end{equation*}
     Observe that (see \cite[page~740]{CK3})
     \begin{align*}
     &\mathcal{Z}((\I_{s_2} \backslash \{a_i+s_2\}) \cup \{-b_j-s_1\},(\J_{s_1} 
       \backslash \{b_j+s_1\}) \cup \{-a_i-s_2\})\\&=\mathcal{Z}(\I \setminus \{a_i\},\{-a_i\})\mathcal{Z}(\{-b_j\},\J \setminus \{ b_j \})\mathcal{Z}((\I \setminus \{a_i\})+s,(\J \setminus \{b_j\})
     )\zeta(1-a_i-b_j-s).
         \end{align*}
     By \eqref{BIJid} and the change of variables $s=s_1+s_2$, we get
   \begin{equation}
     \begin{split}
      \label{M1swaps} \mathcal{M}_{1,\I, \J; \omega}(K)  &= \int_{0}^{\infty} \omega(t) \sum_{i \in \K,j \in \Ll  }
       \Big( \frac{t}{2 \pi} \Big)^{-a_i -b_j } \mathcal{Z}(\I \setminus \{a_i\},\{-a_i\})\mathcal{Z}(\{-b_j\},\J \setminus \{ b_j \})
     \\&\hspace{2em}\times  \frac{1}{2 \pi i}  \int_{(c)} \Phi_2(s)\left(\frac{2\pi K}{t}\right)^{s}\mathcal{Z}((\I \setminus \{a_i\})+s,(\J \setminus \{b_j\})
     )\zeta(1-a_i-b_j-s)\\&\hspace{5em}\times\mathcal{A}( (\I \setminus \{a_i\}) \cup \{ -b_j-s \},   
     ((\J \setminus \{b_j\})+s)\cup\{-a_i\})\;ds\;dt.
\end{split}
\end{equation}
Based on these observations, Conjecture \ref{ck1} simplifies as follows in the case $K=o(T^2)$. 
\begin{conj}[Conrey-Keating] \label{ck2}
If $T \ll K =o(T^2)$, then 
\begin{align*}
        \mathscr{D}_{\I,\J;\omega}(K) =   \mathcal{M}_{0, \I, \J; \omega}(K)+  \mathcal{M}_{1,\I, \J; \omega}(K) 
       +o(T)
\end{align*}
where $ \mathcal{M}_{0,\I, \J; \omega}(K) $ and $ \mathcal{M}_{1,\I, \J; \omega}(K) $ are given in \eqref{M0swaps} and \eqref{M1swaps}. 
\end{conj}
A key goal of this article is to establish Conjecture \ref{ck2} under the assumption of an averaged additive divisor conjecture which provides an asymptotic formula for certain smoothed additive divisor sums (see Theorem \ref{mainthm}). Let us now introduce these additive divisor sums. 
We put 
\begin{equation}
  \label{DfIJr}
D_{f; \I,\J}(r)=\sum_{\substack{m,n\geq1\\m-n=r}}\sigma_{\I}(m)\sigma_{\J}(n)f(m,n).
\end{equation}
Moreover, the partial derivatives of $f$ satisfy growth conditions. 
That is,  there exist $X,Y,$ and $P \ge 1$   such that 
\begin{equation}
  \label{fsupport}
  \text{support}(f) \subset [X,2X] \times [Y,2Y]
\end{equation}
and
\begin{equation} \label{fcond} 
 x^{m} y^{n} f^{(m,n)}(x,y) \ll_{m,n}
P^{m+n}. 
\end{equation}

Before we state a conjectural asymptotic formula for the shifted convolution sum $D_{f;\I,\J}(r)$, we need to introduce the following definition.
\begin{defn}\label{def:functions}
Let $A=\{a_{1},\cdots,a_{m}\}$ be a finite multiset of complex numbers and $s\in\mathbb{C}$. We define two multiplicative functions $n\mapsto g_{A}(s,n)$ and $n\mapsto G_{A}(s,n)$ by
\begin{equation}\label{eqn:g-mult}
  g_{A}(s,n) = \prod_{p^{e}||n}\frac{\sum_{j=0}^{\infty} \frac{\sigma_{A}(p^{j+e})}{p^{js}}}{ \sum_{j=0}^{\infty} \frac{\sigma_{A}(p^j)}{p^{js}}} 
\end{equation}
and
\begin{equation}\label{eqn:G-mult}
G_{A}(s,n) = \sum_{d \mid n} \frac{\mu(d) d^s}{\phi(d)}
  \sum_{e \mid d} \frac{\mu(e)}{e^s} g_{A} \Big(s, \frac{ne}{d} \Big).
\end{equation}
Notice that, for $n\in\mathbb{N}$ we have $\sum_{j=1}^{\infty}\frac{\sigma_{A}(jn)}{j^{s}}=g_{A}(s,n)\prod_{a\in A}\zeta(s+a)$.
\end{defn}
Simpler expressions for $G_{A}(s,n)$ can be derived from  Lemmas \ref{Gmult}, Lemma \ref{comb}, and Lemma \ref{GIk}
below.

We are now prepared to state the averaged additive divisor conjecture.  
\begin{conj} \label{conj:shifted-sum}\text{($k$-$\ell$ Additive divisor conjecture)}
Let $k, \ell \in \mathbb{N}$.  
There exists a triple $(\vartheta_{k,\ell}, C_{k,\ell},\beta_{k,\ell}) \in [\frac{1}{2}, 1) \times [0,\infty) \times (0,1]$ for which 
the following (henceforth to be referred to as $\mathcal{AD}_{k,\ell}(\vartheta_{k,\ell},C_{k,\ell},\beta_{k,\ell})$, or the 
`additive divisor hypothesis') holds.  Let $\varepsilon$  be  a positive
absolute constant. 
Let $P > 1$, $H\geq1$, and let $X,Y > \frac{1}{2}$ satisfy $Y \asymp X$. For each integer $r$ with $1\leq|r|\leq H$, let $f_r$
be a smooth function satisfying \eqref{fsupport} and \eqref{fcond}, and suppose 
$\I = \{ a_1, a_2, \ldots, a_k \}$ and $\J = \{ b_1, \ldots, b_{\ell} \}$ are sets of distinct complex numbers 
satisfying $|a_i|, |b_j| \ll (\log X)^{-1}$ where
$ i \in \{ 1, \ldots, k \}$ and  $j \in \{ 1, \ldots, \ell \}$ (where the implicit constants are absolute).
Then, in those cases where $X$ is sufficiently large (in absolute terms), one has
  \begin{align*}
   &   D_{f_r;\I,\J}(r)  =   \sum_{i_1=1}^{k} \sum_{i_2=1}^{\ell} 
    \prod_{j_1 \ne i_1} \zeta(1-a_{i_1}+a_{j_1})   \prod_{j_2 \ne i_2} \zeta(1-b_{i_2}+b_{j_2})   \\
 &\hspace{1em} \times \sum_{q=1}^{\infty} \frac{c_{q}(r)G_{\I}(1-a_{i_1},q)G_{\J}(1-b_{i_2},q)  }{q^{2-a_{i_1}-b_{i_2}}}
  \int_{\max(0,r)}^{\infty}  f_r(x,x-r) x^{-a_{i_1}}(x-r)^{-b_{i_2}} dx 
+ \Delta_{f_r;\I,\J}(r), 
 \end{align*}
 where \[\sum_{1\leq|r|\leq H}\Delta_{f_r;\I,\J}(r)=O\left(HP^{C_{k,\ell}} X^{\vartheta_{k,\ell}+\varepsilon}\right)\]
 uniformly for $1\leq H \ll X^{\beta_{k,\ell}}$.
 \end{conj}
\noindent {\bf Remarks}. 
\begin{enumerate}
\item Note that  $ c_{q}(r)= \sum_{\substack{a=1 \\ (a,\ell)=1}}^{q}
e ( \frac{ar}{q})$
 is Ramanujan's sum where $e(\theta) := e^{2 \pi i \theta}$. 
\item It should be observed that from Lemma \ref{GIk} below, we can see that 
\[
  G_{\I}(1-a_{i_1},q) \approx  \sigma_{\I \backslash \{ a_{i_1} \}} (q)
  \text{ and }
  G_{\J}(1-b_{i_2},q) \approx  \sigma_{\J \backslash \{ b_{j_2} \}} (q). 
\]
\item The main term in the above conjecture can be derived using Duke, Friedlander, and Iwaniec's $\delta$-method \cite{DFI}. The conjecture provides a bound for the error term on average over $r$ which is sufficient for our purposes. The reader is referred to \cite{Iv2} and \cite{Iv3} among other references for a treatment of the additive divisor conjecture on average.
\item In the case $|\I|=|\J| =2$, Hughes and Young \cite[page~218]{HY} have proven that this holds with
$\vartheta_{2,2} = \frac{3}{4}$, $C_{2,2} = \frac{5}{4}$, and $\beta_{2,2}=1$ (even without the averaging over $r$). 
The main term in the result of Hughes and Young can be established by using
Duke, Friedlander, and Iwaniec's \cite{DFI} $\delta$-method. 
\item Using work of Aryan \cite{Ar} and Topacogullari  \cite{T1}, it may be possible to establish
$\vartheta_{2,2} = \frac{1}{2}+\vartheta_0$, where $\vartheta_0$ is the current best bound 
for the Ramanujan conjecture.  
\item Topacogullari  \cite{T1}, \cite{T2} and Drappeau \cite{Dr} have recently established asymptotic formula for the additive  divisor sums 
$D_{k,\ell}(x,r)= \sum_{n \le x} \tau_k(n) \tau_{\ell}(n+r)$ in the cases $k \ge 3$,  $\ell =2$
where $0 \ne r \in \mathbb{Z}$. It is likely that their work will lead to the additive divisor conjecture
in the case $\ell =2$ with some $\vartheta_{k,2} < 1$ and $C_{k,2} >0$. 
\item It has been conjectured by Conrey and Keating \cite[p.740]{CK3} that $k$-$\ell$ additive divisor conjecture in the unsmoothed case holds with 
$\vartheta_{k,\ell} =\frac{1}{2}$ and $\beta_{k,\ell}=1-\varepsilon_0$, for sufficiently small $\epsilon_0$.   It is thus reasonable to expect that
$\mathcal{AD}_{k,\ell}(\tfrac{1}{2},C_{k,\ell},1-\epsilon_0)$ holds for some $C_{k,\ell}>0$ and $\epsilon_0 >0$.
\end{enumerate}

The main goal of this paper is to prove that Conjecture \ref{conj:shifted-sum} implies Conjecture \ref{ck2}. More precisely, we establish the following theorem.
\begin{thm}\label{mainthm}
Let $|\I|=k$ and $|\J|=\ell$ with $k,\ell \ge 2$, and suppose that $\I$ and $\J$ satisfy \eqref{sizerestriction} and \eqref{sizerestriction-}.  Assume that 
$\mathcal{AD}_{k,\ell}(\vartheta_{k,\ell},C_{k,\ell},\beta_{k,\ell})$ holds for some triple $(\vartheta_{k,\ell}, C_{k,\ell},\beta_{k,\ell}) \in [\frac{1}{2}, 1) \times [0,\infty) \times (0,1]$. Let $K=T^{1+\eta}$ with $\eta>0$, and let $\omega$ satisfy \eqref{cond1}, \eqref{cond2}, and \eqref{cond3} with $b>\frac{(1-\beta_{k,\ell})(1+\eta)}{1-\epsilon}$ and $0<\epsilon<1$. Then we have for any $\varepsilon >0$,  
\begin{equation*}
    \mathscr{D}_{\I,\J;\omega}(K) =   \mathcal{M}_{0, \I, \J; \omega}(K)  +   \mathcal{M}_{1, \I, \J; \omega}(K) 
    +   O \Big( K^{\vartheta_{k,\ell}+\e} \Big( \frac{T}{T_0} \Big)^{1+C_{k,\ell}} \Big).
\end{equation*}
\end{thm}

\noindent {\bf Remarks}. 
\begin{enumerate}
\item This result provides a rigorous proof of some of the arguments in \cite{CK3}. 
The main difference is that in \cite{CK3} the authors focus on the main terms without providing bounds for any error terms. 
Another  key difference is that in \cite{CK3} Perron's formula is applied twice whereas we make use of Mellin inversion. 
\item  In order for this result to be non-trivial the error term needs to be $o(T)$.  
This is the case if $\omega$ satisfies \eqref{cond3} below with
 $b>\frac{C_{k,\ell}+(\vartheta_{k,\ell}+\varepsilon)(\eta+1)}{1+C_{k,\ell}}$.  Observe that since we require $b \le 1$, this condition implies $\eta<\frac{1}{\vartheta_{k,\ell}}-1$.
 If the additive divisor conjecture is true with $\vartheta_{k,\ell}= \frac{1}{2}$, then this theorem allows one to take Dirichlet polynomials with length $K=T^{c}$, for any $c <2$. 
\item In the case $k=\ell=2$, this is an unconditional theorem, due to the work of Hughes-Young \cite{HY}
as they have established $\mathcal{AD}_{2,2}(\tfrac{3}{4},\tfrac{5}{4},1)$.  In this case, we have an asymptotic formula for $0 < \eta < \frac{1}{3}$.   Using ideas from \cite{Ar}, it may be possible to increase the range to $\eta \in (0, 0.61)$.  Furthermore, 
if the Ramanujan conjecture on the size of the Fourier coefficients of Maass forms is true, then the
range can be increased to $\eta \in (0,1)$.  
\item We expect to establish this theorem unconditionally in the cases $k \ge 3$ and $\ell=2$, by using
techniques from \cite{T1}, \cite{T2}, and \cite{Dr} and this is current work in progress.  
\end{enumerate}

As a consequence of the work in this article, we deduce in an accompanying article \cite{HN} the special case $k=\ell=2$ and $\eta \in (0,\frac{1}{3})$ of Conjecture \ref{conjCG} with all lower order terms and a power savings
in the error term.   Moreover, we expect to deduce a version of Conjecture  \ref{conjCG} with the full main term and a  power savings error term for all $k \ge 2$ and $\ell=2$ for 
some range of $\eta$.  We have 
\begin{thm} \label{thm22}
Let $K=T^{1+\eta}$ with $0 <\eta<\frac13$. Let $\omega$ satisfy \eqref{cond1}, \eqref{cond2} and \eqref{cond3} with $b>\frac{5+3(\eta+1)}{9}$. Then we have  
\begin{equation*}
      \mathscr{D}_{2,2;\omega}(K) = 
      \int_{-\infty}^{\infty} \omega(t) \Bigg( \sum_{i=0}^{4} Q_i(\log K, \log \tfrac{t}{2 \pi} )  \Bigg) \, dt 
      + O\left(T^{\frac34(1+\eta)+\e}\left(\frac{T}{T_0}\right)^{\frac94}+T^{1-\frac{\eta}{2}}\right) ,\end{equation*}
 where the $Q_i(x,y) \in \mathbb{R}[x,y]$ are polynomials of degree $i$ and the leading term is given by 
 \begin{equation*}
   Q_4(x,y) =  \frac{1}{\zeta(2)}
 \cdot \frac{1}{4!} (-x^4 +8x^3y-24x^2 y^2 +32 xy^3 -14 y^4). 
 \end{equation*}
 Precise formulae for the  $Q_i$ are given in \cite{HN}. 
  \end{thm}
 We note here that $Q_4(x,y)=\frac{1}{\zeta(2)}\frac{1}{4!}y^4w_{2,2}(x/y)$, where $w_{2,2}$ is the polynomial given in \eqref{wklx} with $k=\ell=2$.


We now explain how our results relate to the previous literature on mean values of Dirichlet polynomials. Setting $\varphi_{0} = \mathbb{1}_{[0,1]}$ in \eqref{ABphi}
we get the following Dirichlet polynomials associated to the real sequences $\{a(n)\}$ and $\{b(n)\}$:
\begin{equation}
  \label{AB} 
 \mathds{A}(s) :=  \mathds{A}_{a,\varphi_{0}}(s) = \sum_{n \le K} \frac{a(n)}{n^s} \text{ and } 
  \mathds{B}(s):= \mathds{B}_{b,\varphi_{0}}(s) = \sum_{n \le K} \frac{b(n)}{n^s}.
\end{equation}
Set $\omega(t)$ to be $\omega_0(t) = \mathbb{1}_{[0,T]}(t)$ \footnote{ $\mathbb{1}_B(t)$ denotes the indicator function corresponding to $B \subset \mathbb{R}$.} in \eqref{DabwK}.
A standard tool in analytic number theory is the mean value estimate
\begin{equation*}
   \mathcal{D}_{a,a;\omega_0} =\int_{\mathbb{R}}\omega_0(t)\mathds{A}(\tfrac{1}{2}+it)|^2\;dt= \int_{0}^{T} | \mathds{A}(\tfrac{1}{2}+it)|^2\;dt =  \sum_{n \le K} \frac{a(n)^2}{n} (T + O(n)),
\end{equation*}
which follows from the work of Montgomery and Vaughan \cite[Corollary 3]{MV}. 
When $K=o(T)$ this implies
$\mathcal{D}_{a,a;\omega_0} \sim T  \sum_{n \le K} \frac{a(n)^2}{n}$,
and this is referred to as the `diagonal contribution'. In the case that $K \gg T$, this is not always the 
correct asymptotic.  Goldston and Gonek \cite{GG} studied  $\mathcal{D}_{a,b;\omega}$
 when $K \gg T$ for a certain smooth weight 
 $\omega$  supported in $[c_1T, c_2T]$ with $0<c_1 < c_2$. They showed that 
$ \mathcal{D}_{a,b;\omega}$ is intimately related to the correlation sums 
$\mathscr{C}_{a,b}(x,r) =  \sum_{n \le x} a(n) b(n+r)$.
They assume uniform formulae of the type
$ \mathscr{C}_{a,b}(x,r) =  \mathscr{M}_{a,b}(x,r) + \mathscr{E}_{a,b}(x,r)$
where $\mathscr{M}_{a,b}(x,r)$ is a main term and $ \mathscr{E}_{a,b}(x,r)$ is an error term.  
They show that when $K \gg T$ the main term of  $\mathcal{D}_{a,b;\omega}$ also contains an  additional 
off-diagonal contribution which arises from certain averages of $ \mathscr{M}_{a,b}(x,r)$
and that the error term is related to averages of $\mathscr{E}_{a,b;h}(x)$
(see \cite[Theorems 1,2, Corollary 1]{GG}).  

In this article we only consider   $\mathscr{D}_{a,b;  \omega}(K)$
in the case $a=\sigma_{\I}$ and $b=\sigma_{\J}$.  A key difference in our approach 
as opposed to \cite{GG} is that we make use of the smoothed additive sums $D_{f;\I,\J}(r)$ rather  than unsmoothed sums $\mathscr{C}_{a,b}(x,r)$.   
We handle the off-diagonals by using smooth partitions of unity so that we can deal with additive divisor sums 
of the type  $D_{f;\I,\J}(r)$.  Goldston and Gonek \cite{GG} treat the off-diagonals via partial summation and the additive divisor sums $\mathscr{C}_{a,b}(x,r)$.   A disadvantage of using the unsmoothed correlation sums 
is that most proofs of asymptotic formula of  $\mathscr{C}_{a,b}(x,r)$ begin with considering a smoothed version of such sums.  At the end of such arguments the smooth function is removed and this increases
the size of the error term $\mathscr{E}_{a,b}(x,r)$.  

It is natural to wonder if the mean values $ \mathscr{D}_{a,b;  \omega}(K)$ can be computed for other sequences.  The main fact we have used is that generalized
divisor functions are coefficients of $L$-functions (in the sense of Selberg).  Here they correspond to the non-primitive $L$-functions 
$\zeta(s+a_1) \cdots \zeta(s+a_k)$ and $\zeta(s+b_1) \cdots \zeta(s+b_{\ell})$. 
A natural generalization would be to consider sequences $a(n),b(n)$ which are the  Dirichlet series coefficients of 
$L(s+a_1, \pi_1) \cdots L(s+a_k, \pi_k)$ and 
 $L(s+b_1, \pi_1') \cdots L(s+b_{\ell}, \pi_{\ell}')$ where the $L(s,\pi_i)$ and $L(s, \pi_j')$ are automorphic $L$-functions. 
It seems likely that the approach in this article
would lead to an asymptotic evaluation of  $ \mathscr{D}_{a,b;  \omega}(K)$, subject to a suitable additive divisor conjecture. 
In some cases, the additive divisor conjecture is known to be true. 
For instance, when $a(n)=b(n)=r_2(n)$ with $r_2(n)$ being the sum of two squares function, the conjecture holds (see \cite[page~174]{Iwaniec-spectral}).   


\subsection{Conventions and Notation} \label{convnot}
Given two functions $f(x)$ and $g(x)$, we shall interchangeably use the notation  $f(x)=O(g(x))$, $f(x) \ll g(x)$, and $g(x) \gg f(x)$  to mean there exists $M >0$ such that $|f(x)| \le M |g(x)|$ for all sufficiently large $x$. 
We write $f(x) \asymp g(x)$ to mean that the estimates $f(x) \ll g(x)$ and $g(x) \ll f(x)$ simultaneously hold.  
  
In this article we shall use the convention that $\varepsilon$ denotes an arbitrarily small positive constant which may vary from instance to instance.  
In addition, $B$ shall denote a positive constant, which may be taken arbitrarily large and which may change from line to line. 
The letter $p$ will always be used to denote a prime number.
For a function $\varphi : \mathbb{R}^{+} \times \mathbb{R}^{+} \to \mathbb{C}$, $\varphi^{(m,n)}(x,y) = \frac{\partial^m}{\partial x^m}
\frac{\partial^n}{\partial y^n} \varphi(x,y)$.  
The integral notation $\int_{(c)} f(s) ds$ for a complex function $f(s)$ and $c \in \mathbb{R}$
 will be used frequently and is defined by  the following contour integral
\begin{equation}
   \label{intc}
  \int_{(c)}f(s) ds = \int_{c-i\infty}^{c + i \infty} f(s) \, ds. 
\end{equation}
In this article we shall consider $s \in \mathbb{C}$ and  usually we shall write its real part as $\sigma =\Re(s)$. 
Throughout this article we often use the fact that $\omega(t)$ has support in
$[c_1 T, c_2T]$ so that $t \asymp T$. For a polynomial $P(X_1, \ldots, X_m) \in \mathbb{C}[X_1, \ldots, X_m]$ 
we write $\text{deg}(P)$ to denote its degree.   \\

\noindent {\it Acknowledgements}.  Thank-you to Dave Morris and Gabriel Verret for helping us with the proof of Lemma 
\ref{comb}. 

\section{Preliminaries}\label{section:weights}

In this section we describe the smooth weights 
$\varphi$ which occur in \eqref{ABphi} and 
$\omega$ which occur in \eqref{DabwK}.  
We should think of $\varphi$ as a smooth approximation to $\mathbb{1}_{[0,K]}$
and $\omega$ as a smooth approximation to  $\mathbb{1}_{[T,2T]}$.
The function $\varphi$ is used so that we can work with the  Mellin transform $\Phi(s)$ of $\varphi$ instead of Perron's formula
(see \cite{FK} and \cite{NT} for other such uses of the Mellin transform in  analytic number theory).  
The smoothness of $\varphi$ allows integrals in the $s$-variable to be absolutely convergent.  
Similarly the weight $\omega$ is used so that its Fourier transform $\widehat{\omega}(u)$ is small if $u$ satisfies 
$|u| \gg T_{0}^{1-\e}$, where $T_0$ is a parameter satisfying of $T^{b} \ll T_0 \ll T$ for some $b>0$. 
In certain situations it is standard to remove the weight $\omega$ from the integral and replace it with a sharp cutoff 
function (for instance, if the integrand is positive).   Previously, the use of functions such as $\omega$ can be found in the works
\cite{BCHB} and \cite{HB}.

\subsection{Properties of $\omega$} \label{ss-omega}
Let $b$ be a positive absolute constant, and let
$\omega$ be a function from $\mathbb{R}$ to $\mathbb{C}$  that satisfies the following:
\begin{align}
 & \label{cond1}  \omega \text{ is smooth}, \\
& \label{cond2} \text{the support of }  \omega  \text{ lies in } [c_1T,c_2T]  \text{ where } 0 < c_1 < c_2,\\
& \label{cond3} \text{there exists } T_0  \ge T^{b} 
\text{ such that } T_0 \ll T \text{ and }
  \omega^{(j)}(t) \ll T_{0}^{-j}.  
\end{align}
The Fourier transform of $\omega$ is 
\begin{equation}
  \label{ft}
  \widehat{\omega}(u) = \int_{\mathbb{R}} \omega(t) e^{-2 \pi i ut} dt. 
\end{equation}
Integrating by parts $j$ times and using the second part of  \eqref{cond3} we see that 
\begin{equation}\label{omegaftbound}
  \widehat{\omega}(u) = \frac{1}{(2 \pi i u)^j} \int_{\mathbb{R}} \omega^{(j)}(t) e^{-2 \pi i u t} dt \ll  \frac{T}{(T_0 u)^j}. 
\end{equation}
Thus,
\begin{equation}
  \label{whatbound}
\text{ if }  |u| \gg T_{0}^{-1+\e}, \ \text{ then }
 | \widehat{\omega}(u)|  \ll T^{-A} \text{ for any } A >0
\end{equation}
by choosing $j  \ge \frac{A+1}{b \e}$.

\subsection{Properties of $\varphi$} \label{ss-phi}
Let $\varrho  \in (0,\frac{1}{2})$.  Let $\varphi$ be a smooth, non-negative function defined on $\R_{\ge 0}$ such that it equals one on $[0,1]$ and zero on $[1+\varrho,\infty)$ and for all $j \ge 0$
\begin{equation}\label{phibound}
  \varphi^{(j)}(t) \ll \varrho^{-j}. 
\end{equation}
Its Mellin transform is
\begin{equation}
   \label{Phi}
  \Phi(s) =\int_{0}^{\infty} \varphi(t) t^{s-1} dt,
\end{equation}
which converges absolutely for $\Re(s) >0$.  
By Mellin inversion, 
\begin{equation}
  \label{melinv}
  \varphi(t) = \frac{1}{2 \pi i} \int_{(c)} \Phi(s) t^{-s} ds
\end{equation}
for $c >0$.  
We now study $\Phi(s)$ further.  
Integrating by parts,  
\begin{equation}
  \label{Psidefinition}
 \Phi(s) = \frac{1}{s} \Psi(s) \text{ where } \Psi(s) = -\int_{0}^{\infty}  \varphi'(t) t^s dt.
\end{equation}
This is valid  for $\Re(s) >0$.  It may be shown that $\Psi(s)$ is entire. 
Thus $\Phi(s)$ is holomorphic  on $\mathbb{C}$ with the exception of a simple pole at $s=0$.  
We have the Laurent expansion
\begin{equation}
  \label{Philaurexp}
   \Phi(s) = \frac{\Psi(0)}{s} + \Psi'(0) + \frac{\Psi''(0)}{2} s + \cdots. 
\end{equation}
Note that 
\[
    \Psi(0) = -\int_{0}^{\infty}  \varphi'(t)  dt = \varphi(0)-\varphi(1+\varrho) =1. 
\]
Next we provide useful bounds for $\Phi$.  
Integrating \eqref{Phi} by parts $m$ times, we find that 
\[
  \Phi(s) = \frac{(-1)^m}{s(s+1) \cdots (s+m-1)} \int_{0}^{\infty} \varphi^{(m)}(t) t^{s+m-1} dt,
\]
which is valid for all $s \in \mathbb{C} \setminus \{ 0 \}$. Note that for $m\ge 2$ the integrand has simple zeros at $s=-1, \ldots, -(m-1)$.    
Thus, for $m \ge 1$ and $s \in \mathbb{C} \setminus \{ 0, -1, \ldots, -(m-1) \}$, 
\begin{equation}
\begin{split}
  \label{Phibd}
  |\Phi(s)| & \le \frac{1}{|s(s+1) \cdots (s+m-1)|} 
   \int_{1}^{1+\varrho} |\varphi^{(m)}(t)| t^{\sigma+m-1} dt 
   \ll_{m}  \frac{\varrho^{1-m} (1+\varrho)^{\sigma+m-1} }{|s(s+1) \cdots (s+m-1)|}. 
\end{split}
\end{equation}
Observe that this implies 
\begin{equation}
   \label{Phibd2}
    |\Phi(s)| \ll_{\sigma,\varrho,m} |\Im(s)|^{-m} \text{ for } |\Im(s)| \ge 1.   
\end{equation}
Let $c >0$.  For  $\Re(s) >c$, we define
\begin{equation}
  \label{phi2}
  \Phi_2(s) = \frac{1}{2 \pi i} \int_{(c)} \Phi(s_1) \Phi(s-s_1) \, ds_1. 
\end{equation}
We shall encounter this function frequently.  Observe that by a version of the convolution formula (see \cite[eq. (3.1.14), p. 83]{PK})
\begin{equation}
   \label{phi2mtimt}
     \Phi_2(s) = \int_{0}^{\infty} \varphi(t)^2 t^{s-1} \, dt
     \text{ and }
      \varphi(t)^2 = \frac{1}{2 \pi i} \int_{(c)} \Phi_2(s) t^{-s} \, ds 
      \text{ for } c>0.
\end{equation}
Note that 
\begin{equation*}
     \Phi_2(s) -\Phi(s)
     = \int_{1}^{1+\rho} \varphi(t)(\varphi(t)-1)t^{s-1} \, dt.
\end{equation*}
By \cite[Ch. 5, p.108]{Co} it follows that the right hand side is an entire function. 
Thus we see that $\Phi_2(s)$ has a simple pole at $s=0$ and $\Phi_2(s) = \Phi(s) +R(s)$
 with $R(s)$ an entire function.  
\subsection{The Dirichlet series $Z_{\I,\J}(s)$}
In this article we shall encounter the Dirichlet series
\begin{equation}\label{ZIJs}
Z_{\I,\J}(s)=\sum_{m=1}^{\infty}\frac{\sigma_{\I}(m)\sigma_{\J}(m)}{m^{1+s}}.
\end{equation}
From \eqref{BIJ} and \eqref{BIJid},  we have the alternate expression
\begin{equation}\label{eqn:Z-sum-CK}
Z_{\I,\J}(s)=\mathcal{B}(\I_{s},\J)=\mathcal{Z}(\I_{s},\J) \mathcal{A}(\I_{s},\J),
\end{equation}
where the first equality follows from \eqref{sigmaidentity} and 
$\mathcal{Z}$ and $\mathcal{A}$ are defined in Definition \ref{ZA}.
For example, when  $\I=\{a_1,a_2\}$ and $\J=\{b_1,b_2\}$, 
it follows from a formula of Ramanujan
(see \cite[Eq~1.3.3]{Ti})
that
\[
   Z_{\I,\J}(s) = \frac{\zeta(1+s+a_1+b_1)\zeta(1+s+a_1+b_2)\zeta(1+s+a_2+b_1)\zeta(1+s+a_2+b_2)}{\zeta(2+2s+a_1+a_2+b_1+b_2)},
\]
and so
\begin{equation}
  \label{AIJs}
   \mathcal{A}(\I_s,\J) = \frac{1}{\zeta(2+2s+a_1+a_2+b_1+b_2)}.
\end{equation}The next lemma gives an analytic continuation of $Z_{\I,\J}(s)$ and demonstrates that it has simple poles at 
\begin{equation}
   \label{poles}
   s=-a_i-b_j \text{ for } i \in \K, j \in \Ll,
\end{equation}
as long as the elements are distinct.
\begin{lem}\label{lem:Z-sum}
Let $\delta \in (0,\frac{1}{2})$. 
Assume that $\I$ and $\J$ satisfy \eqref{sizerestrictiondelta}. 
For $\Re(s) > 2 \delta$, we have that 
\begin{equation}
  \label{ZIJfact}
Z_{\I,\J}(s)= \Big( \prod_{ \substack{ i \in \K \\ j \in \Ll} }\zeta(1+s+a_i+b_j) \Big) A_{\I,\J}(s),
\end{equation}
where 
$A_{\I,\J}(s)=\mathcal{A}(\I_s,\J)$ is holomorphic in $\Re(s) > - \frac{1}{2}+2\delta$.  
\end{lem}  
\begin{proof}
We let $\sigma = \Re(s)$.  
By multiplicativity we have 
\begin{equation*}
  Z_{\I,\J}(s)=\prod_{p} \sum_{u=0}^{\infty} \frac{\sigma_{\I}(p^u) \sigma_{\J}(p^u)}{p^{u(1+s)}}
  = \prod_{p} \Big(  1 + \sum_{ \substack{i \in \K \\ j \in \Ll} }
  p^{-1-s-a_i-b_j} +
   \sum_{u=2}^{\infty} \frac{\sigma_{\I}(p^u) \sigma_{\J}(p^u)}{p^{u(1+s)}}
\Big). 
\end{equation*}
We now factor out $\prod_{\substack{ i \in \K \\ j \in \Ll}  } \prod_{p}   (1 - p^{-1-s-a_i -b_j})^{-1}$ to obtain 
\begin{equation*}
   Z_{\I,\J}(s)=\prod_{\substack{i \in \K \\ j \in \Ll}}\zeta(1+s+a_i+b_j)A_{\I,\J}(s),
\end{equation*}
where $A_{\I,\J}(s) = \prod_{p}A_{p, \I,\J}(s)$ and 
\begin{equation*}
\begin{split}
    A_{p, \I,\J}(s) & =  \prod_{ \substack{i \in \K  \\ j \in \Ll}  }(1-p^{-1-s-a_i-b_j})\Big(  1 + \sum_{ \substack{i \in \K \\ j \in \Ll} }
  p^{-1-s-a_i-b_j} +
   \sum_{u=2}^{\infty} \frac{\sigma_{\I}(p^u) \sigma_{\J}(p^u)}{p^{u(1+s)}}
\Big).   
\end{split}
\end{equation*}
Expanding out and using \eqref{sizerestrictiondelta} the first factor is 
\begin{equation}
\begin{split}
  \label{factor1}
   \prod_{ i\in \K,  j \in \Ll  }(1-p^{-1-s-a_i-b_j})
   & = 1 -  \sum_{ i \in \K, j \in \Ll }
  p^{-1-s-a_i-b_j}  + O \Big( 
  \sum_{u=1}^{k \ell} p^{-2u -2u \sigma +4 j \delta}
  \Big) \\
  & = 1 -  \sum_{ i \in \K, j \in \Ll }
  p^{-1-s-a_i-b_j}  + O \Big( 
 p^{-2 -2 \sigma +4 \delta}
  \Big) 
\end{split}
\end{equation}
if $\sigma > 2\delta -1$. 
Using \eqref{sizerestrictiondelta} and the fact that 
$\sigma_{\I}(p^u) \ll_{k,\ell}  p^{u \delta}$, we find that 
\begin{equation*}
\begin{split}
   \sum_{u=2}^{\infty} \frac{\sigma_{\I}(p^u) \sigma_{\J}(p^u)}{p^{u(1+s)}} 
  \ll  \sum_{u=2}^{\infty} \frac{p^{2u\delta}}{ p^{u(1+\sigma)}}
  \ll   p^{-2-2\sigma+4\delta}.
\end{split}
\end{equation*}
It follows that 
\begin{equation*}
\begin{split}
    A_{p, \I,\J}(s) & = \Big(  1 -  \sum_{ i \in \K, j \in \Ll }
  p^{-1-s-a_i-b_j}  + O \Big( 
 p^{-2 -2 \sigma +4 \delta}
  \Big) \Big) 
   \Big(  1 + \sum_{ \substack{i \in \K \\ j \in \Ll} }
  p^{-1-s-a_i-b_j} + O   \Big(  p^{-2-2\sigma+4\delta} \Big) \Big)
   \\ 
    & = 1 + O_{k,\ell} \Big(  p^{-2-2\sigma+4\delta} \Big).
\end{split}
\end{equation*}
Hence, $A_{\I,\J}(s) = \prod_{p}A_{p, \I,\J}(s)$ converges absolutely if $\sigma>-\frac12+2\delta$.

\end{proof}

\section{Setting up the evaluation of  $\mathscr{D}_{\I,\J;\omega}(K)$}

Let us begin our evaluation of $\mathscr{D}_{\I,\J;\omega}(K)$.  
By splitting into diagonal terms $m=n$ and off-diagonal terms $m \ne n$ we have 
\begin{equation}
\begin{split}
  \label{decomposition}
 \mathscr{D}_{\I,\J;\omega}(K)  
  & =\sum_{m,n=1}^{\infty} \frac{\sigma_{\I}(m) \sigma_{\J}(n) \varphi(\frac{m}{K}) \varphi(\frac{n}{K})}{\sqrt{mn}}
  \widehat{\omega}(\tfrac{1}{2 \pi}\log(\tfrac{m}{n}))
  = \mathscr{D}_{\I,\J;\omega}^{\text{diag}}(K) + \mathscr{D}_{\I, \J;\omega}^{\text{off}}(K)
\end{split}
\end{equation}
where 
\begin{equation}
 \begin{split}
     \label{DIJdoff}
      \mathscr{D}_{\I,\J;\omega}^{\text{diag}}(K)
     = \widehat{\omega}(0) \sum_{m=1}^{\infty} \frac{\sigma_{\I}(m) \sigma_{\J}(m) \varphi^{2}(\frac{m}{K})}{m},  \, 
 \mathscr{D}_{\I, \J;\omega}^{\text{off}}(K)    = \sum_{m \ne n}  \frac{\sigma_{\I}(m) \sigma_{\J}(n) \varphi(\frac{m}{K}) \varphi(\frac{n}{K})}{\sqrt{mn}}
  \widehat{\omega}(\tfrac{1}{2 \pi}\log(\tfrac{m}{n})).
\end{split}
\end{equation}
First we examine the contribution of $\mathscr{D}_{\I, \J;\omega}^{\text{diag}}(K)$. Using \eqref{phi2} and \eqref{phi2mtimt}, we have for $c >0$
\begin{equation*}
\begin{split}
\mathscr{D}_{\I, \J;\omega}^{\text{diag}}(K)
&=\widehat{\omega}(0) \sum_{m=1}^{\infty} \frac{\sigma_{\I}(m) \sigma_{\J}(m) \varphi^{2}(\frac{m}{K})}{m}
 = \frac{\widehat{\omega}(0) }{2 \pi i}
 \sum_{m=1}^{\infty}  \frac{\sigma_{\I}(m) \sigma_{\J}(m)}{m}
 \int_{(2c)}   \Phi_2(s)  \Big( \frac{m}{K} \Big)^{-s} \, ds.
\\
\end{split}
\end{equation*}
Swapping summation and integration order we obtain 
\begin{equation}
\begin{split}
  \label{diagonal}
\mathscr{D}_{\I, \J;\omega}^{\text{diag}}(K) 
&=\frac{\widehat{\omega}(0)}{2\pi i}\int_{(2c)}K^{s}\Phi_2(s)
Z_{\I,\J}(s)
\;ds =   \mathcal{M}_{0, \I, \J; \omega}(K) ,
\end{split}
\end{equation}
where $Z_{\I,\J}(s)$ is defined in \eqref{ZIJs} and $ \mathcal{M}_{0, \I, \J; \omega}(K)$ is defined in \eqref{M0swaps}. Observe that these two expressions are equal since by \eqref{eqn:Z-sum-CK} we have $Z_{\I,\J}(s)=\mathcal{B}(\I_s,\J)$.

The key part of the calculation of $ \mathscr{D}_{\I,\J;\omega}(K)$ is that of the off-diagonal term
$\mathscr{D}_{\I, \J;\omega}^{\text{off}}(K)$.  We establish the following result. 

\begin{prop}
  \label{offdiagonal}
Let $|\I|=k$ and $|\J|=\ell$ with $k,\ell \ge 2$, and suppose that $\I$ and $\J$ satisfy \eqref{sizerestriction} and \eqref{sizerestriction-}.  Assume that $\mathcal{AD}_{k,\ell}(\vartheta_{k,\ell},C_{k,\ell},\beta_{k,\ell})$ holds for some triple $(\vartheta_{k,\ell}, C_{k,\ell},\beta_{k,\ell}) \in [\frac{1}{2}, 1) \times [0,\infty) \times (0,1]$. Let $K=T^{1+\eta}$  with $\eta>0$, and let $\omega$ satisfy \eqref{cond1}, \eqref{cond2}, and \eqref{cond3} with 
$b>\frac{(1-\beta_{k,\ell})(1+\eta)}{1-\epsilon}$ 
and $0<\epsilon<1$. Then we have for any $\varepsilon >0$,   
\begin{equation*}
   \mathscr{D}_{\I, \J;\omega}^{\text{off}}(K)
  =    \mathcal{M}_{1, \I, \J; \omega}(K) + O \Big( K^{\vartheta_{k,\ell}+\e} \Big( \frac{T}{T_0} \Big)^{1+C_{k,\ell}} \Big),
\end{equation*}
where $ \mathcal{M}_1$ is given in \eqref{M1swaps}. 
\end{prop}

This result builds on previous work of Hughes-Young \cite{HY} on the twisted fourth moment of the Riemann zeta function
and work of the second author \cite{Ng} on the sixth moment of the Riemann zeta function.  Those results are directly related to the 
cases $(k,\ell) \in \{ (2,2), (3,3) \}$ of this theorem.

Combining \eqref{decomposition}, \eqref{diagonal}, and Proposition \ref{offdiagonal} yields Theorem \ref{mainthm}. The proof of Proposition \ref{offdiagonal} will be given in Sections \ref{sec:off-diag} and \ref{completeproof}.

\section{Off-Diagonal Terms}\label{sec:off-diag}
In this section we evaluate the off-diagonal terms $ \mathscr{D}_{\I, \J;\omega}^{\text{off}}(K)$.
In order to evaluate this we must impose some initial size conditions on $\I$ and $\J$.  We require that the coefficients satisfy
\eqref{sizerestriction-}. 
These conditions are imposed since some error terms involve factors of the form
$ \zeta(1-a_{i_1}+a_{i_2})$ and $\zeta(1-b_{j_1}+b_{j_2})$ which are unbounded unless
\eqref{sizerestriction-} is imposed.  These conditions can be removed via use of an argument
with Cauchy's integral formula (see the argument in \cite[pp.20-21]{Ng}). 
Recall that 
\begin{equation}\label{eqn:off-diagonal}
\mathscr{D}_{\I,\J;\omega}^{\text{off}}(K)
= \sum_{m \ne n}  \frac{\sigma_{\I}(m) \sigma_{\J}(n) \varphi(\frac{m}{K}) \varphi(\frac{n}{K})}{\sqrt{mn}}
  \widehat{\omega}(\tfrac{1}{2 \pi}\log(\tfrac{m}{n})).
\end{equation}
\subsection{Smooth partition of unity.}
First, we apply a dyadic partition of unity to the sums over $m$ and $n$. 
To do this, we consider a smooth non-negative function $W_{0}$ on $(0,\infty)$ whose support lies in $[1,2]$, and which satisfies $\sum_{k \in \mathbb{Z}} W_0 ( x/2^{\frac{k}{2}} ) =1$, for all $x >0$. This implies that 
\begin{equation}
  \label{W0dyadic}
   \sum_{ \substack{ M= 2^{\frac{k}{2}} \\ k \ge -1} } W_0 \Big( \frac{x}{M} \Big) =1 \text{ for } x \ge 1. 
\end{equation}
An example of such a function is given in \cite[Section 5]{H}. 
Given two integers $m,n\geq1$, we have 
\begin{equation}\label{eqn:partition-of-unity}
\sum_{M}W_{0}\left(\frac{m}{M}\right)\sum_{N}W_{0}\left(\frac{n}{N}\right)=1,
\end{equation}
where $M,N \in \{ 2^{\frac{k}{2}} \ | \ k \in  \mathbb{Z} \text{ and } k \ge -1 \}$.  We shall often use the fact that $ 
\# \{ M \ | \ M \le X \} \ll \log X$. 
Upon inserting the identity (\ref{eqn:partition-of-unity}) in (\ref{eqn:off-diagonal}), we get
\begin{align*}
\mathscr{D}_{\I,\J;\omega}^{\text{off}}(K)&= \sum_{m \ne n} \sum_{M,N} \frac{\sigma_{\I}(m) \sigma_{\J}(n) W_{0}\left(\frac{m}{M}\right)W_{0}\left(\frac{n}{N}\right)\varphi(\frac{m}{K}) \varphi(\frac{n}{K})}{\sqrt{mn}}
  \widehat{\omega}(\tfrac{1}{2 \pi}\log(\tfrac{m}{n}))\\&= \sum_{M,N \le K'}\frac{1}{\sqrt{MN}}\sum_{\substack{m \ne n\\M\leq m\leq2M\\N\leq n\leq2N}}  \sigma_{\I}(m) \sigma_{\J}(n) W\left(\frac{m}{M}\right)W\left(\frac{n}{N}\right)
  \varphi \Big(\frac{m}{K} \Big)\varphi \Big(\frac{n}{K} \Big)
  \widehat{\omega}(\tfrac{1}{2 \pi}\log(\tfrac{m}{n})),
\end{align*}
where $K'=(1+\varrho)K$ and $W(x)=x^{-\frac12}W_{0}\left(x\right)$. By setting
\begin{equation}\label{IMN}
I_{M,N}=\sum_{\substack{m \ne n\\M\leq m\leq2M\\N\leq n\leq2N}}  \sigma_{\I}(m) \sigma_{\J}(n) W \Big(\frac{m}{M} \Big)W \Big(\frac{n}{N} \Big)\varphi \Big(\frac{m}{K} \Big)\varphi \Big(\frac{n}{K} \Big)
  \frac{\widehat{\omega}(\tfrac{1}{2 \pi}\log(\tfrac{m}{n}))}{T},
\end{equation} 
we can write
\begin{equation}
  \label{DIJoffdecomposition}
\mathscr{D}_{\I,\J;\omega}^{\text{off}}(K)
=\sum_{M,N\leq K'}\frac{T}{\sqrt{MN}}I_{M,N}.
\end{equation}
\subsection{Restricting $M$ and $N$.}

First, we observe that if  $M<\frac{N}{3}$ or $M>3N$, then the variables $m$ and $n$ satisfy $\left|\log\left(\tfrac{m}{n}\right)\right|>\log\left(\tfrac32\right)$. 
In this case, we know from \eqref{whatbound} that  for any  $A >0$, we have 
  $\widehat{\omega}(\tfrac{1}{2 \pi}\log(\tfrac{m}{n})) \ll T^{-A}$.  Hence, we see that $I_{M,N}$ can be made very small.
 Throughout the rest of this article we write  $M\asymp N$ to mean
 $\frac{N}{3} \le M \le 3N$.
 Furthermore, we may restrict the sum in \eqref{IMN} to integers $m,n$ such that 
\begin{equation}\label{condmn}
|\log(\tfrac{m}{n})|\ll T_{0}^{-1+\epsilon}.
\end{equation}

In fact, by \eqref{whatbound} we know that if 
 $|\log(\tfrac{m}{n})|\gg T_{0}^{-1+\epsilon}$, then  for any  $A >0$, we have 
  $\widehat{\omega}(\tfrac{1}{2 \pi}\log(\tfrac{m}{n})) \ll T^{-A}$. Therefore, the integers $m,n$ satisfying $|\log(\tfrac{m}{n})|\gg T_{0}^{-1+\epsilon}$
lead to an error term of the form $O(T^{O(1)-A})$ for any $A >0$. 
Without loss of generality, we may assume that $m > n$ and $m=n+r$ with $r \ge 1$. Thus, condition \eqref{condmn} 
becomes $|\log(1+\tfrac{r}{n})|\ll T_{0}^{-1+\epsilon}$, and so $|r| \ll n T_{0}^{-1+\epsilon} \ll M T_{0}^{-1+\epsilon}$
since $N \asymp M$.   
Hence, 
\[I_{M,N}=\sum_{1 \le |r| \ll M T_{0}^{-1+\epsilon}}\sum_{\substack{m,n\geq1\\m-n=r}} \sigma_{\I}(m) \sigma_{\J}(n)
W \Big(\frac{m}{M} \Big)W \Big(\frac{n}{N} \Big)\varphi \Big(\frac{m}{K} \Big)\varphi \Big(\frac{n}{K} \Big)
  \frac{\widehat{\omega}(\tfrac{1}{2 \pi}\log(\tfrac{m}{n}))}{T}+O(T^{-A}).\]
  Finally, observe that the condition on $r$ can be replaced by $0<|r|\ll (MN)^{\frac{1}{2}} T_{0}^{-1+\epsilon}$ since $M \asymp N$. 
 To summarize, we state the following proposition.

\begin{prop} 
 \label{IMNapprox}
If 
$\frac{N}{3}\leq M\leq 3N$,  then  for any $A>0$ we have
\[I_{M,N}=\sum_{0<|r|\ll\frac{\sqrt{MN}}{T_{0}^{1-\epsilon}}}\sum_{\substack{m,n\geq1\\m-n=r}} \sigma_{\I}(m) \sigma_{\J}(n) f^{*}(m,n)+O(T^{-A}),\] where  \begin{equation}
\label{fdefn}
f^{*}(m,n)=W\left(\frac{m}{M}\right)W\left(\frac{n}{N}\right)\varphi\left(\frac{m}{K}\right)\varphi\left(\frac{n}{K}\right)
  \frac{\widehat{\omega}(\tfrac{1}{2 \pi}\log(\tfrac{m}{n}))}{T}.\end{equation} If $M<\frac{N}{3}$ or $M>3N$, then for any $B>0$ we have $I_{M,N}=O(T^{-B})$ .
\end{prop}
Observe that if $m-n=r$, we have
\begin{equation}
  \label{f*special}
    f^{*}(m,n)=W\left(\frac{m}{M}\right)W\left(\frac{n}{N}\right)\varphi\left(\frac{m}{K}\right)\varphi\left(\frac{n}{K}\right)
  \frac{\widehat{\omega}(\tfrac{1}{2 \pi}\log(1+\tfrac{r}{n}))}{T}.
\end{equation}
\subsection{Applying the Additive Divisor Conjecture}

By combining \eqref{DIJoffdecomposition}, Proposition \ref{IMNapprox}, 
and \eqref{f*special}, we obtain
 \begin{equation}\label{eqn:apply-ADC}\mathscr{D}_{\I,\J;\omega}^{\text{off}}(K)=\sum_{\substack{M,N\leq K'\\M\asymp N}}\frac{T}{\sqrt{MN}}\sum_{0<|r|\ll\frac{\sqrt{MN}}{T_{0}^{1-\epsilon}}}D_{f_r; \I,\J}(r)+O(T^{-A}),
 \end{equation}
 where for $0 \ne r \in \mathbb{Z}$ we define
 \begin{equation}
    \label{fr}
     f_r(x,y) := f_{r,M,N}(x,y)
     =W\left(\frac{x}{M}\right)W\left(\frac{y}{N}\right)\varphi\left(\frac{x}{K}\right)\varphi\left(\frac{y}{K}\right)
  \frac{\widehat{\omega}(\tfrac{1}{2 \pi}\log(1+ \tfrac{r}{y}))}{T}
 \end{equation}
 when $x,y > 0$ (and put $f_r(x,y) = 0$ otherwise) and 
 $D_{f_r; \I,\J}(r)$ is the additive divisor sum associated to $\eqref{fr}$, $\I$, $\J$, and $r$ as defined in \eqref{DfIJr}.  Observe that in  
 \eqref{eqn:apply-ADC}, for each fixed $0 \ne |r| \ll \frac{\sqrt{MN}}{T_{0}^{1-\epsilon}}$, the inner sum contains the additive 
 divisor sum $D_{f_r; \I,\J}(r)$ associated to distinct functions $f_r$.
 
 In order to apply the additive divisor conjecture, Conjecture \ref{conj:shifted-sum}, we use the following lemma which shows that  the smooth function $f_r$ and its partial derivatives of all orders satisfy certain bounds. The proof of this  technical lemma is given in Section \ref{sec:technical-lemmas}.
\begin{lem} \label{fderivativebd}
 Let  $M \asymp N$  and $(x,y) \in [M,2M] \times [N,2N]$ with $1 \le |r| \ll \frac{\sqrt{MN}}{T_0} T^{\e}$.  
Then 
\begin{equation*}
  x^m y^n f_r^{(m,n)}(x,y) \ll P^{n},\; \text{ where }
     P = \Big( \frac{T}{T_0} \Big) T^{\e}. 
\end{equation*}
\end{lem}

We must also ensure that the length of the sum over $r$ in \eqref{eqn:apply-ADC}  is within the range $H\ll M^{\beta_{k,\ell}}$ as required by Conjecture \ref{conj:shifted-sum}. We have 
\[\frac{\sqrt{MN}}{T_{0}^{1-\epsilon}}\ll MT^{-b(1-\epsilon)}\leq M^{1-b\frac{1-\epsilon}{1+\eta}}.\] If we assume that $b>\frac{(1-\beta_{k,\ell})(1+\eta)}{1-\epsilon}$ for $0<\epsilon<1$, we get that $\frac{\sqrt{MN}}{T_{0}^{1-\epsilon}}\ll M^{\beta_{k,\ell}}$ as desired. We are now ready to apply the additive divisor conjecture  to compute the main terms in $\mathscr{D}_{\I,\J;\omega}^{\text{off}}(K)$. To simplify notation, 
it is convenient to 
 set 
 \begin{equation} 
   \label{ci1i2} 
 c_{i_{1},i_{2}}  = \prod_{j_1 \in \K \setminus \{ i_1 \}}  \zeta(1-a_{i_1}+a_{j_1})   
      \prod_{j_2 \in  \Ll \setminus \{ i_2 \} }
    \zeta(1-b_{i_2}+b_{j_2}), 
\end{equation} 
where $\K$ and $\Ll$ are as  in \eqref{KL}.
We also set
 \begin{equation}
  \label{Ndefn}
 N_{f;i_1,i_2}(r)=c_{i_{1},i_{2}}\sum_{q=1}^{\infty} \frac{c_q(r)G_{\I}(1-a_{i_1},q)G_{\J}(1-b_{i_2},q)  }{q^{2-a_{i_1}-b_{i_2}}} \int_{\max(0,r)}^{\infty} f(x,x-r) x^{-a_{i_1}} (x-r)^{-b_{i_2}}\; dx, 
 \end{equation}
and
 \begin{equation}
     \label{uqri1i2}
 u_{q,r;i_1,i_2}  =  c_{q}(r)G_{\I}(1-a_{i_1},q)G_{\J}(1-b_{i_2},q)  q^{-2+a_{i_1}+b_{i_2}}.
 \end{equation}
 
It follows from the definition of $D_{f;\I,\J}(r)$ in \eqref{DfIJr}, Proposition \eqref{IMNapprox}, and an application of Conjecture~\ref{conj:shifted-sum} with $X=M$, $Y=N$, $M\asymp N$ and $P=  (\frac{T}{T_0}) T^{\e}$ that  \begin{equation}
  \label{IMNdecomposition}
    I_{M,N}=\sum_{i_{1}=1}^{k}\sum_{i_{2}=1}^{\ell}\sum_{0<|r|\ll\frac{\sqrt{MN}}{T_{0}^{1-\epsilon}}}
N_{f;i_1,i_2}(r)+\mathcal{E}_{M,N}+O\left(T^{-A}\right), \end{equation}
where
\begin{equation}
  \label{EMN}
   \mathcal{E}_{M,N}  =
    \sum_{0 < |r| \ll \frac{\sqrt{MN}}{T_0} T^{\varepsilon}}  \Delta_{f;\I,\J}(r)=O\left(\left(\frac{T}{T_0}\right)^{C_{k,\ell}}\frac{T^{\varepsilon}}{T_0}M^{\vartheta_{k,\ell}+1+\varepsilon}\right).
\end{equation} 
We now estimate the contribution of the error terms $\mathcal{E}_{M,N}$ when substituted in \eqref{DIJoffdecomposition}.
\begin{lem}  \label{EMNsum}
Let $M \asymp N$. We have 
\begin{equation*}
 \sum_{\substack{M,N \le K' \\ M \asymp N}}  \frac{T}{\sqrt{MN}}  \mathcal{E}_{M,N}  \ll T^{\varepsilon} \Big( \frac{T}{T_0} \Big)^{1+C_{k,\ell}} K^{\vartheta_{k,\ell}}.
\end{equation*}
\end{lem}
\begin{proof}
We have
\begin{align*}
\frac{T}{M} \mathcal{E}_{M,N} &  
   \ll 
   T^{\varepsilon} \Big( \frac{T}{T_0} \Big)^{1+C_{k,\ell}} M^{\vartheta_{k,\ell}+\varepsilon}. 
 \end{align*}
It follows that 
 \begin{align*}
    \sum_{\substack{M,N \le K' \\ M \asymp N}}  \frac{T}{\sqrt{MN}}  \mathcal{E}_{M,N} 
 &   \ll  T^{\varepsilon} \Big( \frac{T}{T_0} \Big)^{1+C_{k,\ell}}  
   \sum_{M \le K'} \sum_{N \asymp M} M^{\vartheta_{k,\ell}+\varepsilon}  
      \ll K^{\vartheta_{k,\ell}+\varepsilon} \Big( \frac{T}{T_0} \Big)^{1+C_{k,\ell}}.
 \end{align*}
\end{proof}
Hence, by \eqref{DIJoffdecomposition}, \eqref{IMNdecomposition},  and Lemma \ref{EMNsum}  we have 
\begin{equation}
  \label{Doff}
\mathscr{D}_{\I,\J;\omega}^{\text{off}}(K)=\sum_{\substack{M,N\le K'\\M\asymp N}}\frac{T}{\sqrt{MN}}\sum_{i_{1}=1}^{k}\sum_{i_{2}=1}^{\ell}\sum_{0<|r|\ll\frac{\sqrt{MN}}{T_{0}^{1-\epsilon}}}
N_{f;i_1,i_2}(r)+O\left(K^{\vartheta_{k,\ell}} T^{\varepsilon} \Big( \frac{T}{T_0} \Big)^{1+C_{k,\ell}}\right).
\end{equation}
The next step is to extend the summation over $r$ to $|r| \le R_0$, where we set $R_0 = T^{A}$ for some large enough fixed $A$. 
This will be useful later when we extend the summation over $r$ to all integers.
 
Recalling the definition of $f$ in \eqref{fdefn}, the integral in \eqref{Ndefn} is
\begin{equation*}
 \mathfrak{i}= \int_{\max(0,r)}^{\infty}
 W\left(\frac{x}{M}\right)W\left(\frac{x-r}{N}\right)\varphi\left(\frac{x}{K}\right)\varphi\left(\frac{x-r}{K}\right)
  \frac{\widehat{\omega}(\tfrac{1}{2 \pi}\log(\tfrac{x}{x-r}))}{T}
  x^{-a_{i_1}} (x-r)^{-b_{i_2}}\; dx.
\end{equation*}
First, observe that the integrand is zero unless $x \in [M,2M]$ and $x \in [N+r,2N+r]$, since otherwise $W\left(\frac{x}{M}\right)W\left(\frac{x-r}{N}\right)=0$. 
Suppose without loss of generality that $r \ge 1$.  In order for these conditions to hold, the intervals $[M,2M]$ and $[N+r,2N+r]$ must intersect.   In particular, 
we must have $N+r \le 2M$ and thus $r \le 2M-N \le 2M-\frac{M}{3}=\frac{5M}{3}$ since $\frac{M}{3} \le N \le 3N$.

Notice that $x-r \ge N \ge 2^{-\frac{1}{2}}$, and so $x \ge r+2^{-\frac{1}{2}}$.  
By \eqref{omegaftbound} we have 
\begin{equation}\label{eqn:an-omega-hat-bound}
\begin{split}
  \frac{\widehat{\omega}(\tfrac{1}{2 \pi}\log(\tfrac{x}{x-r}))}{T} 
  & \ll \frac{1}{ |\log(\frac{x}{x-r})  T_0|^j} \le \frac{1}{ |\log(\frac{2M}{2M-r})  T_0|^j} \text{ since } x \in [r+2^{-\frac{1}{2}},2M]  \\
  & =   \frac{1}{ |\log(1-\frac{r}{2M})  T_0|^j}   <   \frac{1}{  |\frac{r}{2M} T_0|^j},
\end{split}
\end{equation}
where for the last inequality we used $|\log\left(1-\frac{r}{2M}\right)| > \frac{r}{2M}$ since
  $0 < \frac{r}{2M} \leq\frac{5}{6}$. If $ \frac{|r|}{M}  \gg  \frac{T^{\varepsilon}}{T_0} $, then \eqref{fdefn} and \eqref{eqn:an-omega-hat-bound} give
  \begin{equation}
\begin{split}
  \label{fstarxr}
  f(x,x-r) 
  & \ll T^{-j \varepsilon}  \ll T^{-B} \text{ for any } B>0
\end{split}
\end{equation} by choosing $j$ sufficiently large.  We note that \eqref{fstarxr} still holds for $\frac{r}{2M}>\frac{5}{6}$ since then $f(x,x-r)=0$. It follows that $ \mathfrak{i} \ll M T^{-B}$ 
and that for $|r| \le R_0$ all of these terms contribute

\begin{equation*}
 \mathcal{ET} :=  \frac{T}{\sqrt{MN}} \sum_{ \frac{\sqrt{MN}}{T_0} T^{\varepsilon} \le    |r| \le R_0 }
   \sum_{i_1=1}^{k} \sum_{i_2=1}^{\ell}  |c_{i_1,i_2}|
 \sum_{q=1}^{\infty}  |u_{q,r;i_1,i_2}|   MT^{-B} ,
\end{equation*}
where $c_{i_1,i_2}$ and $u_{q,r;i_1,i_2}$ are defined in \eqref{ci1i2} and \eqref{uqri1i2} respectively. 
Observe that by \eqref{sizerestriction-} it follows that 
\begin{equation}
   \label{ci1i2bound}
 c_{i_1,i_2} = O( (\log T)^{(k-1)(\ell-1)}) 
\end{equation}
 and 
\begin{equation}
\begin{split}
  \label{ulrbound}
  \sum_{q \ge 1} |u_{q,r;i_1,i_2}|  & \ll \sum_{q \ge 1} \gcd(q,r) q^{-2+4 \delta}
  = \sum_{g \mid r} g^{-1+4 \delta}  \sum_{\substack{ q' \ge 1 \\ \gcd(q',r/g)=1 }} (q')^{-2 +4 \delta} 
  \ll \tau_2(r),
\end{split}
\end{equation}
since $|c_{q}(r)| \le \gcd(q,r)$ and
\begin{equation}
 \label{GIJqbd}
|G_{\I}(1-a_{i_1},q)G_{\J}(1-b_{i_2},q)|  \ll
C^{\omega(q)} q^{2 \delta}
\end{equation}
for some $C>0$. This follows from the
the multiplicativity of $G_{\I}(s,n)$ and the bound \eqref{GIpabd} established in Lemma \ref{GIk} below.
Hence it follows that 
\begin{equation}
  \label{Erbound}
  \mathcal{ET} \ll   T^{1-B}\left(\log T\right)^{(k-1)(\ell-1)} \Big(\sum_{1 \le |r| \le R_0} \tau_2(r) \Big) \ll R_0 \left(\log T\right)^{(k-1)(\ell-1)+1} T^{1-B} \ll T^{-A}, 
\end{equation}
by choosing $B=1+2A+\varepsilon$. This shows that we can extend the summation over $r$ in \eqref{Doff}  to $|r|\leq R_0$ with a negligible error as desired.

We now extend the sum in \eqref{Doff} to all $M,N$.  
If  $M \not \asymp N$, then we have $N > 3M$ or $N < M/3$.  Suppose without loss of generality that $N > 3M$. 
Since $M \le x \le 2M$ and $N \le x-r \le 2N$,  it follows that $
  |\log ( \frac{x}{x-r} )| = \log  ( \frac{x-r}{x}  ) \ge \log(\frac{N}{2M}) \ge \log(\frac{3}{2})$.
Therefore,  $ \mathfrak{i}
 \ll \max(2M,2N+r)  T^{-B}$ for any $B>0$. It follows that we can add in the condition
 $M \not \asymp N$ into \eqref{Doff} with an error of size $O(T^{-A})$, by arguing as we did in the previous paragraph.  Finally, we note that the terms with $M>K'$ or $N>K'$ vanish since $f(x,x-r)=0$ then. Hence, we have shown that we can extend the summations in \eqref{Doff}  to $|r|\leq R_0$ and all $M,N\in\mathbb{N}$ with negligible error, thus proving the following proposition.
\begin{prop}. \label{firststepoff}
\begin{equation}\label{eqn:off-diag-main}\mathscr{D}_{\I,\J;\omega}^{\text{off}}(K)=\sum_{\substack{M,N}}\frac{T}{\sqrt{MN}}\sum_{i_{1}=1}^{k}\sum_{i_{2}=1}^{\ell}\sum_{1 \le |r| \le R_0}
N_{f;i_1,i_2}(r)+O\left(K^{\vartheta_{k,\ell}} T^{\varepsilon} \Big( \frac{T}{T_0} \Big)^{1+C_{k,\ell}} 
\right),
\end{equation} 
where by \eqref{fdefn} and \eqref{Ndefn} we have 
\begin{equation*}
\begin{split}
   N_{f;i_1,i_2}(r)
 & =c_{i_{1},i_{2}}\sum_{q=1}^{\infty} \frac{c_q(r)G_{\I}(1-a_{i_1},q)G_{\J}(1-b_{i_2},q)  }{q^{2-a_{i_1}-b_{i_2}}}\\
&\hspace{1em}\times \int_{\max(0,r)}^{\infty} x^{-a_{i_1}} (x-r)^{-b_{i_2}}W\left(\tfrac{x}{M}\right)W\left(\tfrac{x-r}{N}\right)\varphi\left(\tfrac{x}{K}\right)\varphi\left(\tfrac{x-r}{K}\right) \frac{\widehat{\omega}(\tfrac{1}{2 \pi}\log(\tfrac{x}{x-r}))}{T} \;dx. 
\end{split}
\end{equation*}
\end{prop}
Our goal is to evaluate $\mathscr{D}_{\I,\J;\omega}^{\text{off}}(K)$ asymptotically. More precisely, we want to prove that
\begin{equation}\label{eqn:prop3.1}
    \sum_{\substack{M,N}}\frac{T}{\sqrt{MN}}\sum_{i_{1}=1}^{k}\sum_{i_{2}=1}^{\ell}\sum_{1 \le |r| \le R_0}
N_{f;i_1,i_2}(r)
 =   \mathcal{M}_{1,\I, \J; \omega}(K)   + O(T^{8 \delta} \log T), 
\end{equation}
where $ \mathcal{M}_{1,\I, \J; \omega}(K)$ is the double integral given by \eqref{M1swaps}. 
Towards proving \eqref{eqn:prop3.1}, we shall establish the following proposition in the remaining part of this section.

\begin{prop} \label{secondstepoff}
\begin{equation*}
\begin{split}
 &    \sum_{\substack{M,N}}\frac{T}{\sqrt{MN}}\sum_{i_{1}=1}^{k}\sum_{i_{2}=1}^{\ell}\sum_{1 \le |r| \le R_0}
N_{f;i_1,i_2}(r)
 =  \sum_{i_{1}=1}^{k}\sum_{i_{2}=1}^{\ell}  \frac{c_{i_{1},i_{2}}}{(2\pi i)^2} \\
  & 
 \int_{-\infty}^{\infty}\omega(t)\int_{(1)} \int_{(1)}\Phi(s_{1}) \Phi(s_{2})K^{s_1+s_2}
   H_{\I,\J; \{ a_{i_1} \}, \{ b_{i_2} \}}(s_{1}+s_{2})\nonumber\\&\hspace{1em}\times 
 \Gamma(a_{i_1}+b_{i_2}+s_{1}+s_{2})   \Big( 
   \frac{\Gamma(\frac12-b_{i_2}-s_{2}+it)}{\Gamma(\frac12+a_{i_1}+s_{1}+it)}
   +\frac{\Gamma(\frac12-a_{i_1}-s_{1}-it)}{\Gamma(\frac12+b_{i_2}+s_{2}-it)} \Big)
  \;ds_{1}  \;ds_{2}\;dt +O(1),
\end{split}
\end{equation*}
where 
\begin{equation}\label{H}
H_{\I,\J; \{ a_{i_1} \} , \{ b_{i_2} \}}(s)=\sum_{r=1}^{\infty}\sum_{q=1}^{\infty} \frac{c_q(r)G_{\I}(1-a_{i_1},q)G_{\J}(1-b_{i_2},q)  }{q^{2-a_{i_1}-b_{i_2}}r^{a_{i_1}+b_{i_2}+s}}.
\end{equation}
\end{prop}

There are quite a few steps involved in proving this proposition which we now outline:
\begin{enumerate}
\item  The sum over $M,N$ is executed and the smooth partition of unity \eqref{W0dyadic} is applied again.  This removes the functions $W_0$.
\item Each occurrence of the  function $\varphi(u)$ is replaced by the inverse Mellin transform \[\frac{1}{2 \pi i} \int_{(c)} \Phi(s) u^{-s} \, du.\] 
This creates a double integral in the variables $s_1$, $s_2$. 

\item  The sum over $r$ is extended to all integers which introduce the Dirichlet series $H_{\I,\J; \{ a_{i_1} \} , \{ b_{i_2} \}}(s)$. 
\end{enumerate}

In establishing Proposition \ref{secondstepoff} we require a number of technical results related to Stirling's formula. 

We often use the following weak form of Stirling's formula
    \begin{equation*}
  |\Gamma(\s+iu)| =\sqrt{2 \pi} |u|^{\s-1/2}e^{-\frac{1}{2} \pi  |u|} (1+O(u^{-1}))
  \text{ for } 0<\sigma\ll 1 \text{ and } |u| \ge 1, 
\end{equation*}
in addition to the trivial bound \begin{equation*}
     |\Gamma(\sigma+iu)|   \ll_{\eta_0, A} 1  \text{ for }   |u| \le 1, \sigma\in[-1+\eta_0,-\eta_0]\cup[\eta_0,A],
\end{equation*}
where $\eta_0\in (0,\frac12)$ and $A>0$ are fixed constants. 
It follows that for $\Re(s_1)=\Re(s_2)=1$, we have
   \begin{equation}\label{eqn:weak-stirling}
 \Gamma(a_{i_1}+b_{i_2}+s_{1}+s_{2})\ll\begin{cases} |\Im(s_1+s_2)|^{\frac32+2\delta}\exp\left(-\frac{\pi}{2}|\Im(s_1+s_2)|\right)&\text{if }\;|\Im(s_1+s_2)|\geq1\\
1&\text{if}\; |\Im(s_1+s_2)|\leq1. \end{cases}
   \end{equation}

We also need the following lemma, the proof of which is given in Section \ref{sec:technical-lemmas}.
\begin{lem} \label{Stirling} Suppose that $a_{i_1}$ and $b_{i_2}$ satisfy \eqref{sizerestrictiondelta}. Let $\eta_0 \in (0,\tfrac{1}{2})$ and $A>0$. Assume that  $\Re(a_{i_1}+s_1) \in [0,A]$ and $\Re(b_{i_2}+s_2) \in [0, \tfrac{1}{2}-\eta_0]\cup[\tfrac12+\eta_0,\tfrac32-\eta_0]$.  \\
(i)  Assume that $\Re(s_1+s_2+a_{i_1}+b_{i_2})\leq1$. When $|\Im(s_1)|\leq t+1$ and $|\Im(s_2)|\leq t+1$, we have
 \begin{eqnarray}
 \label{ratio-gamma-0}
 \frac{\Gamma(\frac12-b_{i_2}-s_{2}\pm it)}{\Gamma(\frac12+a_{i_1}+s_{1}\pm it)}&=&t^{-(s_{1}+s_{2}+a_{i_1}+b_{i_2})}\exp\left(\mp i\frac{\pi}{2}(s_{1}+s_{2}+a_{i_1}+b_{i_2})\right)\\&&\hspace{1em}\times\left(1+O\left(\frac{1+|s_{1}|^2+|s_{2}|^2}{t}\right)\right).\nonumber
\end{eqnarray}
(ii) When $|\Im(s_1)|\geq t+1$ or $|\Im{s_2}|\geq t+1$, we have the  bound
\begin{equation}\label{ratio-gamma-1}
\frac{\Gamma(\frac12-b_{i_2}-s_{2}+ it)}{\Gamma(\frac12+a_{i_1}+s_{1}+it)}\ll \frac{\Im(s_1)^2+\Im(s_2)^2}{t^2}e^{\frac{\pi}{2}|\Im(s_1+s_2)|}.
\end{equation}
\end{lem}
We remark that if  $\Re(s_1+s_2+a_{i_1}+b_{i_2})\geq1$, $\Re(a_{i_1}+s_1) \in [0,A]$ and $\Re(b_{i_2}+s_2) \in [0, \tfrac{1}{2}-\eta_0]\cup[\tfrac12+\eta_0,\tfrac32-\eta_0]$, then the proof of Lemma \ref{Stirling} (i) yields the asymptotic bound
\begin{equation}\label{eqn:weak-lemma4.6-i}
\frac{\Gamma(\frac12-b_{i_2}-s_{2}\pm it)}{\Gamma(\frac12+a_{i_1}+s_{1}\pm it)}\ll \frac{\Im(s_1)^2+\Im(s_2)^2 }{t^{2}}e^{\pm\frac{\pi}{2}\left(\Im(s_1+s_2)\right)}, \end{equation}
for $|\Im(s_1)|\leq t+1$ and $|\Im(s_2)|\leq t+1$.

Moreover, by applying \eqref{ratio-gamma-0} twice we get the asymptotic formula:
\begin{align}\label{sum-of-gamma-quotients} 
& \frac{\Gamma(\frac12-b_{i_2}-s_{2}+it)}{\Gamma(\frac12+a_{i_1}+s_{1}+it)}
   +\frac{\Gamma(\frac12-a_{i_1}-s_{1}-it)}{\Gamma(\frac12+b_{i_2}+s_{2}-it)}\nonumber\\&\hspace{2em}= 2t^{-a_{i_1}-b_{i_2}-s_1-s_2}\cos\left(\frac{\pi}{2}(a_{i_1}+b_{i_2}+s_1+s_2)\right)+O\left(e^{\frac{\pi}{2}|\Im\left(s_1+s_2\right)|}\left(\frac{1+|s_{1}|^2+|s_{2}|^2}{t^{1-2\delta+\Re(s_1+s_2)}}\right)\right)
   \end{align}
  for $|\Im(s_1)|\leq t+1$ and $|\Im(s_2)|\leq t+1$ under the assumptions $\Re(a_{i_1}+s_1), \Re(b_{i_2}+s_2) \in (0, \tfrac{1}{2}-\eta_0]\cup[\tfrac12+\eta_0,1]$, 
   and $\Re(s_1+s_2+a_{i_1}+b_{i_2})\leq1$.

\begin{proof}[Proof of Prop \ref{secondstepoff}]

Recall that 
\begin{align*}
N_{f;i_1,i_2}(r)&=c_{i_{1},i_{2}}\sum_{q=1}^{\infty}u_{q,r;i_1,i_2}   \int_{\max(0,r)}^{\infty} f(x,x-r) x^{-a_{i_1}} (x-r)^{-b_{i_2}}\; dx, 
\end{align*}
 where 
 \begin{equation*}
 u_{q,r;i_1,i_2}  =  c_{q}(r)G_{\I}(1-a_{i_1},q)G_{\J}(1-b_{i_2},q)  q^{-2+a_{i_1}+b_{i_2}}.
 \end{equation*}
 and \begin{equation*}f(x,x-r)=W\left(\frac{x}{M}\right)W\left(\frac{x-r}{N}\right)\varphi\left(\frac{x}{K}\right)\varphi\left(\frac{x-r}{K}\right)
  \frac{\widehat{\omega}(\tfrac{1}{2 \pi}\log(\tfrac{x}{x-r}))}{T}.\end{equation*}
  
Using the fact that $W(x)=x^{-\frac12}W_0(x)$, we see that $N_{f;i_1,i_2}(r)$ can be written as 
\begin{align*}
N_{f;i_1,i_2}(r)&= \frac{\sqrt{MN}}{T} c_{i_{1},i_{2}} \sum_{q=1}^{\infty}u_{q,r;i_1,i_2}\int_{\max(0,r)}^{\infty}W_0\left(\frac{x}{M}\right)W_0\left(\frac{x-r}{N}\right)x^{-a_{i_1}-\frac12} (x-r)^{-b_{i_2}-\frac12}\\&\hspace{13em}\times \varphi\left(\frac{x}{K}\right)\varphi\left(\frac{x-r}{K}\right)
  \widehat{\omega}(\tfrac{1}{2 \pi}\log(\tfrac{x}{x-r})) \; dx.
\end{align*}

Hence,
\begin{align}\label{MNsum1}
&  \sum_{\substack{M,N}}\frac{T}{\sqrt{MN}}\sum_{i_{1}=1}^{k}\sum_{i_{2}=1}^{\ell}\sum_{1 \le |r| \le R_0}
N_{f;i_1,i_2}(r) \nonumber\\
& =\sum_{i_{1}=1}^{k}\sum_{i_{2}=1}^{\ell}c_{i_{1},i_{2}} \sum_{1 \le |r| \le R_0} \sum_{q=1}^{\infty}u_{q,r;i_1,i_2}\int_{\max(0,r)}^{\infty} \left(\sum_{\substack{M,N}}W_0\left(\frac{x}{M}\right)W_0\left(\frac{x-r}{N}\right)\right)x^{-a_{i_1}-\frac12} (x-r)^{-b_{i_2}-\frac12}\nonumber\\&\hspace{17em}\times \varphi\left(\frac{x}{K}\right)\varphi\left(\frac{x-r}{K}\right)
  \widehat{\omega}(\tfrac{1}{2 \pi}\log(\tfrac{x}{x-r})) \; dx.
\end{align}
Using \eqref{W0dyadic} we observe that, 
for $0 \ne r \in \mathbb{Z}$, we have 
\[
  \sum_{M,N}  W_0 \Big( \frac{x}{M} \Big) W_0 \Big(  \frac{x-r}{N} \Big)  =
  \begin{cases}
  1 & \text{ if } x >  \max(0,r)+1, \\
    W_0 \Big( \frac{x -\max(0,r)}{2^{-\frac{1}{2}}} \Big) & \text{ if }   \max(0,r) +2^{-\frac{1}{2}}< x \le  \max(0,r)+1, \\
    0 & \text{ if }   \max(0,r) \le  x \le  \max(0,r)+2^{-\frac{1}{2}}.
  \end{cases}
\]
Substituting this expression, we see that we may replace $\sum_{M,N}  W_0 ( \frac{x}{M} ) W_0 (  \frac{x-r}{N} )$
by $1$ with an error  of size $O(1)$.   This is very similar to an analogous calculation in \cite[pp.28-29]{Ng}.
Therefore, 
\begin{align}\label{MNsum1-MNremoved}
&  \sum_{\substack{M,N}}\frac{T}{\sqrt{MN}}\sum_{i_{1}=1}^{k}\sum_{i_{2}=1}^{\ell}\sum_{1 \le |r| \le R_0}
N_{f;i_1,i_2}(r) \nonumber\\
& =\sum_{i_{1}=1}^{k}\sum_{i_{2}=1}^{\ell}c_{i_{1},i_{2}} \sum_{1 \le |r| \le R_0} \sum_{q=1}^{\infty}c_{q}(r)G_{\I}(1-a_{i_1},q)G_{\J}(1-b_{i_2},q)  q^{-2+a_{i_1}+b_{i_2}}\nonumber\\&\hspace{9em}\times\int_{\max(0,r)}^{\infty}x^{-a_{i_1}-\frac12} (x-r)^{-b_{i_2}-\frac12}\varphi\left(\frac{x}{K}\right)\varphi\left(\frac{x-r}{K}\right)
  \widehat{\omega}(\tfrac{1}{2 \pi}\log(\tfrac{x}{x-r})) \; dx +O(1).
\end{align}
Applying \eqref{ft} and \eqref{melinv} we get
\begin{equation}
\begin{split}
\label{MNsum1-MNremovedB}
&  \sum_{\substack{M,N}}\frac{T}{\sqrt{MN}}\sum_{i_{1}=1}^{k}\sum_{i_{2}=1}^{\ell}\sum_{1 \le |r| \le R_0}
N_{f;i_1,i_2}(r) \\
& = \sum_{i_{1}=1}^{k}\sum_{i_{2}=1}^{\ell} c_{i_{1},i_{2}} \sum_{1 \le |r| \le R_0}
\sum_{q=1}^{\infty} \frac{c_q(r)G_{\I}(1-a_{i_1},q)G_{\J}(1-b_{i_2},q)  }{q^{2-a_{i_1}-b_{i_2}}}\\&\hspace{1em}\times \int_{\max(0,r)}^{\infty}x^{-a_{i_1}-\frac12} (x-r)^{-b_{i_2}-\frac12} 
\\&\hspace{1em}\times \frac{1}{(2 \pi i)^2}\int_{(c_{1})} \int_{(c_{2})}\Phi(s_{1}) \Phi(s_{2}) \left(\tfrac{K}{x}\right)^{s_{1}} (\tfrac{K}{x-r})^{s_{2}}\;ds_{1}\;ds_{2} \int_{-\infty}^{\infty}\omega(t)\left(1-\tfrac{r}{x}\right)^{it}\;dt\;dx +O(1)\\&=\sum_{i_{1}=1}^{k}\sum_{i_{2}=1}^{\ell}(I_{i_{1},i_{2}}^{+}+I_{i_{1},i_{2}}^{-})+O(1),
\end{split}
\end{equation}

where
\begin{align*}
I_{i_{1},i_{2}}^{+}&= c_{i_{1},i_{2}}\sum_{r=1}^{R_0} \sum_{q=1}^{\infty} \frac{c_q(r)G_{\I}(1-a_{i_1},q)G_{\J}(1-b_{i_2},q)  }{q^{2-a_{i_1}-b_{i_2}}} \int_{r}^{\infty}x^{-a_{i_1}-\frac12} (x-r)^{-b_{i_2}-\frac12}\\&\hspace{1em}\times\frac{1}{(2 \pi i)^2} \int_{(c_{1})} \int_{(c_{2})}\Phi(s_{1}) \Phi(s_{2}) \left(\tfrac{K}{x}\right)^{s_{1}} \Big(\tfrac{K}{x-r}\Big)^{s_{2}}\;ds_{1}\;ds_{2} \int_{-\infty}^{\infty}\omega(t)\left(1-\tfrac{r}{x}\right)^{it}\;dt\;dx
\end{align*}
and 
\begin{align*}
I_{i_{1},i_{2}}^{-}&= c_{i_{1},i_{2}}\sum_{r=-R_0}^{-1} \sum_{q=1}^{\infty} \frac{c_q(r)G_{\I}(1-a_{i_1},q)G_{\J}(1-b_{i_2},q)  }{q^{2-a_{i_1}-b_{i_2}}} \int_{0}^{\infty}x^{-a_{i_1}-\frac12} (x-r)^{-b_{i_2}-\frac12} \\&\hspace{1em}\times \frac{1}{(2 \pi i)^2}\int_{(c_{1})} \int_{(c_{2})}\Phi(s_{1}) \Phi(s_{2}) \left(\tfrac{K}{x}\right)^{s_{1}}  \Big(\tfrac{K}{x-r}\Big)^{s_{2}}\;ds_{1}\;ds_{2} \int_{-\infty}^{\infty}\omega(t)\left(1-\tfrac{r}{x}\right)^{it}\;dt\;dx\\&=c_{i_{1},i_{2}}\sum_{r=1}^{R_0}\sum_{q=1}^{\infty} \frac{c_q(r)G_{\I}(1-a_{i_1},q)G_{\J}(1-b_{i_2},q)  }{q^{2-a_{i_1}-b_{i_2}}} \int_{0}^{\infty}x^{-a_{i_1}-\frac12} (x+r)^{-b_{i_2}-\frac12} \\&\hspace{1em}\times \frac{1}{(2 \pi i)^2}\int_{(c_{1})} \int_{(c_{2})}\Phi(s_{1}) \Phi(s_{2}) (\tfrac{K}{x} )^{s_{1}} (\tfrac{K}{x+r})^{s_{2}}\;ds_{1}\;ds_{2} \int_{-\infty}^{\infty}\omega(t)\left(1+\tfrac{r}{x}\right)^{it}\;dt\;dx.
\end{align*}
Note that in view of \eqref{eqn:off-diag-main} we have
\begin{equation}
  \label{DIJwKmainterm}
  \mathscr{D}_{\I,\J;\omega}^{\text{off}}(K) =
\sum_{i_{1}=1}^{k}\sum_{i_{2}=1}^{\ell}(I_{i_{1},i_{2}}^{+}+I_{i_{1},i_{2}}^{-})
+O\left(K^{\vartheta_{k,\ell}} T^{\varepsilon} \Big( \frac{T}{T_0} \Big)^{1+C_{k,\ell}} 
+ 1\right).
\end{equation}

First, we carry out computations for $I_{i_{1},i_{2}}^{+}$. The change of variable $x\mapsto rx+r$ gives
\begin{align*}
I_{i_{1},i_{2}}^{+}&= c_{i_{1},i_{2}}\sum_{r=1}^{R_0}\sum_{q=1}^{\infty} \frac{c_q(r)G_{\I}(1-a_{i_1},q)G_{\J}(1-b_{i_2},q)  }{q^{2-a_{i_1}-b_{i_2}}r^{a_{i_1}+b_{i_2}}} \int_{0}^{\infty}(1+x)^{-a_{i_1}-\frac12} x^{-b_{i_2}-\frac12} \\&\hspace{1em}\times\frac{1}{(2 \pi i)^2} \int_{(c_{1})} \int_{(c_{2})}\Phi(s_{1}) \Phi(s_{2})\frac{K^{s_{1}+s_{2}}}{r^{s_{1}+s_{2}}(1+x)^{s_{1}}x^{s_{2}}} \;ds_{1}\;ds_{2} \int_{-\infty}^{\infty}\omega(t)x^{it}(x+1)^{-it}\;dt\;dx.
\end{align*}
We rearrange the orders of integration to get 
\begin{align}\label{eqn:I+}
I_{i_{1},i_{2}}^{+}&= c_{i_{1},i_{2}}\sum_{r=1}^{R_0}\sum_{q=1}^{\infty} \frac{c_q(r)G_{\I}(1-a_{i_1},q)G_{\J}(1-b_{i_2},q)  }{q^{2-a_{i_1}-b_{i_2}}r^{a_{i_1}+b_{i_2}}(2\pi i)^2}\int_{-\infty}^{\infty}\omega(t)\int_{(c_{1})} \int_{(c_{2})}\Phi(s_{1}) \Phi(s_{2})\frac{K^{s_{1}+s_{2}}}{r^{s_{1}+s_{2}}}\nonumber\\&\hspace{1em}\times \int_{0}^{\infty}(1+x)^{-\frac12-a_{i_1}-s_{1}-it} x^{-\frac12-b_{i_2}-s_{2}+it}\;dx\;ds_{1}\;ds_{2}\;dt.
\end{align}
We now simplify the inner integral with the following useful identities.  The first
gives an integral representation for the Beta function and the second relates it to Gamma functions:
\begin{equation}\label{eqn:beta}
\int_{0}^{\infty}\frac{x^{u-1}}{(1+x)^{v}}\;dx=B(u,v-u)=\frac{\Gamma(u)\Gamma(v-u)}{\Gamma(v)},\quad\quad\quad\Re(u),\Re(v)>0.
\end{equation}
Applying \eqref{eqn:beta} with $u =  \frac12-b_{i_2}-s_{2}+it $ and $v=\frac12+a_{i_1}+s_{1}+it$
to the innermost integral in \eqref{eqn:I+},
 we obtain 
\begin{align*}
I_{i_{1},i_{2}}^{+}&= c_{i_{1},i_{2}}\sum_{r=1}^{R_0}\sum_{q=1}^{\infty} \frac{c_q(r)G_{\I}(1-a_{i_1},q)G_{\J}(1-b_{i_2},q)  }{q^{2-a_{i_1}-b_{i_2}}r^{a_{i_1}+b_{i_2}}(2\pi i)^2}\int_{-\infty}^{\infty}\omega(t)\int_{(c_{1})} \int_{(c_{2})}\Phi(s_{1}) \Phi(s_{2})\frac{K^{s_{1}+s_{2}}}{r^{s_{1}+s_{2}}}\\&\hspace{1em}\times \frac{\Gamma(\frac12-b_{i_2}-s_{2}+it)\Gamma(a_{i_1}+b_{i_2}+s_{1}+s_{2})}{\Gamma(\frac12+a_{i_1}+s_{1}+it)}\;ds_{1}\;ds_{2}\;dt.
\end{align*}
The computations for $I_{i_{1},i_{2}}^{-}$ are similar. We make the change of variable $x\mapsto rx$ and interchange the orders of integration to get
\begin{align}\label{eqn:I-}
I_{i_{1},i_{2}}^{-}&=  c_{i_{1},i_{2}}\sum_{r=1}^{R_0}\sum_{q=1}^{\infty} \frac{c_q(r)G_{\I}(1-a_{i_1},q)G_{\J}(1-b_{i_2},q)  }{q^{2-a_{i_1}-b_{i_2}}r^{a_{i_1}+b_{i_2}}(2\pi i)^2}\int_{-\infty}^{\infty}\omega(t)\int_{(c_{1})} \int_{(c_{2})}\Phi(s_{1}) \Phi(s_{2})\frac{K^{s_{1}+s_{2}}}{r^{s_{1}+s_{2}}}\nonumber\\&\hspace{1em}\times \int_{0}^{\infty}(1+x)^{-\frac12-b_{i_2}-s_{2}+it} x^{-\frac12-a_{i_1}-s_{1}-it}\;dx\;ds_{1}\;ds_{2}\;dt.
\end{align}
Applying (\ref{eqn:beta}) in equation (\ref{eqn:I-}) gives
\begin{align*}
I_{i_{1},i_{2}}^{-}&= c_{i_{1},i_{2}}\sum_{r=1}^{R_0}\sum_{q=1}^{\infty} \frac{c_q(r)G_{\I}(1-a_{i_1},q)G_{\J}(1-b_{i_2},q)  }{q^{2-a_{i_1}-b_{i_2}}r^{a_{i_1}+b_{i_2}}(2\pi i)^2}\int_{-\infty}^{\infty}\omega(t)\int_{(c_{1})} \int_{(c_{2})}\Phi(s_{1}) \Phi(s_{2})\frac{K^{s_{1}+s_{2}}}{r^{s_{1}+s_{2}}}\\&\hspace{1em}\times \frac{\Gamma(\frac12-a_{i_1}-s_{1}-it)\Gamma(a_{i_1}+b_{i_2}+s_{1}+s_{2})}{\Gamma(\frac12+b_{i_2}+s_{2}-it)}\;ds_{1}\;ds_{2}\;dt.
\end{align*}
We combine $I_{i_{1},i_{2}}^{+}$ and $I_{i_{1},i_{2}}^{-}$ to obtain 
\begin{align}\label{Iplus+Iminus1}
  &  I_{i_{1},i_{2}}^{+}+I_{i_{1},i_{2}}^{-}
   = 
   c_{i_{1},i_{2}}\sum_{r=1}^{R_0}\sum_{q=1}^{\infty} \frac{c_q(r)G_{\I}(1-a_{i_1},q)G_{\J}(1-b_{i_2},q)  }{q^{2-a_{i_1}-b_{i_2}}r^{a_{i_1}+b_{i_2}}(2\pi i)^2}\int_{-\infty}^{\infty}\omega(t)\int_{(c_{1})} \int_{(c_{2})}\Phi(s_{1}) \Phi(s_{2})\frac{K^{s_{1}+s_{2}}}{r^{s_{1}+s_{2}}}\nonumber\\&\hspace{1em}\times 
 \Gamma(a_{i_1}+b_{i_2}+s_{1}+s_{2})   \Big( 
   \frac{\Gamma(\frac12-b_{i_2}-s_{2}+it)}{\Gamma(\frac12+a_{i_1}+s_{1}+it)}
   +\frac{\Gamma(\frac12-a_{i_1}-s_{1}-it)}{\Gamma(\frac12+b_{i_2}+s_{2}-it)} \Big)
  \;ds_{1}  \;ds_{2}\;dt.
\end{align}
We move the $s_2$ contour right to $\Re(s_2)=1$ and pass
a pole at  $s_{2}=P_{2}$, where  $P_2=\frac{1}{2}-b_{i_2}+it$. Therefore,
\[
 \frac{1}{2 \pi i} \int_{(c_2)} f(s_1,s_2) ds_2 =   \frac{1}{2 \pi i} \int_{(1)} f(s_1,s_2) ds_2
 - \text{Res}_{s_2 =P_2} \Big( f(s_1,s_2) \Big),
\]
where 
\[
  f(s_1,s_2) =  \Phi(s_{2})\frac{K^{s_{1}+s_{2}}}{r^{s_{1}+s_{2}}}   
 \Gamma(a_{i_1}+b_{i_2}+s_{1}+s_{2})   \Big( 
   \frac{\Gamma(\frac12-b_{i_2}-s_{2}+it)}{\Gamma(\frac12+a_{i_1}+s_{1}+it)}
   +\frac{\Gamma(\frac12-a_{i_1}-s_{1}-it)}{\Gamma(\frac12+b_{i_2}+s_{2}-it)} \Big).
\]
Since $\Gamma(z) \sim \frac{1}{z}$ as $z\to0$, it follows that 
\[
  \text{Res}_{s_2 =P_2} \Big( f(s_1,s_2) \Big) = -\Phi(P_2)   \Big( \frac{K}{r} \Big)^{s_1+ P_2}
   = -\Phi(\tfrac{1}{2}-b_{i_2}+  it)   \Big( \frac{K}{r} \Big)^{s_1+ P_2}
   \ll  T^{-B} \Big( \frac{K}{r} \Big)^{c_1+ \frac{1}{2}+\delta},
\] for any $B>0$. 
By \eqref{ci1i2bound}, \eqref{sizerestrictiondelta}, and \eqref{GIJqbd},  the residue's contribution to \eqref{Iplus+Iminus1} is
\begin{align*}
  (\log T)^{(k-1)(\ell-1)} 
   \sum_{r=1}^{R_0}\sum_{q=1}^{\infty} \frac{|c_q(r)|r^{2 \delta} C^{\omega(q)} }{q^{2-4 \delta}} \int_{-\infty}^{\infty}\omega(t)\int_{(c_{1})} |\Phi(s_{1})|  \Big( \frac{K}{r} \Big)^{c_1+ \frac{1}{2}+\delta}  T^{-B} |ds_1| dt 
   \ll T^{-B+O(1)}. 
\end{align*}

Thus, we have
\begin{align}\label{Iplus+Iminus2}
 &  I_{i_{1},i_{2}}^{+}+I_{i_{1},i_{2}}^{-}
   = 
   c_{i_{1},i_{2}}\sum_{r=1}^{R_0}\sum_{q=1}^{\infty} \frac{c_q(r)G_{\I}(1-a_{i_1},q)G_{\J}(1-b_{i_2},q)  }{q^{2-a_{i_1}-b_{i_2}}r^{a_{i_1}+b_{i_2}}(2\pi i)^2}\int_{-\infty}^{\infty}\omega(t)\int_{(c_{1})} \int_{(1)}\Phi(s_{1}) \Phi(s_{2})\frac{K^{s_{1}+s_{2}}}{r^{s_{1}+s_{2}}}\nonumber\\&\hspace{1em}\times 
 \Gamma(a_{i_1}+b_{i_2}+s_{1}+s_{2})   \Big( 
   \frac{\Gamma(\frac12-b_{i_2}-s_{2}+it)}{\Gamma(\frac12+a_{i_1}+s_{1}+it)}
   +\frac{\Gamma(\frac12-a_{i_1}-s_{1}-it)}{\Gamma(\frac12+b_{i_2}+s_{2}-it)} \Big)
  \;ds_{1}  \;ds_{2}\;dt +O(T^{-A}),
\end{align}
for any $A >0$. Next, we move the $s_1$ contour right to $\Re(s_1)=1$ and pass a pole at $s_1=P_1$, where $P_1=\frac12-a_{i_1}-it$. We have 
\begin{align*}
  & \text{Res}_{s_1=P_1} \Phi(s_{1})\frac{K^{s_{1}+s_{2}}}{r^{s_{1}+s_{2}}}   
 \Gamma(a_{i_1}+b_{i_2}+s_{1}+s_{2})   \Big( 
   \frac{\Gamma(\frac12-b_{i_2}-s_{2}+it)}{\Gamma(\frac12+a_{i_1}+s_{1}+it)}
   +\frac{\Gamma(\frac12-a_{i_1}-s_{1}-it)}{\Gamma(\frac12+b_{i_2}+s_{2}-it)} \Big)\\&= -\Phi(P_1)   \Big( \frac{K}{r} \Big)^{s_2+ P_1}
   = -\Phi(\tfrac{1}{2}-a_{i_1}-  it)   \Big( \frac{K}{r} \Big)^{s_2+ P_1}
   \ll  T^{-B} \Big( \frac{K}{r} \Big)^{\frac{3}{2}+\delta},
\end{align*} for any $B>0$. 
By \eqref{ci1i2bound}, \eqref{sizerestrictiondelta}, and \eqref{GIJqbd},
the residue's contribution to \eqref{Iplus+Iminus2} is
\begin{align*}
  (\log T)^{(k-1)(\ell-1)} 
   \sum_{r=1}^{R_0}\sum_{q=1}^{\infty} \frac{|c_q(r)|r^{2 \delta} C^{\omega(q)}}{q^{2-4 \delta}} \int_{-\infty}^{\infty}\omega(t)\int_{(1)} |\Phi(s_{2})|  \Big( \frac{K}{r} \Big)^{ \frac{3}{2}+\delta}  T^{-B} |ds_2| dt 
   \ll T^{-B+O(1)}. 
\end{align*}
Thus, we have
\begin{align}\label{eqn:Ii1i2}
 &  I_{i_{1},i_{2}}^{+}+I_{i_{1},i_{2}}^{-}
   = 
   c_{i_{1},i_{2}}\sum_{r=1}^{R_0}\sum_{q=1}^{\infty} \frac{c_q(r)G_{\I}(1-a_{i_1},q)G_{\J}(1-b_{i_2},q)  }{q^{2-a_{i_1}-b_{i_2}}r^{a_{i_1}+b_{i_2}}(2\pi i)^2}\int_{-\infty}^{\infty}\omega(t)\int_{(1)} \int_{(1)}\Phi(s_{1}) \Phi(s_{2})\frac{K^{s_{1}+s_{2}}}{r^{s_{1}+s_{2}}}\nonumber\\&\hspace{1em}\times 
 \Gamma(a_{i_1}+b_{i_2}+s_{1}+s_{2})   \Big( 
   \frac{\Gamma(\frac12-b_{i_2}-s_{2}+it)}{\Gamma(\frac12+a_{i_1}+s_{1}+it)}
   +\frac{\Gamma(\frac12-a_{i_1}-s_{1}-it)}{\Gamma(\frac12+b_{i_2}+s_{2}-it)} \Big)
  \;ds_{1}  \;ds_{2}\;dt +O(T^{-A}),
\end{align}
for any $A >0$.

By \eqref{ratio-gamma-1} and \eqref{eqn:weak-lemma4.6-i}, we have 
 \begin{equation}\label{eqn:ratio-gamma}
  \frac{\Gamma(\frac12-b_{i_2}-s_{2}+it)}{\Gamma(\frac12+a_{i_1}+s_{1}+it)}
   +\frac{\Gamma(\frac12-a_{i_1}-s_{1}-it)}{\Gamma(\frac12+b_{i_2}+s_{2}-it)}\ll t^{-2}\left(\Im(s_1)^{2}+\Im(s_2)^{2}\right)e^{\frac{\pi}{2}|\Im(s_1+s_2)|},
   \end{equation}
 for all $s_1$, $s_2$ with $\Re(s_1)=\Re(s_2)=1$.

   Using \eqref{Phibd}, \eqref{eqn:weak-stirling} and  \eqref{eqn:ratio-gamma} we see that the absolute value of the $s_1,s_2$ integrals in \eqref{eqn:Ii1i2} is 
   \begin{align*}
   &\ll \frac{K^2}{r^{2}t^{2}}\int_{(1)} \int_{(1)}\left|\Phi(s_{1}) \Phi(s_{2})\right| 
\left|\Gamma(a_{i_1}+b_{i_2}+s_{1}+s_{2})  \right|
  \left(\Im(s_1)^{2}+\Im(s_2)^{2}\right)e^{\frac{\pi}{2}|\Im(s_1+s_2)|}
  \;\left|ds_{1}\right|  \;\left|ds_{2}\right|\\&\ll \frac{K^2}{r^{2}t^{2}}.
   \end{align*}
   Therefore, \begin{align*}
   &c_{i_{1},i_{2}}\sum_{r=R_0}^{\infty}\sum_{q=1}^{\infty} \frac{c_q(r)G_{\I}(1-a_{i_1},q)G_{\J}(1-b_{i_2},q)  }{q^{2-a_{i_1}-b_{i_2}}r^{a_{i_1}+b_{i_2}}(2\pi i)^2}\int_{-\infty}^{\infty}\omega(t)\int_{(1)} \int_{(1)}\Phi(s_{1}) \Phi(s_{2})\frac{K^{s_{1}+s_{2}}}{r^{s_{1}+s_{2}}}\nonumber\\&\hspace{1em}\times 
 \Gamma(a_{i_1}+b_{i_2}+s_{1}+s_{2})   \Big( 
   \frac{\Gamma(\frac12-b_{i_2}-s_{2}+it)}{\Gamma(\frac12+a_{i_1}+s_{1}+it)}
   +\frac{\Gamma(\frac12-a_{i_1}-s_{1}-it)}{\Gamma(\frac12+b_{i_2}+s_{2}-it)} \Big)
  \;ds_{1}  \;ds_{2}\;dt \\&\ll K^2\left(\log T\right)^{(k-1)(\ell-1)}\sum_{r=R_0}^{\infty}\tau_2(r)r^{-2+2\delta}\int_{c_1T}^{c_2T}t^{-2}\;dt
  \ll K^2T^{-1}R_0^{-1+2\delta+\epsilon}\left(\log T\right)^{(k-1)(\ell-1)} \ll 1,
   \end{align*}
   where the last bound is obtained by choosing $R_0=T^A$ for a sufficiently large value of $A$ and assuming that $\delta$ is small enough. Therefore, at the cost of increasing the error term in \eqref{eqn:Ii1i2} to $O(1)$, we may extend the range of summation for $r$  to include all positive integers. By absolute convergence, we may swap the order of summation and integration to get
\begin{align}\label{eqn:I+I-}
 &  I_{i_{1},i_{2}}^{+}+I_{i_{1},i_{2}}^{-}
   = J_{i_1,i_2} +O(1),
\end{align}
where \begin{equation}
\label{Ji1i2}
\begin{split}
J_{i_1,i_2}&=\frac{c_{i_{1},i_{2}}}{(2\pi i)^2}\int_{-\infty}^{\infty}\omega(t)\int_{(1)} \int_{(1)}\Phi(s_{1}) \Phi(s_{2})K^{s_1+s_2}
   H_{\I,\J; \{ a_{i_1} \}, \{ b_{i_2} \}}(s_{1}+s_{2})\\&\hspace{1em}\times 
 \Gamma(a_{i_1}+b_{i_2}+s_{1}+s_{2})   \Big( 
   \frac{\Gamma(\frac12-b_{i_2}-s_{2}+it)}{\Gamma(\frac12+a_{i_1}+s_{1}+it)}
   +\frac{\Gamma(\frac12-a_{i_1}-s_{1}-it)}{\Gamma(\frac12+b_{i_2}+s_{2}-it)} \Big)
  \;ds_{1}  \;ds_{2}\;dt ,
  \end{split}
    \end{equation}
and $H_{\I,\J; \{ a_{i_1} \} , \{ b_{i_2} \}}(s)$ is the Dirichlet series defined in \eqref{H}.
\end{proof}

Proposition  \ref{secondstepoff} relates the off-diagonal terms to a certain double integral of the Dirichlet series $H_{\I,\J; \{ a_{i_1} \} , \{ b_{i_2} \}}(s)$. 
By Proposition \ref{Hid} below we are able to meromorphically extend this series to $\Re(s) > - \frac{1}{2}+2 \delta$. 
  Thus a main term arising from the off-diagonals can be obtained
by moving the contours to the left.  

\section{Completing the proof of Proposition \ref{offdiagonal}} \label{completeproof}

Combining Proposition \ref{firststepoff} and Proposition \ref{secondstepoff} gives
\begin{equation}
\label{prop4.4prop4.5}
\begin{split}
    \mathscr{D}_{\I,\J;\omega}^{\text{off}}(K) = \sum_{i_{1}=1}^{k}\sum_{i_{2}=1}^{\ell}J_{i_1,i_2} 
    +O\left(K^{\vartheta_{k,\ell}} T^{\varepsilon} \Big( \frac{T}{T_0} \Big)^{1+C_{k,\ell}} +1\right),
    \end{split}
\end{equation}
where \begin{align*}J_{i_1,i_2}&= \frac{c_{i_{1},i_{2}}}{(2\pi i)^2} 
 \int_{-\infty}^{\infty}\omega(t)\int_{(1)} \int_{(1)}\Phi(s_{1}) \Phi(s_{2})K^{s_1+s_2}
   H_{\I,\J; \{ a_{i_1} \}, \{ b_{i_2} \}}(s_{1}+s_{2})\nonumber\\&\hspace{1em}\times 
 \Gamma(a_{i_1}+b_{i_2}+s_{1}+s_{2})   \Big( 
   \frac{\Gamma(\frac12-b_{i_2}-s_{2}+it)}{\Gamma(\frac12+a_{i_1}+s_{1}+it)}
   +\frac{\Gamma(\frac12-a_{i_1}-s_{1}-it)}{\Gamma(\frac12+b_{i_2}+s_{2}-it)} \Big)
  \;ds_{1}  \;ds_{2}\;dt.\end{align*}

In order to complete the proof of Proposition \ref{offdiagonal} we require the following proposition. 
\begin{prop} \label{thirdstepoff}
We have
\begin{equation*}
\begin{split}
   \sum_{i_{1}=1}^{k}\sum_{i_{2}=1}^{\ell} J_{i_1,i_2}
  & =   \mathcal{M}_{1,\I, \J; \omega}(K)   + O(T^{8 \delta} \log T). 
\end{split}
\end{equation*}
\end{prop}
With this proposition in hand, we are able to establish Proposition  \ref{offdiagonal}.
\begin{proof}[Proof of Proposition  \ref{offdiagonal}]

From \eqref{prop4.4prop4.5} and Proposition  \ref{thirdstepoff}, we get

\[
    \mathscr{D}_{\I,\J;\omega}^{\text{off}}(K) =  \mathcal{M}_{1,\I, \J; \omega}(K)  
    +O\left(K^{\vartheta_{k,\ell}} T^{\varepsilon} \Big( \frac{T}{T_0} \Big)^{1+C_{k,\ell}}
    +T^{8\delta}\log T  \right).
\]
If we assume that $\delta$ is small enough (e.g. $8\delta<(1+\eta)\vartheta_{k,\ell}$),  we find that the second error term is smaller than the first, i.e. 
 \[T^{8\delta}\log T=O\left(K^{\vartheta_{k,\ell}} T^{\varepsilon} \Big( \frac{T}{T_0} \Big)^{1+C_{k,\ell}}\right)\] as $T\to\infty$. 
\end{proof}
The rest of this section is dedicated to establishing
Proposition \ref{thirdstepoff}.  Note that we  require the following proposition which gives a meromorphic continuation  of $H_{\I,\J; \{ a_{i_1} \}, \{ b_{i_2} \}}(s)$.
\begin{prop} \label{Hid}
Let $\delta \in (0,\frac{1}{12})$, $\e \in (0,\frac{1}{4})$,  and assume $|a_i|,|b_j| \le \delta$ for $i \in \K$, $j \in \Ll$.  \\
Fix $i_1 \in \K$ and $i_2 \in \Ll$.  
For $\Re(s) > 2 \delta$, we have
\begin{equation}
  \label{Hfactorization}
    H_{\I,\J; \{ a_{i_1} \} , \{ b_{i_2} \}}(s) =  \zeta(a_{i_1}+b_{i_2}+s) \prod_{\substack{k_1 \in \K \backslash \{ i_1 \}  
    \\ k_2 \in \Ll \backslash \{ i_2 \} }} \zeta(1+a_{k_1}+b_{k_2}+s)     \mathcal{C}_{\I,\J; \{ a_{i_1} \},\{ b_{i_2} \}}(s),
\end{equation}
where
\begin{equation}
  \label{CIJa1b1sdefn}
 \mathcal{C}_{\I,\J; \{ a_{i_1} \}, \{ b_{i_2} \}}(s) = \prod_{p} \mathcal{C}_{\I,\J; \{ a_{i_1} \}, \{ b_{i_2} \}}(p;s).
\end{equation}
For $\Re(s)=\sigma \ge -1+2\delta+\e$,  we have
\begin{equation}
  \label{CIJexpansion}
     \mathcal{C}_{\I,\J; \{ a_{i_1} \} , \{ b_{i_2} \}}(p;s) =1 - \Big(
      \sum_{ \substack{ \S \subset \I \backslash \{ a_{i_1} \}  \\ \mathcal{T} \subset \J \backslash \{ b_{i_2} \}  \\ |\mathcal{S}|=|\mathcal{T}|=2}}
  p^{-\mathcal{S}-\mathcal{T}} \Big) p^{-2-2s}     
     +O_{\varepsilon}(p^{8 \delta+\vartheta(\sigma)})
,\end{equation}
where 
\begin{equation}
   \label{thetasigma}
    \vartheta(\sigma) =
    \begin{cases}
    -2 & \text{ if } \sigma \ge 0; \\
    -2-\sigma & \text{ if } 0 > \sigma \ge -\frac{1}{2}; \\
    -3-3 \sigma & \text{ if } - \frac{1}{2} > \sigma
   , \end{cases}
\end{equation}
and we use the following notation:  if $\S \subset \I = \{ a_1, \ldots, a_k \}$ 
and $\S= \{ a_{i_1}, \ldots, a_{i_s} \}$ 
and  if $\mathcal{T} \subset \J = \{ b_1, \ldots, b_{\ell} \}$ 
and $\mathcal{T}= \{ b_{j_1}, \ldots, b_{j_t} \}$ 

then 
\[
   p^{-\S} := p^{-a_{i_1}-a_{i_2} - \cdots - a_{i_s}}
   \text{ and }
   p^{-\mathcal{T}} := p^{-b_{j_1} - b_{j_2} - \cdots -b_{j_t}}. 
\]
It follows that $\mathcal{C}_{\I,\J; \{ a_i \}, \{ b_j \}}(s)$ is absolutely convergent and holomorphic for $\Re(s) > -\frac{1}{2} + 2\delta$. Hence, $ H_{\I,\J; \{ a_i \} , \{ b_j \}}(s)$ has an analytic continuation to $\Re(s) > -\frac{1}{2}+2\delta$
with 
the exception of simple poles at 
\begin{equation}
  \label{Hpoles}
  1-(a_{i_1}+b_{i_2}) \text{ and }
  -(a_{k_1}+b_{k_2}) \text{ for } k_1 \ne i_1, k_2 \ne i_2.
\end{equation} 
\end{prop}
The proof of Proposition \ref{Hid} is given in Section \ref{sec:Hfunction}. We note that special cases of this result were previously proven in \cite[Corollary 6.2,~p.223]{HY}
for the case $|\I|=|\J|=2$ (by setting $h=k=1$ in their article)
and in \cite[Proposition 5.1]{Ng} for the case $|\I|=|\J|=3$.  Furthermore, our proof 
simplifies the proof of Proposition 5.1
in \cite{Ng}, as we do not  use Maple. We also require the following bounds on $H_{\I,\J; \{ a_{i_1} \} , \{ b_{i_2} \}}(s)$ which follow from  \eqref{Hfactorization} and the well known bounds for $\zeta(x+iy)$ when $0\leq x<1$. 

\begin{lem}
\label{lem:Hbounds}
Let $\delta\in(0,\frac{1}{7})$ and suppose that $|a_{i_1} |, |b_{i_2} |<\delta$. When $\Re(s)=4\delta$, we have 
\begin{equation} 
  \label{HIJdeltabd}
H_{\I,\J; \{ a_{i_1} \} , \{ b_{i_2} \}}(s)\ll_{\delta} (1+|\Im(s)|)^\frac{1}{2}.
\end{equation}
When $|\Re(s)-1|\leq 4\delta$ and $\Im(s)\asymp T$, we have 
\[H_{\I,\J; \{ a_{i_1} \} , \{ b_{i_2} \}}(s)\ll T^{4\delta}.\]
\end{lem}
Note that the implied constant in \eqref{HIJdeltabd} goes to infinity as 
$\delta \to 0$.
\begin{proof}
By  \eqref{Hfactorization} we have for $s=4\delta+iu$
\begin{align*}H_{\I,\J; \{ a_{i_1} \} , \{ b_{i_2} \}}(s)&\ll \left|\zeta(a_{i_1}+b_{i_2}+4\delta+iu)\right| \prod_{\substack{k_1 \in \K \backslash \{ i_1 \}  
    \\ k_2 \in \Ll \backslash \{ i_2 \} }} \left|\zeta(1+a_{k_1}+b_{k_2}+4\delta+iu) \right|\\&\ll  \left|\zeta(a_{i_1}+b_{i_2}+4\delta+iu)\right|,\end{align*}
    since $\Re(1+a_{k_1}+b_{k_2}+4\delta+iu)>1+2\delta$. By \cite[Theorem~1.9]{ivic}, we know that 
    \[\zeta(a_{i_1}+b_{i_2}+4\delta+iu)\ll (1+|u|)^{\frac{1-2\delta}{2}}\log(2+|u|),\]
    since $0<2\delta\leq \Re(a_{i_1}+b_{i_2})+4\delta\leq 6\delta<1$ and $u-2\delta\leq\Im(a_{i_1}+b_{i_2})+u\leq u+2\delta$. Hence,
    
    \[H_{\I,\J; \{ a_{i_1} \} , \{ b_{i_2} \}}(s)\ll(1+|u|)^{\frac{1}{2}}.\]
    Let us now assume that $1-4\delta\leq \Re(s) \leq1+4\delta$, then $\Re(1+a_{i_1}+b_{i_2}+s)\geq 2-6\delta>1$. It follows that $\prod_{\substack{k_1 \in \K \backslash \{ i_1 \}  
    \\ k_2 \in \Ll \backslash \{ i_2 \} }} \left|\zeta(1+a_{k_1}+b_{k_2}+s) \right|\ll1$. If we further assume that $\Im(s)\asymp T$,  then $\zeta(a_{i_1}+b_{i_2}+s)\ll T^{\frac{1-(1-6\delta)}{2}}\log(T)$. Hence, \[H_{\I,\J; \{ a_{i_1} \} , \{ b_{i_2} \}}(s)\ll T^{3\delta}\log T\ll T^{4\delta},\] when $|\Re(s)-1|\leq 4\delta$ and $\Im(s)\asymp T$.
\end{proof}

We also require an identity relating $\mathcal{C}_{\I,\J; \{ a_{i_1} \}, \{ b_{i_2} \}}(s)$ to $\mathcal{A}(\tilde{I},\tilde{J})$
for certain sets $\tilde{I}$ and $\tilde{J}$.  
\begin{prop} \label{CAprop}
Let $\delta >0$, and let
 $\I$ and $\J$ satisfy the size restrictions \eqref{sizerestrictiondelta}. Let $C>0$, and suppose that $|\Re(s)| \le C \delta$. 
If $\delta < \frac{1}{2(C+1)}$, then we have 
\begin{equation}
  \label{CAidentity}
\mathcal{C}_{\I,\J;\{ a_{i_1} \}, \{ b_{i_2} \}}(s)= \mathcal{A}( (\I \setminus \{ a_{i_1} \}) \cup \{ -b_{i_2}-s \},   
     ((\J \setminus \{ b_{i_2} \})+s)
      \cup \{ - a_{i_1} \} ). 
\end{equation}
\end{prop}
The proof of the identity \eqref{CAidentity} is given in \cite[Section~4]{CK3}.  Special cases of this result were proven in \cite[Lemmas 6.10-6.12]{HY} and \cite[Proposition 6.2]{Ng}
corresponding to $|\I|=|\J|=2$  and $|\I|=|\J|=3$ respectively.

\begin{proof}[Proof of Proposition \ref{thirdstepoff}]
By Proposition \ref{Hid},   $H_{\I,\J; \{ a_{i_1} \}, \{ b_{i_2} \}}(s)$ has a meromorphic continuation to 
   $\Re(s)>-\frac12+ 2\delta$
 with  simple poles at $1-a_{i_1}-b_{i_2}$ and $\{-a_{j_1}-b_{j_2}\}_{j_1\neq i_1,j_2\neq i_2}$. 
 Going back to \eqref{Ji1i2}, we move the line of integration in $s_1$ to $\Re(s_1)=\varepsilon_1$  with $\varepsilon_1=2\delta$. In doing so, we pass a pole of $\frac{\Gamma(\frac12-b_{i_2}-s_2+it)}{\Gamma(\frac12+a_{i_1}+s_{1}+it)}
   +\frac{\Gamma(\frac12-a_{i_1}-s_{1}-it)}{\Gamma(\frac12+b_{i_2}+s_2-it)} $ at $s_1=\frac12-a_{i_1}-it$ which contributes an error term of $o(1)$
   as $T \to \infty$. Next, we move the line of integration in $s_2$ to $\Re(s_2)=\varepsilon_1$,  passing a pole of $\frac{\Gamma(\frac12-b_{i_2}-s_2+it)}{\Gamma(\frac12+a_{i_1}+s_{1}+it)}
   +\frac{\Gamma(\frac12-a_{i_1}-s_{1}-it)}{\Gamma(\frac12+b_{i_2}+s_2-it)} $ at $s_2=\frac12-b_{i_2}+it$ which also contributes an error term of $o(1)$ as $T \to \infty$. Notice that the pole of $H_{\I,\J; \{ a_{i_1} \}, \{ b_{i_2} \}}(s_1+s_2)$ at $s_2=1-a_{i_1}-b_{i_2}-s_1$ cancels with a corresponding zero of $\frac{\Gamma(\frac12-b_{i_2}-s_2+it)}{\Gamma(\frac12+a_{i_1}+s_{1}+it)}
   +\frac{\Gamma(\frac12-a_{i_1}-s_{1}-it)}{\Gamma(\frac12+b_{i_2}+s_2-it)} $, and the other poles at $s_2=-a_{j_1}-b_{j_2}-s_1$ for $j_1\neq i_1,\;j_2\neq i_2$ are avoided by our choice of $\varepsilon_1$. 
   Now we have  
   \begin{equation}\label{eqn:s=epsilon_1}
   \begin{split}
  &  J_{i_1,i_2}
  =\frac{c_{i_{1},i_{2}}}{(2\pi i)^2}\int_{-\infty}^{\infty}\omega(t)\int_{(\varepsilon_1)} \int_{(\varepsilon_1)}\Phi(s_1) \Phi(s_2)K^{s_1+s_2}
   H_{\I,\J; \{ a_{i_1} \}, \{ b_{i_2} \}}(s_1+s_2)\nonumber\\&\hspace{3em}\times 
 \Gamma(a_{i_1}+b_{i_2}+s_1+s_2)   \Big( 
   \frac{\Gamma(\frac12-b_{i_2}-s_2+it)}{\Gamma(\frac12+a_{i_1}+s_{1}+it)}
   +\frac{\Gamma(\frac12-a_{i_1}-s_{1}-it)}{\Gamma(\frac12+b_{i_2}+s_2-it)} \Big)
  \;ds_1  \;ds_2\;dt +O(1).
  \end{split}
  \end{equation}
   
We consider the portions of the $s_1$, $s_2$ integrals with $|\Im(s_1)|\geq t+1$ or $|\Im(s_2)|\geq t+1$. We have 

\begin{align}\label{eqn:bd-Im(s)>t+1}
&\int_{\substack{\Re(s_2)=\varepsilon_1\\|\Im(s_2)|\geq t+1}} \int_{\substack{\Re(s_1)=\varepsilon_1\\|\Im(s_1)|\geq t+1}}\Phi(s_1) \Phi(s_2)K^{s_1+s_2}H_{\I,\J; \{ a_{i_1} \} , \{ b_{i_2} \}}(s_1+s_2)\nonumber\\&\hspace{1em}\times 
 \Gamma(a_{i_1}+b_{i_2}+s_1+s_2)   \Big( 
   \frac{\Gamma(\frac12-b_{i_2}-s_2+it)}{\Gamma(\frac12+a_{i_1}+s_{1}+it)}
   +\frac{\Gamma(\frac12-a_{i_1}-s_{1}-it)}{\Gamma(\frac12+b_{i_2}+s_2-it)} \Big)
  \;ds_1  \;ds_2\nonumber\\&\ll K^{2\varepsilon_1}\int_{\substack{\Re(s_2)=\varepsilon_1\\|\Im(s_2)|\geq t+1}} \int_{\substack{\Re(s_1)=\varepsilon_1\\|\Im(s_1)|\geq t+1}}\left|\Phi(s_1) \right|\left|\Phi(s_2)\right|
  |H_{\I,\J; \{ a_{i_1} \} , \{ b_{i_2} \}}(s_1+s_2)|\nonumber\\&\hspace{1em}\times 
\left| \Gamma(a_{i_1}+b_{i_2}+s_1+s_2)\right|   \Big| 
   \frac{\Gamma(\frac12-b_{i_2}-s_2+it)}{\Gamma(\frac12+a_{i_1}+s_{1}+it)}
   +\frac{\Gamma(\frac12-a_{i_1}-s_{1}-it)}{\Gamma(\frac12+b_{i_2}+s_2-it)} \Big|
  \;\left|ds_1\right|  \;\left|ds_2\right|
  \end{align}
  Using \eqref{Phibd2}, \eqref{eqn:weak-stirling}, Lemma \ref{Stirling} (ii) and Lemma \ref{lem:Hbounds}, we see that \eqref{eqn:bd-Im(s)>t+1} is  \begin{align}&\ll K^{2\varepsilon_1}t^{-2}\int_{\substack{\Re(s_2)=\varepsilon_1\\|\Im(s_2)|\geq t+1}} \int_{\substack{\Re(s_1)=\varepsilon_1\\|\Im(s_1)|\geq t+1}}\left|\Im(s_1) \right|^{-m}\left|\Im(s_2)\right|^{-m}|H_{\I,\J; \{ a_{i_1} \} , \{ b_{i_2} \}}(s_1+s_2)|\nonumber\\&\hspace{1em}\times 
\left| \Gamma(a_{i_1}+b_{i_2}+s_1+s_2)\right| (\Im(s_1)^2+\Im(s_1)^2)  \exp\left(\frac{\pi}{2}\left(|\Im(s_1+s_2)|\right)\right)
  \;\left|ds_1\right|  \;\left|ds_2\right|
  \\&\ll K^{2\varepsilon_1}t^{-2}.
\end{align}
   It follows that the contribution to \eqref{eqn:s=epsilon_1} arising from $\left|\Im (s_1)\right|, \left|\Im (s_2)\right| \ge t+1$  is 
   $\ll T^{2\varepsilon_1(1+\eta)}T^{-1}=O(1)$. 
   The contribution to \eqref{eqn:s=epsilon_1} arising from the other portions is likewise $O(1)$.
    Therefore, we have 
   \begin{align}\label{eqn:s=epsilon_1+it}
    J_{i_1,i_2}
   &= \frac{c_{i_{1},i_{2}}}{(2\pi i)^2}\int_{-\infty}^{\infty}\omega(t)\int_{\varepsilon_1-i(t+1)}^{\varepsilon_1+i(t+1)} \int_{\varepsilon_1-i(t+1)}^{\varepsilon_1+i(t+1)}\Phi(s_1) \Phi(s_2)K^{s_1+s_2}H_{\I,\J; \{ a_{i_1} \} , \{ b_{i_2} \}}(s_1+s_2)\nonumber\\&\hspace{3em}\times 
 \Gamma(a_{i_1}+b_{i_2}+s_1+s_2)   \Big( 
   \frac{\Gamma(\frac12-b_{i_2}-s_2+it)}{\Gamma(\frac12+a_{i_1}+s_{1}+it)}
   +\frac{\Gamma(\frac12-a_{i_1}-s_{1}-it)}{\Gamma(\frac12+b_{i_2}+s_2-it)} \Big)
  \;ds_1  \;ds_2\;dt +O(1)\nonumber\\&=\frac{c_{i_{1},i_{2}}}{(2\pi i)^2}\int_{-\infty}^{\infty}\omega(t)\int_{\varepsilon_1-i(t+1)}^{\varepsilon_1+i(t+1)} \int_{\varepsilon_1-i(t+1)}^{\varepsilon_1+i(t+1)}\Phi(s_1) \Phi(s_2)K^{s_1+s_2}H_{\I,\J; \{ a_{i_1} \} , \{ b_{i_2} \}}(s_1+s_2)\nonumber\\&\hspace{3em}\times
 \Gamma(a_{i_1}+b_{i_2}+s_1+s_2)  2t^{-a_{i_1}-b_{i_2}-s_1-s_2}\cos\left(\frac{\pi}{2}(a_{i_1}+b_{i_2}+s_1+s_2)\right)\;ds_1\;ds_2\;dt\nonumber\\&\hspace{1em}+\frac{c_{i_{1},i_{2}}}{(2\pi i)^2}\int_{-\infty}^{\infty}\omega(t)\int_{\varepsilon_1-i(t+1)}^{\varepsilon_1+i(t+1)} \int_{\varepsilon_1-i(t+1)}^{\varepsilon_1+i(t+1)}\Phi(s_1) \Phi(s_2)K^{s_1+s_2}H_{\I,\J; \{ a_{i_1} \} , \{ b_{i_2} \}}(s_1+s_2)\nonumber\\&\hspace{3em}\times
 \Gamma(a_{i_1}+b_{i_2}+s_1+s_2) O\left(\exp\left(\frac{\pi}{2}|\Im\left(s_1+s_2\right)|\right)\left(\frac{1+|s_{1}|^2+|s_{2}|^2}{t^{1-2\delta+\Re(s_1+s_2)}}\right)\right)
  \;ds_1  \;ds_2\;dt +O(1),\end{align}
  where the last equality follows from  \eqref{sum-of-gamma-quotients}. 
  
   We use  \eqref{Phibd}, \eqref{eqn:weak-stirling}  and Lemma \ref{lem:Hbounds} to get that the contribution of the second term to \eqref{eqn:s=epsilon_1+it}  is bounded by $T^{8\delta}\log T$.
 Therefore,
  \begin{align*}
   J_{i_{1},i_{2}}&=\frac{c_{i_{1},i_{2}}}{(2\pi i)^2}\int_{-\infty}^{\infty}\omega(t)\int_{\varepsilon_1-i(t+1)}^{\varepsilon_1+i(t+1)} \int_{\varepsilon_1-i(t+1)}^{\varepsilon_1+i(t+1)}\Phi(s_1) \Phi(s_2)K^{s_1+s_2}H_{\I,\J; \{ a_{i_1} \} , \{ b_{i_2} \}}(s_1+s_2)\nonumber\\
   &\times
 \Gamma(a_{i_1}+b_{i_2}+s_1+s_2)  2t^{-a_{i_1}-b_{i_2}-s_1-s_2}\cos\left(\frac{\pi}{2}(a_{i_1}+b_{i_2}+s_1+s_2)\right)\;ds_1\;ds_2\;dt+ O(T^{8\delta}\log T).
  \end{align*}
 Since 
 \[\Gamma(a_{i_1}+b_{i_2}+s_1+s_2)\cos\left(\frac{\pi}{2}(a_{i_1}+b_{i_2}+s_1+s_2)\right)\ll \left|\Im(s_1+s_2)\right|^{\Re(s_1+s_2)+2\delta-\frac12},\] the part of the integral where $\left|\Im(s_1)\right|\geq t+1$ or $\left|\Im(s_2)\right|\geq t+1$ is $O(1)$ and can be added to the error term. Hence, we may extend the bounds of integration in $s_1$ and $s_2$ to all of $\Re(s_1)=\varepsilon_1$ and $\Re(s_2)=\varepsilon_1$. This yields,
 
   \begin{align*}
   J_{i_{1},i_{2}}&=\frac{c_{i_{1},i_{2}}}{(2\pi i)^2}\int_{-\infty}^{\infty}\omega(t)\int_{(\varepsilon_1)}\int_{(\varepsilon_1)}\Phi(s_1) \Phi(s_2)K^{s_1+s_2}H_{\I,\J; \{ a_{i_1} \} , \{ b_{i_2} \}}(s_1+s_2)\Gamma(a_{i_1}+b_{i_2}+s_1+s_2)\nonumber \\
  &\hspace{1em} \times  2 t^{-a_{i_1}-b_{i_2}-s_1-s_2}\cos\left(\frac{\pi}{2}(a_{i_1}+b_{i_2}+s_1+s_2)\right)\;ds_1\;ds_2\;dt+ O(T^{8\delta}\log T).
 \end{align*}
 Observe that the inner double integral is of the shape $\int_{(\varepsilon_1)}\int_{(\varepsilon_1)} \Phi(s_1) \Phi(s_2)f(s_1+s_2) \, ds_2 \, ds_1$, where $f(z) = K^{z}H_{\I,\J; \{ a_{i_1} \} , \{ b_{i_2} \}}(z)
 \Gamma(a_{i_1}+b_{i_2}+z)  2t^{-a_{i_1}-b_{i_2}-z}\cos(\frac{\pi}{2}(a_{i_1}+b_{i_2}+z))$.  In the inner integral, we make the 
 variable change $s=s_1+s_2$ and then change order of integration to find
 \begin{align*}
  \frac{1}{2 \pi i} \int_{(\varepsilon_1)}\int_{(\varepsilon_1)} \Phi(s_1) \Phi(s_2)f(s_1+s_2) \, ds_2 \, ds_1
   & = \int_{(2\varepsilon_1)} \Phi_2(s) f(s) \, ds,
 \end{align*} 
 where we recall that  $\Phi_2$ is defined in \eqref{phi2}.  It follows that 
 \begin{align*}
  J_{i_{1},i_{2}}
&=\frac{c_{i_{1},i_{2}}}{2\pi i}\int_{-\infty}^{\infty}\omega(t)\int_{\Re(s)=2\varepsilon_1}\Phi_2(s) K^{s}H_{\I,\J; \{ a_{i_1} \} , \{ b_{i_2} \}}(s)
 \Gamma(a_{i_1}+b_{i_2}+s)\nonumber\\&\hspace{1em}\times  2t^{-a_{i_1}-b_{i_2}-s}\cos\left(\frac{\pi}{2}(a_{i_1}+b_{i_2}+s)\right)\;ds\;dt+ O(T^{8\delta}\log T).
  \end{align*}
By the functional equation $\zeta(1-z)=2^{1-z}\pi^{-z}\cos\left(\frac{\pi}{2}z\right)\Gamma(z)\zeta(z)$
and \eqref{Hfactorization}
it follows that 
 \begin{align*}
 &2H_{\I,\J; \{ a_{i_1} \} , \{ b_{i_2} \}}(s)
 \Gamma(a_{i_1}+b_{i_2}+s)\cos\left(\frac{\pi}{2}(a_{i_1}+b_{i_2}+s)\right)\\
 &=2\Big(\zeta(a_{i_1}+b_{i_2}+s)\prod_{\substack{a_{j_1}\in \I\backslash\{a_{i_1}\}\\b_{j_2}\in \J\backslash\{b_{i_2}\}}}\zeta(1+a_{j_1}+b_{j_2}+s)\Big) \mathcal{C}_{\I,\J,a_{i_1},b_{i_2}}(s)\\&\hspace{4em}\times \Gamma(a_{i_1}+b_{i_2}+s)\cos\left(\frac{\pi}{2}(a_{i_1}+b_{i_2}+s)\right)\\&=\prod_{\substack{a_{j_1}\in \I\backslash\{a_{i_1}\}\\b_{j_2}\in \J\backslash\{b_{i_2}\}}}\zeta(1+a_{j_1}+b_{j_2}+s)
 \mathcal{C}_{\I,\J,a_{i_1},b_{i_2}}(s)\zeta(1-a_{i_1}-b_{i_2}-s)(2\pi)^{a_{i_1}+b_{i_2}+s}.
 \end{align*}
Hence, 
\begin{align*}
   J_{i_{1},i_{2}}&=\frac{c_{i_{1},i_{2}}}{2\pi i}\int_{-\infty}^{\infty}\omega(t)\int_{\Re(s)=2\varepsilon_1}\Phi_2(s) K^{s} \left( \frac{2\pi}{t}\right)^{a_{i_1}+b_{i_2}+s}\prod_{\substack{a_{j_1}\in \I\backslash\{a_{i_1}\}\\b_{j_2}\in \J\backslash\{b_{i_2}\}}}\zeta(1+a_{j_1}+b_{j_2}+s)\\&\hspace{3em}\times \mathcal{C}_{\I,\J,a_{i_1},b_{i_2}}(s)\zeta(1-a_{i_1}-b_{i_2}-s)\;ds\;dt+ O(T^{8\delta}\log T).
\end{align*}
By an application of \eqref{CAidentity} we obtain 
\begin{align*}
    J_{i_{1},i_{2}}&=\frac{c_{i_{1},i_{2}}}{2\pi i}\int_{-\infty}^{\infty}\omega(t)\int_{\Re(s)=2\varepsilon_1}\Phi_2(s) K^{s} \left( \frac{2\pi}{t}\right)^{a_{i_1}+b_{i_2}+s}\zeta(1-a_{i_1}-b_{i_2}-s)  \\
   & \times
   \prod_{\substack{a_{j_1}\in \I\backslash\{a_{i_1}\}\\b_{j_2}\in \J\backslash\{b_{i_2}\}}}\zeta(1+a_{j_1}+b_{j_2}+s)
    \mathcal{A}( (\I \setminus \{ a_{i_1} \}) \cup \{ -b_{i_2}-s \},   
     ((\J \setminus \{ b_{i_2} \})+s)
      \cup \{ - a_{i_1} \} )\;ds\;dt\\&+ O(T^{8\delta}\log T).
 \end{align*}
 Hence,
\begin{align}\label{eqn:offdiag-CK}
&\sum_{i_{1}=1}^{k}\sum_{i_{2}=1}^{\ell}J_{i_{1},i_{2}}=\sum_{i_{1}=1}^{k}\sum_{i_{2}=1}^{\ell}\frac{c_{i_{1},i_{2}}}{2\pi i}\int_{-\infty}^{\infty}\omega(t)\int_{\Re(s)=2\varepsilon_1}\Phi_2(s) K^{s} \left( \frac{2\pi}{t}\right)^{a_{i_1}+b_{i_2}+s}\zeta(1-a_{i_1}-b_{i_2}-s)\nonumber\\
&\hspace{3em} \times
\prod_{\substack{a_{j_1}\in \I\backslash\{a_{i_1}\}\\b_{j_2}\in \J\backslash\{b_{i_2}\}}}\zeta(1+a_{j_1}+b_{j_2}+s) \mathcal{A}( (\I \setminus \{ a_{i_1} \}) \cup \{ -b_{i_2}-s \},   
     ((\J \setminus \{ b_{i_2} \})+s)
      \cup \{ - a_{i_1} \} )\;ds\;dt\nonumber\\&\hspace{3em}+ O(T^{8\delta}\log T).
\end{align}
We now remark that from definitions \eqref{ci1i2} and \eqref{ZIJid} we have 
\begin{equation*}
   c_{i_{1},i_{2}} = 
   \Z(\I \backslash \{ a_{i_1} \}, \{ -a_{i_1} \}) \Z( \{ - b_{i_2} \}, \J \backslash \{ -b_{i_2} \}),
\end{equation*}
and similarly from \eqref{ZIJid} we have 
\begin{equation*}
 \prod_{\substack{a_{j_1}\in \I\backslash\{a_{i_1}\}\\b_{j_2}\in \J\backslash\{b_{i_2}\}}}\zeta(1+a_{j_1}+b_{j_2}+s)
 = \Z(\I \backslash \{ a_{i_1} \} + \{ s \}  , \J \backslash \{ b_{i_2} \}  ).
\end{equation*}
From these identities, it follows from  \eqref{eqn:offdiag-CK} and \eqref{M1swaps} that 
\begin{equation*}
   \sum_{i_{1}=1}^{k}\sum_{i_{2}=1}^{\ell}J_{i_{1},i_{2}}
   =   \mathcal{M}_{1,\I, \J; \omega}(K)  + O(T^{8\delta}\log T).
\end{equation*}
\end{proof}

\section{The function $H_{\I,\J; \{ a_{i_1} \}, \{ b_{i_2} \}}(s)$}\label{sec:Hfunction}

In this section, we study the behaviour of the Dirichlet series $H_{\I,\J, \{ a_{i_1} \} , \{ b_{i_2} \}}(s)$  that was introduced in Section \ref{sec:off-diag}:

\[H_{\I,\J; \{ a_{i_1} \} , \{ b_{i_2} \}}(s)=\sum_{r=1}^{\infty}\sum_{q=1}^{\infty} \frac{c_q(r)G_{\I}(1-a_{i_1},q)G_{\J}(1-b_{i_2},q)  }{q^{2-a_{i_1}-b_{i_2}}r^{a_{i_1}+b_{i_2}+s}}.\]

For the rest of this section,  we establish the proof of Proposition \ref{Hid}. 
We require  \cite[Lemma~7.2]{Ng} which we state as follows. 
\begin{lem} \label{Gmult}
Let $k \in \mathbb{N}$, $I=\{1, \ldots, k \}$, and
$X =\{ x_1, x_2, \ldots, x_k\}$, where the $x_{i}$'s are distinct complex numbers. 
For $\Re(s) > -\min_{i=1, \ldots k} \Re(x_{i})$,
$p$ prime and $j \ge 1$, we have
\begin{equation}
 \label{GXspj}
  G_{X}(s,p^j) = 
(1-p^{-s-x_1}) \cdots (1-p^{-s-x_k}) \frac{1}{p-1}
   \sum_{i=1}^{k}
  \frac{ p^{1-x_i j} - p^{s-x_i (j-1)}}{1-p^{-x_i-s}}
  \prod_{\ell \in I \setminus \{ i \}} (1-p^{x_i-x_{\ell}})^{-1}.
\end{equation}
\end{lem}
We now further simplify the multiplicative functions $ G_{X}(s,p^j)$.   We shall express $ G_{X}(s,p^j)$ in terms of the rational function
\begin{equation}
  \label{Fa}
F_a(Y_1, \ldots, Y_m;Z) :=  \sum_{i=1}^{m} Y_i^a 
 \prod_{\ell \in M \backslash \{ i \}}   \frac{(1 - Z Y_{\ell})}{(1-Y_{\ell}/Y_i)},
 \end{equation}
 where $Y_1, \ldots, Y_m,Z$ are variables and $M=\{1,2,\dots,m\}$.  A key point will be to demonstrate that $F_a$ is a polynomial and this shall be established 
 in the combinatorial result,  Lemma \ref{comb}, which follows. This lemma was  proven by the authors in the case $a=1$.  In the case $a=2$, Gabriel Verret conjectured the formula \eqref{newidentity}, where
$q_{2,j}$ is given by \eqref{q2j}.  Based on this,  Dave Morris extended the lemma to the case  $a >1$.  In order to describe the polynomials $q_{a,j}$ which appear in the lemma, we
require the elementary symmetric polynomials.  Associated to variables $Y_1, \ldots, Y_m$, we let
\begin{align}
  e_1 := e_1(Y_1, \ldots, Y_{m}) & = Y_1 + \cdots + Y_m, \\
  e_2 := e_2(Y_1, \ldots, Y_{m}) & = Y_1 Y_2 + \cdots + Y_{m-1}Y_m = \sum_{1 \le i_1 < i_2 \le m}
  Y_{i_1} Y_{i_2}, \\
  & \vdots   \nonumber \\
  e_j := e_{j} (Y_1, \ldots, Y_{m}) & =   \sum_{1 \le i_1 < i_2 < \cdots < i_j \le m}
  Y_{i_1} Y_{i_2} \cdots Y_{i_j}, \\
  & \vdots \nonumber \\
  e_m :=e_{m}(Y_1, \ldots, Y_{m}) & =Y_1 \cdots Y_m. 
\end{align} 
It is convenient to set 
$e_0 := e_0(Y_1, \ldots, Y_m)=1$ and  $
 e_{m+1} := e_{m+1}(Y_1, \ldots, Y_m) =0$. 
 With this notation in hand, we may now state the lemma.
\begin{lem} \label{comb}
Let $a, m \in \mathbb{N}$. Then  
   \begin{equation}
 \label{newidentity}
 F_a(Y_1, \ldots, Y_m;Z) 
  = \sum_{j=0}^{m-1} q_{a,j}(Y_1, \ldots, Y_m) Z^{j},
\end{equation}
where $q_{a,j}(Y_1, \ldots, Y_m)$ is a polynomial in $\mathbb{Q}[Y_1, \ldots, Y_m]$ of degree $j+a$. 
We have
\begin{align}
  \label{q1j}
   q_{1,j} & : = q_{1,j}(Y_1, \ldots, Y_m) = (-1)^{j} e_{j+1}, \\
   \label{q2j}
   q_{2,j} &  : = q_{2,j}(Y_1, \ldots, Y_m) = (-1)^{j-1} (e_{j+2}-e_1 e_{j+1}  ) ,
\end{align}
and in general, 
\begin{equation}
  \label{qaj}
 q_{a,j}:=q_{a,j}(Y_1, \ldots, Y_m) \in \mathbb{Z}[e_1, \ldots, e_m].
\end{equation}
\end{lem}
A direct consequence of this lemma, is the following bound.  Let $Z \in \C$ and $X \geq1$.  If $|Y_j| \le X$ for each $1 \le j \le m$, then 
\begin{equation}
  \label{Fbound}
  | F_a(Y_1, \ldots, Y_m;Z)| \ll \sum_{j=0}^{m-1} X^{j+a} |Z|^j.
\end{equation}
\begin{proof}[Proof of Lemma \ref{comb}]
In order to simplify notation we set $F= F_a(Y_1, \ldots, Y_m;Z)$
and write
\begin{equation}
  \label{P}
 P=P(Y_1, \ldots, Y_m;Z) := \prod_{i \in M}  (1-ZY_i) =  \sum_{j=0}^{m} (-1)^j e_j Z^j.
\end{equation}
  Note that since $P(Z) = \sum_{j=0}^{m} (P^{(j)}(0)/j!) Z^j$, it follows that
\begin{equation}
  \label{Pj0}
   P^{(j)}(0) = (-1)^j j! e_j \text{ for }
   0 \le j \le m.  
\end{equation}
Let us consider the partial fraction 
decomposition of $1/(Z^aP)$.  This is of the form 
\begin{equation}\label{partial-fractions}
   \frac{1}{Z^aP} = \sum_{j=1}^{a} \frac{c_j}{Z^j}
   + \sum_{i=1}^{m} \frac{ Y_{i}^{a}}{\prod_{\ell \in M \backslash \{ i \} } (1-Y_{\ell}/Y_i)   } \cdot \frac{1}{ (1-ZY_i)}, 
\end{equation}
for certain polynomials $c_j \in \R[Y_1,\dots,Y_m]$. 
  Letting \begin{equation}\label{R}R(Z) := \sum_{i=0}^{a-1} c_{a-i} Z^i,\end{equation} we can rewrite this as 
\begin{equation*}
   \frac{1}{Z^aP} =  \sum_{j=1}^{a} \frac{c_j}{Z^j} + \frac{F(Z)}{P(Z)}
   =\frac{R(Z)}{Z^a} + \frac{F(Z)}{P(Z)}
\end{equation*}
which follows from \eqref{Fa} and \eqref{P}.
Note that
\begin{equation}
  \label{Rj0}
  R^{(j)}(0) = j! c_{a-j} \text{ for } 0 \le j \le a-1. \end{equation}
We now compute the coefficients $c_j$.   
Rearranging \eqref{partial-fractions} gives
\begin{equation}
  \label{pfnum}
  1 = P(Z)R(Z) +F(Z)Z^a.
\end{equation}
Letting $Z=0$   we obtain
\begin{equation}
  \label{caval}
 1= P(0)R(0) = c_a. 
\end{equation}

Then we differentiate \eqref{pfnum} $i$ times where $1 \le i \le a-1$.  Observe
that $\left.  \frac{d^i}{dZ^{i}} (F(Z) Z^a)  \right|_{Z=0} =0$, and thus we obtain 
by the generalized product rule, \eqref{Pj0},  and \eqref{Rj0}
\begin{equation}
   \label{pfnumdiff}
   0 =  \sum_{u+v=i} \binom{i}{u} P^{(u)}(0) R^{(v)}(0)
   =   \sum_{ u+v=i }   \binom{i}{u}   (-1)^u u! e_u v! c_{a-v}.
\end{equation}
Simplifying yields the condition
\begin{equation*}
  0 = \sum_{u+v=i} (-1)^u e_u c_{a-v}
  \text{ for } 1 \le i \le a-1. 
\end{equation*}
Note that if $i=1$ then this simplifies to $c_{a-1} = c_a \cdot e_1 = e_1$ 
(since $e_0 =1$), and we have used \eqref{caval}. 
Since $e_0=1$, we have
\begin{equation*}
  c_{a-i} = - \sum_{\substack{u+v=i \\ u \ge 1}} (-1)^u e_u c_{a-v}
  \text{ for } 1 \le i \le a-1.  
\end{equation*}
It follows from this that 
 $c_{a-i} \in \mathbb{Z}[e_1, \ldots, e_i]$ and $\text{degree}(c_{a-i})=i$.   From \eqref{pfnum} and the definitions of $P(Z)$ and $R(Z)$ in \eqref{P} and \eqref{R}, we get 
\begin{equation*}
  F =  \frac{1-P(Z) R(Z)}{Z^a} 
  = Z^{-a} \Big( 1 - \sum_{i=0}^{m+a-1}  \theta_i  Z^i  \Big),
  \text{ where }
  \theta_i = \sum_{ \substack{u+v=i \\ 0 \le u \le m \\ 0 \le v \le a-1} }   (-1)^u e_u c_{a-v}.
\end{equation*}
From the definition of $\theta_i$ and \eqref{caval}, we observe that $\theta_0 =1$. In addition, we have $\theta_i =0$ for $1 \le i \le a-1$ since $\left. \frac{d^i}{dZ^{i}} \left(P(Z)R(Z)\right)  \right|_{Z=0} =0$. This follows from \eqref{pfnum} and the observation that  $\left. \frac{d^i}{dZ^{i}} \left(F(Z) Z^a\right)  \right|_{Z=0} =0$. Thus, we have 
\begin{equation*}
   F = - \sum_{i=a}^{m+a-1}  \theta_i  Z^{i-a}
   = -  \sum_{j=0}^{m-1}  \theta_{a+j}  Z^{j},
\end{equation*}
where we note that the degree of $\theta_{a+j}$ is $a+j$. Setting $q_{a,j} :=q_{a,j}(Y_1, \ldots, Y_m) = - \theta_{a+j}$ gives the desired result.  
\end{proof}

\begin{lem}  \label{GIk}
Let $k \ge 2$,  let $\I = \{a_1,a_2,a_3, \ldots, a_k \} \subset \mathbb{C}$, $p$ a prime, and $j \ge 1$.  Then
\begin{equation}
\begin{split}
  \label{GIpn}
  G_{\I}(1-a_1,p^n) 
  & =  \sum_{j=0}^{k-2} q_{n,j}(X_2, \ldots, X_k)  (X_{1}^{-1})^j 
  p^{-j},
\end{split}
\end{equation}
where $q_{n,j}$ is defined in \eqref{qaj} and
\begin{equation}
  \label{Xi}
X_i = p^{-a_i} \text{ for } 1 \le i \le k.
\end{equation}
In particular, we obtain 
\begin{equation}
\label{GIp}
  G_{\I}(1-a_1,p) 
  = \sum_{j=0}^{k-2} (-1)^{j}  e_{j+1}(X_2, \ldots, X_k) 
   (X_{1}^{-1})^j 
  p^{-j}, 
\end{equation}
and 
\begin{equation}
\begin{split}
  \label{GIp2}
  G_{\I}(1-a_1,p^2) 
  & =  \sum_{j=0}^{k-2} (-1)^{j-1} ( e_{j+2}(X_2, \ldots, X_k)-
   e_{j+1}(X_2, \ldots, X_k)  e_1(X_2, \ldots, X_k)) (X_{1}^{-1})^j 
  p^{-j}.
\end{split}
\end{equation}
For any $n \ge 1$ and $\delta \in (0,\frac{1}{2})$ satisfying \eqref{sizerestrictiondelta}, we have the bound
\begin{equation}
  \label{GIpabd}
   |G_{\I}(1-a_1,p^n)| \ll  p^{ n \delta}. 
\end{equation}  
\end{lem}
\begin{proof}
We apply Lemma \ref{Gmult} with $s=1-a_1$ and $x_i=a_i$ for $1 \le i \le k$ 
to obtain
\begin{align*}
  & G_{\I}(1-a_1,p^j) 
   =  (1-p^{-(1-a_1)-a_1}) (1-p^{-(1-a_1)-a_2})  \cdots (1-p^{-(1-a_1)-a_k}) \frac{1}{p-1}  \\
   & \times \sum_{i=1}^{k}
  \frac{ p^{1-a_i j} - p^{(1-a_1)-a_i (j-1)}}{1-p^{-a_i-(1-a_1)}}
  \prod_{\ell \in I \setminus \{ i \}} (1-p^{a_i-a_\ell})^{-1}  \\
  & = \frac{1}{p}  \Big( \prod_{r=2}^{k} (1-p^{-1+a_1-a_r}) \Big)
  \sum_{i=1}^{k} \frac{ p^{1-a_i j}(1-p^{a_i-a_1}) }{ 1-p^{-1+a_1-a_i} }
  \prod_{\ell \in I \backslash \{ i \}} (1-p^{a_i-a_{\ell}})^{-1}. 
\end{align*}
Observe that the $i=1$ term vanishes and 
further simplification yields
\begin{align*}
  G_{\I}(1-a_1,p^j) & =
  \Big( \prod_{r=2}^{k} (1-p^{-1+a_1-a_r}) \Big)
  \sum_{i=2}^{k} \frac{ p^{-a_i j}(1-p^{a_i-a_1}) }{ 1-p^{-1+a_1-a_i} }
  \prod_{\ell \in I \backslash \{ i \}} (1-p^{a_i-a_{\ell}})^{-1} \\
  & = \sum_{i=2}^{k}  p^{-a_i j} 
   \Big( \prod_{r \in I \backslash \{ i, 1 \}} (1-p^{-1+a_1-a_r}) \Big)
  \prod_{\ell \in I \backslash \{ 1, i \}} (1-p^{a_i-a_{\ell}})^{-1} \\
  & =  \sum_{i=2}^{k}  X_i^{ j} 
   \Big( \prod_{r \in I \backslash \{ i, 1 \}} (1-X_r Z) \Big)
  \prod_{\ell \in I \backslash \{ 1, i \}} (1-X_{\ell}/X_i)^{-1},
\end{align*} 
where  $X_i=p^{-a_i}$ and $Z = p^{-1+a_1}$. 
We now let $m=k-1$ and set $Y_1 =X_2, \ldots, Y_m = X_k$ to get
\[
  G_{\I}(1-a_1,p^n) = F_n(X_2, \ldots, X_k; Z),
\]
where $F_n$ is defined in \eqref{Fa}.
By  Lemma \ref{comb}
we immediately obtain \eqref{GIp}, \eqref{GIp2}, and \eqref{GIpn}.  Finally, we note that 
by \eqref{Fbound} it follows that 
\[
  | G_{\I}(1-a_1,p^n) | \ll \sum_{j=0}^{k-2} (p^{\delta})^{j+n} (p^{\delta-1})^j
  \ll p^{n \delta}   \sum_{j=0}^{k-2} (p^{2 \delta-1})^j 
  \ll p^{n \delta}, 
\]
since $\delta < \frac{1}{2}$. 
\end{proof}
With these two lemmas in hand we are ready to prove Proposition \ref{Hid}.

\begin{proof}[Proof of Proposition \ref{Hid}] 
We shall prove the Proposition in the case $i_1=1$ and $i_2=1$.  The general case follows by a permutation of the variables.
 Throughout this proof we let $\sigma =\Re(s)$. 
Using the bound $|c_{q}(r)|  \le (q,r)$ and Lemma \ref{GIk}, we can check that the series 
for $H_{\I,\J;\{ a_1 \} , \{ b_1 \} }(s)$, defined in \eqref{H},
is absolutely convergent 
for  $\Re(s) >1 +2 \delta$ and $\delta < \frac{1}{4}$. Furthermore, 
since $c_{q}(r) = \sum_{d \mid (q,r)} d \mu(\tfrac{q}{d})$ we have
\begin{align*}
 H_{\I,\J;\{ a_1 \} , \{ b_1 \} }(s) & =\sum_{r=1}^{\infty}
      \sum_{q=1}^{\infty} \frac{G_{\I}(1-a_1,q)G_{\J}(1-b_1,q)  }{q^{2-a_1-b_1} r^{a_1+b_1+s}}
      \sum_{d \mid (q,r)} d \mu(\tfrac{q}{d}) 
        =  \sum_{q=1}^{\infty} \alpha_{q} \sum_{r=1}^{\infty} \frac{1}{r^c} 
   \sum_{d \mid q, d \mid r} d \mu(\tfrac{q}{d}), 
\end{align*} 
where 
$\alpha_{q} =  \frac{G_{\I}(1-a_1,q)G_{\J}(1-b_1,q)  }{q^{2-a_1-b_1}}$ and $c=a_1+b_1+s$.
Thus, 
\begin{equation}
\begin{split}
  \label{HIJe1}
    H_{\I,\J;\{ a_1 \} , \{ b_1 \} }(s) & =\sum_{q=1}^{\infty} \alpha_{q} \sum_{d \mid q} d 
    \mu(\tfrac{q}{d}) \sum_{r \ge 1, d \mid r} \frac{1}{r^c} 
  = \sum_{q=1}^{\infty} \alpha_{q} 
  \sum_{d \mid q}  \frac{d \mu(\tfrac{q}{d})}{d^c} \zeta(c) \\
  & = \zeta(c) \sum_{q=1}^{\infty} \alpha_{q} \sum_{d \mid q} d^{1-c} \mu(\tfrac{q}{d}) 
   = \zeta(c) \sum_{q=1}^{\infty} \alpha_{q}  q^{1-c} \sum_{d \mid q} \frac{\mu(d)}{d^{1-c}}.
\end{split}
 \end{equation}
 For prime powers $p^j$  we have $
    \sum_{d \mid p^j} \frac{\mu(d)}{d^{1-c}}
    = 1- \frac{1}{p^{1-c}}$.  By multiplicativity, we have
 \begin{equation}
 \begin{split}
  \label{ellsum} 
  \sum_{q=1}^{\infty} \alpha_{q}  q^{1-c} \sum_{d \mid q} \frac{\mu(d)}{d^{1-c}}
  & = \prod_{p} \Big(
  1 + \sum_{j=1}^{\infty} 
  \frac{G_{\I}(1-a_1,p^j)G_{\J}(1-b_1,p^j)}{(p^j)^{2-a_1-b_1}}
  (p^j)^{1-a_1-b_1-s} ( 1- p^{a_1+b_1+s-1})
  \Big)  \\
&   = \prod_{p} \Big( 
  1 + \sum_{j=1}^{\infty} G_{\I}(1-a_1,p^j)G_{\J}(1-b_1,p^j)
        \cdot \frac{  1-p^{a_1+b_1+s-1} }{ (p^j)^{1+s}}
  \Big) \\
 & =  \prod_{p} \Big( 1 + \sum_{j=1}^{\infty} T_j f_j \Big)
\end{split}
\end{equation}
where we have set 
\begin{equation}
  \label{Tjfj}
  T_j  = G_{\I}(1-a_1,p^j)G_{\J}(1-b_1,p^j)  \text{ and }  f_j  = \frac{  1-p^{a_1+b_1+s-1} }{ (p^j)^{1+s}}
  \text{ for } j \in \mathbb{Z}_{\ge 0}. 
\end{equation}
We now aim to simplify the last expression within the brackets
in \eqref{ellsum}.
At this point it will be convenient to introduce some notation.  Let 
\begin{equation}
  \label{variables}
U=p^{-1},  \ V = p^{-1-s}, \ X_i=p^{-a_i},  \ Y_i = p^{-b_i},  \text{ for } i\in\mathbb{N}. 
\end{equation}
Observe that we have the bounds 
\begin{align}
  \label{variableboundsa}
 &  |V| \le p^{-1- \sigma}, \\
  \label{variableboundsb}
 &  |X_i^{\pm 1}|, |Y_i^{\pm 1}| \le p^{\delta} \text{ for }  i=1,2,\dots.
\end{align}
Notice that 
\begin{equation}
  \label{fjformula}
  f_j =   V^j - X_{1}^{-1} Y_{1}^{-1} U^2 V^{j-1} \text{ for } j \in \mathbb{N}. 
\end{equation}
Given  $I= \{ i_1, \ldots, i_s \} \subset \K = \{ 1, \ldots, k \}$,
and $J = \{ j_1, \ldots, j_t \} \subset \Ll= \{ 1, \ldots, \ell \}$, we put 
\[
  X_I :=  X_{i_1} X_{i_2} \cdots X_{i_s}  \text{ and } Y_J := Y_{j_1}  Y_{j_2} \cdots Y_{j_t}. 
\]
We also define the polynomial rings 
\begin{equation}
\begin{split}
  \label{rings}
  & R_1 = \mathbb{Z}[X_1^{-1},X_2, \ldots, X_k],    \, 
  R_2 = \mathbb{Z}[Y_1^{-1}, Y_2, \ldots, Y_{\ell}],  \\ 
  & \text{and } R = \mathbb{Z}[X_{1}^{-1}, X_2, \ldots, X_k, Y_{1}^{-1}, Y_2, \ldots, Y_{\ell}]. 
 \end{split}
\end{equation}

By \eqref{GIpn} it follows that 
\begin{equation}\label{Tn}
\begin{split}
   T_n & = G_{\I}(1-a_1,p^n)G_{\J}(1-b_1,p^n) 
   =  \Big( \sum_{j=0}^{k-2} \alpha_{n,j} U^j \Big)
   \Big( \sum_{j'=0}^{\ell-2} \beta_{n,j'} U^{j'} \Big)  
    = \sum_{i=0}^{k+\ell-4} A_{n,i}  U^i, 
\end{split}
\end{equation}
where  by Lemma \ref{comb}
\begin{align}
  \label{alphabeta}
  \alpha_{n,j} &=  q_{n,j}(X_2, \ldots, X_k) (X_{1}^{-1})^j\in R_1,\nonumber \\ 
  \beta_{n,j'} &= q_{n,j'}(Y_2, \ldots, Y_{\ell}) (Y_{1}^{-1})^{j'}\in R_2, \nonumber\\
\end{align}
and thus 
\begin{equation}
  \label{Ani} 
     A_{n,i} = \sum_{j+j'=i}  \alpha_{n,j} \beta_{n,j'}\in R.
\end{equation}
In particular, by \eqref{q1j} we have  
\begin{equation}
  \label{A0}
  A_{1,0} := \alpha_{1,0} \beta_{1,0} =  e_1(X_2, \ldots, X_k) \cdot e_1(Y_2, \ldots, Y_{\ell})
  =\sum_{ \substack{ I \subset  \K_1   \\ J \subset \Ll_1 \\ |I|=|J|=1}}
 X_I Y_J,
\end{equation}
where 
\begin{equation*}
    \K_1=\K\setminus\{1\} \text{ and } \Ll_1=\Ll\setminus\{1\}.
\end{equation*}
By \eqref{q2j} we have
\begin{equation}\label{eqn:B0}
\begin{split}
  A_{2,0}  = \alpha_{2,0} \beta_{2,0}  & = (e_1(X_2, \ldots, X_k)^2 
  -e_2(X_2, \ldots, X_k))
  (  e_1(Y_2, \ldots, Y_{\ell})^2 
  -e_2(Y_2, \ldots, Y_{\ell}) ) \\
    & = \Big(  
    \sum_{ \substack{ I \subset  \K_1   \\  |I|=2}}
 X_I +
    \sum_{ \substack{ I \subset  \K_1   \\  |I|=1}}
 X_I^2 
   \Big) \cdot 
    \Big(  
    \sum_{ \substack{ J \subset  \Ll_1   \\  |J|=2}}
 Y_J +
    \sum_{ \substack{ J \subset \Ll_1   \\  |J|=1}}
 Y_J^2 
   \Big) \\
   & =  \sum_{ \substack{ I \subset  \K_1   \\ J \subset \Ll_1 \\ |I|=|J|=2}}
 X_I Y_J +
 \sum_{ \substack{ I \subset  \K_1  \\ J \subset \Ll_1 \\ |I|=2 , |J|=1}}
 X_I Y_J^2
 +\sum_{ \substack{ I \subset  \K_1   \\ J \subset \Ll_1 \\ |I|=1, |J|=2}}
 X_I^2 Y_J
 + \sum_{ \substack{ I \subset  \K_1   \\ J \subset \Ll_1 \\ |I|=|J|=1}}
 X_I ^2Y_J^2.
\end{split}
\end{equation}
Observe that 
$\text{deg}(\alpha_{n,j}) =2j+n$ since $q_{n,j}$ has degree $j+n$
and 
$\text{deg}(\beta_{n,j'}) = 2j'+n$ since $q_{n,j'}$ has degree $j'+n$.
It follows from \eqref{Ani} that $\text{deg}(A_{n,i}) = 2(i+n)$. 
Using this fact, with the bounds  \eqref{variableboundsb} we find that
\begin{equation}
    \label{Anibound}
    |A_{n,i}| \ll p^{2(n+i)\delta}.
\end{equation}
Therefore by \eqref{variables}, \eqref{variableboundsa}, and \eqref{variableboundsb}, we also have 
\begin{equation}
   \label{Aniterms1}
    A_{n,i} U^i V^n \ll p^{-(1-2 \delta) i - (1+\sigma-2 \delta)n},
\end{equation}
and
\begin{equation}
   \label{Aniterms2}
    A_{n,i} X_1^{-1}Y_1^{-1}U^{i+2} V^{n-1} \ll  p^{-(1-2\delta)i-(1+\sigma-2\delta)n+2\delta-1+\sigma}
  \end{equation}
for $n\ge1$ and $i \ge 0$. These bounds will be employed frequently in the sequel.
By \eqref{fjformula} and \eqref{Tn} we have 
\begin{equation}\label{T1f1}
\begin{split}
 T_1 f_1 & =  \Big( \sum_{i=0}^{k+\ell-4} A_{1,i}  U^i \Big) \Big(  V-X_{1}^{-1} Y_{1}^{-1}U^2\Big) 
  = A_{1,0}V+\sum_{i=1}^{k+\ell-4} A_{1,i}  U^i V-  \sum_{i=0}^{k+\ell-4} A_{1,i} X_{1}^{-1}Y_{1}^{-1} U^{i+2}.
\end{split}
\end{equation}

It follows from \eqref{Aniterms1} and \eqref{Aniterms2} that
\begin{equation}
\begin{split}
 \label{oneT1f1}
  1+T_1 f_1 
    &   =  1+A_{1,0} V + O \Big( 
  \sum_{i=1}^{k+\ell-4} 
   p^{-(1-2 \delta) i - (1+\sigma-2 \delta)}
   +  \sum_{i=0}^{k+\ell-4} 
    p^{-(1-2 \delta) (i+2)} \Big) \\
    & =  1+\Big( \sum_{ \substack{ I \subset  \K_1   \\ J \subset \Ll_1 \\ |I|=|J|=1}}
 X_I Y_J \Big) V + O( (p^{-2-\sigma}+ p^{-2})p^{4 \delta}). 
\end{split}
\end{equation}
Next, observe that by \eqref{Tjfj} and \eqref{GIpabd}
\begin{equation}
   \label{Tjbound}
    |T_j| \ll p^{2j \delta}
\end{equation}
and by \eqref{Tjfj}, \eqref{variables}, \eqref{variableboundsa}, and \eqref{variableboundsb}
\begin{equation}
    \label{fjbound}
    |f_j| \ll |V|^j + |X_1^{-1} Y_1^{-1}|p^{-2} |V|^{j-1}
    \ll (p^{-1-\sigma})^j + (p^{-1-\sigma})^{j-1} p^{2\delta-2}.
\end{equation}
Combining these bounds gives
\begin{equation}
    \label{Tjfjgr2}
    \sum_{j=2}^{\infty} T_j f_j \ll 
    \sum_{j=2}^{\infty} p^{2j \delta} ( (p^{-1-\sigma})^j + (p^{-1-\sigma})^{j-1} p^{2\delta-2})
    =\sum_{j=2}^{\infty} (p^{2\delta-1-\sigma})^j(1+p^{\sigma+2 \delta-1})
    \ll 1
\end{equation}
for $\Re(s) \ge 0$. Therefore, from \eqref{T1f1} and \eqref{Tjfjgr2} we deduce, that for $\Re(s) > 2 \delta$, the sum over $q$ in \eqref{ellsum} equals
$
  \prod_{  (i,j) \in \K_1 \times \Ll_1 } \zeta(1+a_i+b_j+s)
 \mathcal{C}_{\I,\J; \{ a_1 \} , \{ b_1 \}}(s),
$
where
\begin{equation*}
    \mathcal{C}_{\I,\J; \{ a_1 \} , \{ b_1 \}}(s)
  = \prod_{p} \mathcal{C}_{\I,\J; \{ a_1 \} , \{ b_1 \}}(p;s),
\end{equation*}
  and
\begin{equation}
\begin{split}
  \label{CIJa1b1ps}
   \mathcal{C}_{\I,\J; \{ a_1 \} , \{ b_1 \}}(p;s) 
&  =   \Big(
 1 + \sum_{j=1}^{\infty} \frac{G_{\I}(1-a_1,p^j)G_{\J}(1-b_1,p^j)}{(p^j)^{1+s}}
    (1-p^{a_1+b_1+s-1})
\Big)\\&\hspace{3em}\times   \prod_{(i,j) \in \K_1 \times \Ll_1} 
 \Big(1- \frac{1}{p^{1+a_i+b_j+s}} \Big).
\end{split}
\end{equation}
Hence,
\begin{align*}
   & H_{\I,\J;\{ a_1 \} , \{ b_1 \} }(s) = \zeta(a_1+b_1+s) 
 \prod_{(i,j) \in \K_1 \times \Ll_1}  
  \zeta(1+a_i+b_j+s)  
 \mathcal{C}_{\I,\J; \{ a_1 \} , \{ b_1 \}}(s).
\end{align*}
We now establish \eqref{CIJexpansion}.  
It is convenient to set $
    \Pi = 
    \prod_{ \substack{ i \in \K_1   \\ j \in \Ll_1} }
    (1 - X_i Y_j V)$.
Observe that 
\begin{equation*}
\begin{split}
    \Pi & 
    =   \sum_{i=0}^{k+\ell-2} g_i V^i,
\end{split}
\end{equation*}
where $g_i \in R$, $\text{deg}(g_i)  =2i$, and $g_0=1$. 
We find that 
\begin{equation*}\label{eqn:g1}
 g_1  = -\sum_{ \substack{ I \subset  \K_1   \\ J \subset \Ll_1 \\ |I|=|J|=1}}
 X_I Y_J =-e_1(X_2, \ldots, X_k) e_1(Y_2, \ldots, Y_{\ell})
 =-A_{1,0}, 
 \end{equation*}
by \eqref{A0}, and
\begin{equation}
   \label{eqn:g2}
  g_2  =   \sum_{ \substack{ I \subset  \K_1   \\ J \subset \Ll_1 \\ |I|=1, |J|=2}}
 X_I^2 Y_J+ \sum_{ \substack{ I \subset  \K_1   \\ J \subset \Ll_1 \\ |I|=2, |J|=1}}
 X_I Y_J^2
 + 2   \sum_{ \substack{ I \subset  \K_1   \\ J \subset \Ll_1 \\ |I|=|J|=2}}
 X_I Y_J.
\end{equation}

From \eqref{CIJa1b1ps} it follows that 
\begin{equation}
  \label{CIJTjfjPi}
  \mathcal{C}_{\I,\J; \{ a_1 \} , \{ b_1 \}}(p;s) = \Big( 1+\sum_{j=1}^{\infty} T_j f_j \Big) \Pi. 
\end{equation}
We now consider $T_2f_2$.  
By \eqref{fjformula} and \eqref{Tn} we have 
\begin{equation}\label{T2f2}
\begin{split}
 T_2 f_2  & =  \Big( \sum_{i=0}^{k+\ell-4} A_{2,i}  U^i \Big) \Big(  V^2 - X_{1}^{-1} Y_{1}^{-1} U^2 V  \Big) 
 \\&= A_{2,0} V^2+ \sum_{i=1}^{k+\ell-4} A_{2,i}  U^i  V^2 -  \sum_{i=0}^{k+\ell-4} A_{2,i}  X_{1}^{-1} Y_{1}^{-1} U^{i+2} V.
\end{split}
\end{equation}
Combining \eqref{T1f1} and \eqref{T2f2} yields
\begin{equation*}
\begin{split}
  1+ T_1 f_1 + T_2 f_2 
  & =1+A_{1,0}V+\sum_{i=1}^{k+\ell-4} A_{1,i}  U^i V-  \sum_{i=0}^{k+\ell-4} A_{1,i} X_{1}^{-1}Y_{1}^{-1} U^{i+2}\\&\hspace{2em}+ 
  A_{2,0} V^2+ \sum_{i=1}^{k+\ell-4} A_{2,i}  U^i  V^2 -  \sum_{i=0}^{k+\ell-4} A_{2,i}  X_{1}^{-1} Y_{1}^{-1} U^{i+2} V.
\end{split}
\end{equation*}
Thus,
\begin{equation}
\begin{split}
  \label{initialsum}
& (1+ T_1 f_1 + T_2 f_2) \Pi =  \\
&
  \Big( 1+ A_{1,0} V + A_{2,0} V^2 + \sum_{i=1}^{k+\ell-4} A_{1,i}  U^i V-  \sum_{i=0}^{k+\ell-4} A_{1,i} X_{1}^{-1}Y_{1}^{-1} U^{i+2}
 \\&\hspace{2em} + \sum_{i=1}^{k+\ell-4} A_{2,i}  U^i  V^2 -  \sum_{i=0}^{k+\ell-4} A_{2,i}  X_{1}^{-1} Y_{1}^{-1} U^{i+2} V
  \Big)   \Big(  1 -A_{1,0} V + g_2 V^2 + \sum_{i=3}^{k+\ell-2} g_i V^i  \Big).
\end{split}
\end{equation}
Observe that 
\begin{equation}
  \label{firstterms}
(1+ A_{1,0} V + A_{2,0} V^2)(  1 -A_{1,0} V + g_2 V^2)
= 1 + (g_2+A_{2,0}-A_{1,0}^2  )V^2  + (A_{1,0} g_2 -A_{1,0}A_{2,0}) V^3 +g_2 A_{2,0} V^4. 
\end{equation}
By expanding out \eqref{initialsum} and using \eqref{firstterms} we find that 
\begin{equation}
\begin{split}
  \label{jle2bounds}
    & (1+ T_1 f_1 +T_2 f_2) \Pi  =1 + c_{02}V^2 + \sum_{\substack{u \ge 2 \\ (u,0) \in S}} c_{u0} U^u + 
    \sum_{\substack{u \ge 1 \\ (u,0) \in S}} c_{u1} U^u V + \sum_{\substack{(u,v) \in S \\  v \ge 2, (u,v) \ne (0,2) }} c_{uv} U^u V^v,
\end{split}
\end{equation}
where $S\subset\mathbb{Z}_{\geq0}\times\mathbb{Z}_{\geq0}$ is finite, $c_{uv} \in R$,
$\text{deg}(c_{uv}) = 2(u+v)$, and $c_{02} = g_2+A_{2,0}-A_{1,0}^2$. 
Here we have made use of the fact that 
$\text{deg}(A_{n,i})=2(i+n)$ and $\text{deg}(g_i)=2i$.   Furthermore, we have
\begin{equation}
   \label{cuvterms}
    c_{uv} U^u V^v \ll p^{2(u+v) \delta-u-(1+\sigma)v}
    = p^{-(1-2 \delta) u - (1+\sigma-2 \delta)v}
    \text{ for } u,v \ge 0. 
\end{equation}
In particular, if $v\geq2$ and $(u,v)\neq(0,2)$, we have 
\begin{equation}
\label{cuvtermsvgeq2}
    c_{uv} U^u V^v \ll p^{-(u+v)\min(1-2 \delta+\sigma,1-2\delta)}\leq p^{-3(1-2 \delta+\min(\sigma,0))}.
 \end{equation}
Given that $S$ is finite, that  $\sigma>-1+2\delta$, and that we certainly have $\delta < \frac{1}{2}$, it follows from \eqref{variables}, \eqref{jle2bounds},  \eqref{cuvterms} and \eqref{cuvtermsvgeq2}
that
\begin{equation}
\begin{split}
   \label{firstparteq}
    (1+ T_1 f_1 +T_2 f_2) \Pi  & =1 + c_{02}V^2 + O \Big(   p^{-2(1-2 \delta)}   +    
    p^{-(1-2 \delta)-(1+\sigma-2\delta)}  
    +p^{-3\left(1-2\delta+\min(\sigma,0)\right)}  \Big)  \\
    & =1+c_{02} V^2 + O \Big( p^{6 \delta+ \max\left(-2,-2-\sigma,-3-3 \min(\sigma,0)\right)} \Big) \\
     & =1+c_{02} V^2+ O \Big( p^{6 \delta +\vartheta(\sigma)} \Big) ,
\end{split}
\end{equation}    
where $\vartheta(\sigma)$ is defined by \eqref{thetasigma}.  Finally, we bound the contribution from $j \ge 3$ to 
\eqref{CIJTjfjPi}.  We make use of
\eqref{Tjbound}, \eqref{fjbound}
and $|\Pi| \le (1+p^{-(1+ \sigma-2 \delta)}  )^{(k-1)(\ell-1)} < 2^{(k-1)(\ell-1)}  $  (since $\sigma > -1+2\delta$)  to obtain 
\begin{equation}
\begin{split}
  \label{jg3bounds}
 &  \Big( \sum_{j=3}^{\infty}  T_j f_j  \Big) \Pi   \ll 
  \sum_{j=3}^{\infty}   p^{2j \delta} \Big(  \Big( \frac{1}{p^{1+ \sigma}} \Big)^{j} +   \frac{p^{2 \delta}}{p^2} 
   \Big( \frac{1}{p^{1+ \sigma}} \Big)^{j-1} \Big)  \\
   & \ll  p^{6 \delta } \Big(  \Big( \frac{1}{p^{1+ \sigma}} \Big)^{3} +   \frac{p^{2 \delta}}{p^2} 
   \Big( \frac{1}{p^{1+ \sigma}} \Big)^{2} \Big)  
       \le  p^{8 \delta} \Big(   \frac{1}{p^{3+3 \sigma}}   +  
    \frac{1}{p^{4+2 \sigma}} \Big)  
     \ll p^{8 \delta} p^{\vartheta(\sigma)}
\end{split}
\end{equation}
for $\sigma \ge -1+2\delta +\e$. 

Next, we simplify the formula for $c_{02}$ in \eqref{firstparteq}.  By \eqref{A0} we have
\begin{equation*}
\begin{split}
A_{1,0}^2 = \sum_{\substack{ I,I' \subset \K_1 \\
J,J' \subset \Ll_1
 }} X_I X_{I'} Y_J Y_{J'}. 
\end{split}
\end{equation*}
We proceed as follows in 4 cases:  (i) $I=I'$, $J=J'$, (ii)  $I\ne I'$, $J=J'$,  (iii) $I= I'$, $J \ne J'$, 
and (iv) $I \ne I'$, $J \ne J'$ to obtain 
\begin{equation*}
A_{1,0}^2 = \sum_{ \substack{ I \subset  \K_1   \\ J \subset \Ll_1 \\ |I|=|J|=1}} X_I^2 Y_J^2
+ 2\sum_{ \substack{ I \subset  \K_1   \\ J \subset \Ll_1 \\ |I|=2 , |J|=1}} X_I Y_J^2
+  2\sum_{ \substack{ I \subset  \K_1   \\ J \subset \Ll_1 \\ |I|=1 , |J|=2}} X_I^2 Y_J
+ 4 \sum_{ \substack{ I \subset  \K_1   \\ J \subset \Ll_1 \\ |I|= |J|=2}} X_I  Y_J.
\end{equation*}

From  \eqref{eqn:B0} and \eqref{eqn:g2}  we see that 
\begin{equation*}
 g_2+A_{2,0} 
 =  3\sum_{ \substack{ I \subset \K_1  \\ J \subset \Ll_1 \\ |I|=|J|=2}}
 X_I Y_J +
 2\sum_{ \substack{ I \subset  \K_1   \\ J \subset \Ll_1 \\ |I|=2 , |J|=1}}
 X_I Y_J^2
 +2\sum_{ \substack{ I \subset  \K_1   \\ J \subset\Ll_1 \\ |I|=1, |J|=2}}
 X_I^2 Y_J
 + \sum_{ \substack{ I \subset \K_1   \\ J \subset \Ll_1 \\ |I|=|J|=1}}
 X_I ^2Y_J^2
\end{equation*}
and it follows that 
\begin{equation}
  \label{c02identity}
 c_{02}= g_2+A_{2,0}-A_{1,0}^2 = -\sum_{ \substack{ I \subset  \K_1   \\ J \subset \Ll_1 \\ |I|=|J|=2}}
 X_I Y_J.
\end{equation}

Combining \eqref{CIJTjfjPi}, \eqref{firstparteq}, \eqref{jg3bounds}, and \eqref{c02identity} we establish 
\eqref{CIJexpansion} along with \eqref{thetasigma}.
Note that from \eqref{thetasigma}, it follows that 
if  $\sigma=\Re(s) = -\frac{1}{2}+2\delta+\e$ with $\e >0$, then \[ \mathcal{C}_{\I,\J; \{ a_1 \} , \{ b_1 \}}(p;s) 
= 1 + O(p^{-1-2 \e}).\] Therefore, $\mathcal{C}_{\I,\J; \{ a_1 \} , \{ b_1 \}}(s)$ is holomorphic and 
absolutely convergent for $\Re(s) > - \frac{1}{2} + 2\delta$.  Furthermore, we see that the poles listed 
in \eqref{Hpoles}  arise from the zeta factors in \eqref{Hfactorization}.
\end{proof}

\section{Technical Lemmas}\label{sec:technical-lemmas}
In this section we establish Lemma \ref{fderivativebd} and  Lemma \ref{Stirling}. 

\begin{proof}[Proof of Lemma \ref{fderivativebd}]
We have 
$$f_{r,M,N}(x,y)=W\left(\frac{x}{M}\right)W\left(\frac{y}{N}\right)\varphi\left(\frac{x}{K}\right)\varphi\left(\frac{y}{K}\right) a_r(y)$$
where 
\begin{equation}
  \label{a}
   a(y)=a_r(y) =  \frac{\widehat{\omega}(\tfrac{1}{2 \pi}\log(1+\tfrac{r}{y}))}{T}
   = \frac{1}{T} \int_{-\infty}^{\infty} \omega(t)   \Big(1+\frac{r}{y} \Big)^{-it}   \, dt 
   ,
\end{equation}
and $1 \le |r| \ll \tfrac{M}{T_0} T^{\varepsilon}$.  First we shall show
\begin{equation}
   \label{avbd}
    y^{v} \frac{d^v}{dy^v} a(y) \ll P^{v} \text{ for } v \ge 0 \text{ where } P= \Big(\frac{T}{T_0} \Big) T^{\e},
\end{equation}
and then deduce the lemma from this bound.  The case $v=0$ is trivial.  Observe that for $v \ge 1$
\begin{equation}
   \label{avformula}
     \frac{d^v}{dy^v}   a(y)  =  
\frac{1}{T} \int_{-\infty}^{\infty}  \omega(t) \frac{d^v}{dy^v}   \Big( 
1+\frac{r}{y} \Big)^{-it}    dt. 
\end{equation}
It is shown in \cite[Equation~8.18]{Ng} that for 
 $v \ge 1$, one has
\begin{equation}
\begin{split}
  \label{derivativeoneplusryit}
\Big| \frac{d^{v}}{dy^{v}} \Big( 1 + \frac{r}{y} \Big)^{-it} \Big| 
& 
\ll_{\nu} \Big( \frac{P}{y} \Big)^{v}
\end{split}
\end{equation}
when $y \asymp N \asymp M$, $t \asymp T$, and 
$1 \le |r|  \ll \frac{M}{T_0} T^{\varepsilon} =o(M)$.  Inserting this last bound in \eqref{avformula} establishes \eqref{avbd}. 

We now deduce the lemma. 
Observe that for $m \ge 0$, we have
\begin{align*}
  \frac{d^m}{dx^m}  W\left(\frac{x}{M}\right)\varphi\left(\frac{x}{K}\right)
  =\sum_{i+j=m} \binom{m}{i}  \frac{d^i}{dx^i} W\left(\frac{x}{M}\right) \frac{d^j}{dx^j} \varphi\left(\frac{x}{K}\right)
  & \ll \sum_{i+j=m} \binom{m}{i}  M^{-i} K^{-j}  \\
  & = \Big( \frac{1}{M}+\frac{1}{K} \Big)^{m} \ll M^{-m}, 
\end{align*}
since $M \le K$. 
Similarly, $ \frac{d^u}{dx^u}  W\left(\frac{y}{N}\right)\varphi\left(\frac{y}{K}\right)  \ll N^{-u}$ for $u \ge 0$.
By the generalized product rule in conjunction with the last two derivative bounds and \eqref{avbd}
\begin{align*}
    f^{(m,n)}(x,y)  & = \frac{d^m}{dx^m} \Big( W\left(\frac{x}{M}\right)\varphi\left(\frac{x}{K}\right) \Big)
   \frac{d^n}{dy^n}  \Big( W\left(\frac{y}{N}\right)\varphi\left(\frac{y}{K}\right)  a(y)  \Big) \\
   & \ll M^{-m} \sum_{u+v=n} \binom{n}{u} \frac{d^u}{dy^{u}} \Big( W\left(\frac{y}{N}\right)\varphi\left(\frac{y}{K}\right)\Big)
   a^{(v)}(y) 
    \ll   M^{-m} \sum_{u+v=n} \binom{n}{u} N^{-u} y^{-v}  P^v \\
  &  \ll M^{-m} N^{-n} 
   \sum_{u+v=n} \binom{n}{u} P^{v} 
    = M^{-m} N^{-n}   (1+P)^n. 
\end{align*}
Therefore $
    x^m y^n f^{(m,n)}(x,y) \ll  P^n$
since $P \ge 1$, $x \asymp M$, and $y \asymp N$.  
 \end{proof}
We now prove Lemma \ref{Stirling}, which makes extensive use of Stirling's formula.

\begin{proof}[Proof of Lemma \ref{Stirling} (i)]
Let $i_1 \in \K$ and $i_2 \in \Ll$. 
We write  
\begin{equation}
  \label{s1s2}
a_{i_1}+s_{1}=\sigma_{1}+iu_{1} \text{ and }  b_{i_2}+s_{2}=\sigma_{2}+iu_{2}.
\end{equation}
  We begin by assuming that $|u_{i}|\leq\sqrt{t}$ for $i=1,2$.  Let $\log z$ be the principal branch of the logarithm, so that $-\pi<\Im(\log z)<\pi$ for $z\in\mathbb{C}\backslash (-\infty,0]$. For $\epsilon>0$ and $0<a\leq1$, we have 
\[\log \Gamma(z+a)= \Big(z+a-\frac12 \Big)\log z-z+\frac12\log(2\pi)+O(|z|^{-1})\] in the sector $\left|\arg(z)\right|\leq\pi-\epsilon$ (see \cite[Section 13.6]{WW}). We have
 \[\log \Gamma\left(\frac{1}{2}-b_{i_2}-s_2+it\right)=-\sigma_{2}\ln(t-u_{2})-\frac{\pi}{2}(t-u_{2})+i\left((t-u_{2})\ln(t-u_{2})-\sigma_{2}\frac{\pi}{2}\right)+O\left(\frac{1}{t-u_{2}}\right),\] and 
 \[\log \Gamma\left(\frac{1}{2}+a_{i_1}+s_1+it\right)=\sigma_{1}\ln(t+u_{1})-\frac{\pi}{2}(t+u_{1})+i\left((t+u_{1})\ln(t+u_{1})+\sigma_{1}\frac{\pi}{2}\right)+O\left(\frac{1}{t+u_{1}}\right).\]
Hence, \begin{align*}&\log \Gamma\left(\frac{1}{2}-b_{i_2}-s_2+it\right)-\log \Gamma\left(\frac{1}{2}+a_{i_1}+s_1+it\right)\\&=(-s_{2}-b_{i_2})\ln(t-u_{2})-(s_{1}+a_{i_1})\ln(t+u_{1})+it(\ln(t-u_{2})-\ln(t+u_{1}))\\&\hspace{2em}+i(u_{1}+u_{2})-i\frac{\pi}{2}(s_{1}+s_{2}+a_{i_1}+b_{i_2})+O\left(\frac{1}{t-u_{2}}\right)+O\left(\frac{1}{t+u_{1}}\right).\end{align*}
Notice that 
\begin{equation}
 \label{ln1}
 \ln(t-u_{2})=\ln t+\ln\left(1-\frac{u_{2}}{t}\right)
 \end{equation}
and 
\begin{equation}
  \label{ln2}
 \ln(t+u_{1})=\ln t+\ln\left(1+\frac{u_{1}}{t}\right).
\end{equation}
Using this observation, we get
\begin{align*}&\log \Gamma\left(\frac{1}{2}-b_{i_2}-s_2+it\right)-\log \Gamma\left(\frac{1}{2}+a_{i_1}+s_1+it\right)\\&=-(s_{1}+s_{2}+a_{i_1}+b_{i_2})\ln t-(s_{2}+b_{i_2})\ln \Big(1-\frac{u_{2}}{t} \Big)-(s_{1}+a_{i_1})\ln \Big(1+\frac{u_{1}}{t} \Big)+it(\ln(t-u_{2})-\ln(t+u_{1}))\\&\hspace{2em}+i(u_{1}+u_{2})-i\frac{\pi}{2}(s_{1}+s_{2}+a_{i_1}+b_{i_2})+O\left(\frac{1}{t-u_{2}}\right)+O\left(\frac{1}{t+u_{1}}\right).\end{align*}
Exponentiating both sides of the above equation yields
\begin{align*}
\frac{\Gamma(\frac12-b_{i_2}-s_{2}+it)}{\Gamma(\frac12+a_{i_1}+s_{1}+it)}&=t^{-(s_{1}+s_{2}+a_{i_1}+b_{i_2})}\exp\left(-i\frac{\pi}{2}(s_{1}+s_{2}+a_{i_1}+b_{i_2})\right)\exp\left(it(\ln(t-u_{2})-\ln(t+u_{1}))\right)\\&\hspace{2em}\times\exp\left(i(u_1+u_2)\right)\exp\left(-(s_{2}+b_{i_2})
\ln \Big(1-\frac{u_{2}}{t} \Big)\right)\exp\left(-(s_{1}+a_{i_1})\ln \Big(1+\frac{u_{1}}{t} \Big)\right)\\&\hspace{2em}\times\exp\left(O\left(\frac{1}{t-u_{2}}\right)+O\left(\frac{1}{t+u_{1}}\right)\right).
\end{align*}
Notice that \eqref{ln1} and \eqref{ln2} imply 
 \[it(\ln(t-u_{2})-\ln(t+u_{1}))+i(u_{1}+u_{2})=O\left( \frac{u_{1}^2+u_{2}^2}{t}\right).\]
Moreover, we have \[-(s_{2}+b_{i_2})\ln \Big(1-\frac{u_{2}}{t} \Big)-(s_{1}+a_{i_1})\ln \Big(1+\frac{u_{1}}{t} \Big)=-(s_{2}+b_{i_2})O\left(\frac{u_{2}}{t}\right)-(s_{1}+a_{i_1})O\left(\frac{u_{1}}{t}\right)=O\left(\frac{u_{1}^2+u_{2}^2}{t}\right),\]
and
\[\exp\left(O\left(\frac{u_{1}^2+u_{2}^2}{t}+\frac{1}{t}\right)\right)=1+O\left(\frac{u_{1}^2+u_{2}^2+1}{t}\right)=1+O\left(\frac{|s_{1}|^2+|s_{2}|^2+1}{t}\right).\]
Therefore, when $|s_{1}|^2+|s_{2}|^2\ll t$, we have 
\begin{equation*}
\frac{\Gamma(\frac12-b_{i_2}-s_{2}+it)}{\Gamma(\frac12+a_{i_1}+s_{1}+it)}=t^{-(s_{1}+s_{2}+a_{i_1}+b_{i_2})}\exp\left(-i\frac{\pi}{2}\left(s_{1}+s_{2}+a_{i_1}+b_{i_2}\right)\right)\left(1+O\left(\frac{|s_{1}|^2+|s_{2}|^2+1}{t}\right)\right).
\end{equation*}

When $|s_1|^2+|s_2|^2 \gg t$, the $O$ term in \eqref{ratio-gamma-0} becomes larger than 1, and so the dominant term becomes
\[
   \frac{|s_1|^2+|s_2|^2}{t}  t^{-s_1-s_2-a_{i_1}-b_{i_2}} \exp \Big(i\frac{\pi}{2}(-s_1-s_2-a_{i_1}-b_{i_2})   \Big) .
\]
Hence, it suffices to establish for $|s_1|^2+|s_2|^2 \gg t$ that 
\[
     \Big| \frac{\Gamma(\frac{1}{2}-b_{i_2}-s_2+it)}{\Gamma(\frac{1}{2}+a_{i_1}+s_1+it)} \Big|
     \ll \frac{|s_1|^2+|s_2|^2}{t^{1+\Re(s_1+s_2+a_{i_1}+b_{i_2})}} \Big|   e^{\frac{\pi}{2}\left( \Im(s_1+s_2+a_{i_1}+b_{i_2})\right)}
     \Big|,
\]
assuming 
\begin{equation}
   \label{sigmaconditions}
   \sigma_1 \in (0,A-\tfrac12], \, \sigma_2 \in (0, \tfrac{1}{2}-\eta_0]\cup[\tfrac12+\eta_0,\tfrac32-\eta_0], 
   \text{ and } \sigma_1 + \sigma_2 \le 1. 
\end{equation}
Here $\eta_0\in (0,\frac12)$ and $A>0$ are fixed constants. Since we assume that $\sigma_1$ and $\sigma_2$ are in  bounded intervals, then $|s_1| \asymp \left|u_1\right|$ and $|s_2|\asymp \left|u_2\right|$. We proceed to prove the asymptotic bound
\begin{equation}
 \label{asymptoticbound}
   \Big| \frac{\Gamma(\frac{1}{2}-b_{i_2}-s_2+it)}{\Gamma(\frac{1}{2}+a_{i_1}+s_1+it)} \Big| 
    \ll  \frac{u_1^2+u_2^2}{t^{1+\sigma_1+\sigma_2}}   e^{\frac{\pi}{2} (u_1+u_2)}
    \text{ where }  \sqrt{t}\leq |u_1|, |u_2| \leq t+1. 
\end{equation}
We write the interval $[\sqrt{t},t+1]=I_1 \cup I_2 \cup I_3$ where 
\begin{equation}
  \label{intervals}
  I_1 = [\sqrt{t},\tfrac{t}{2}], \ 
  I_2 = [\tfrac{t}{2},t-1], \
  I_3 = [t-1,t+1]. 
\end{equation}
There are nine cases according to $|u_1| \in I_i, |u_2| \in I_j$
with 
\begin{equation*}
  \label{cases}
(i,j) \in \Big\{ (1,1), (1,2),(1,3), (2,1),(2,2),(2,3),(3,1),(3,2),(3,2) \Big\}.
\end{equation*} 
Recall the Stirling estimate
\begin{equation}
  \label{Stirlingasym}
  |\Gamma(\s+iu)| =\sqrt{2 \pi} |u|^{\s-1/2}e^{-\frac{1}{2} \pi  |u|} (1+O(u^{-1}))
  \text{ for } 0<\sigma\ll 1 \text{ and } |u| \ge 1. 
\end{equation}
Note that we also have the bounds 
\begin{align}
   \label{gammaO1}
     |\Gamma(\sigma+iu)|   \ll_{\eta_0, A} 1  \text{ for }   |u| \le 1, \sigma\in[-1+\eta_0,-\eta_0]\cup[\eta_0,A],
\end{align}
and 
\begin{equation}
     \label{gammam1O1}
     |\Gamma(\sigma+iu)|^{-1}  \ll_{A} 1  \text{ for }   -A \le \sigma \le A, \, |u| \le 1 ,
\end{equation}
where $\eta_0\in (0,\frac12)$ and $A>0$ are fixed constants.  The last two bounds follow from the facts that $\Gamma(s)$ 
is holomorphic on and within the given region and $\Gamma(s)^{-1}$ is entire. 
Thus, we obtain
\begin{equation}
    \label{gammanum}
     | \Gamma(\tfrac{1}{2}-b_{i_2}-s_2+it) | 
         \ll
        \begin{cases}
       |t-u_2|^{-\sigma_2} e^{-\frac{\pi}{2}|t-u_2|  } & \text{ if }  |t-u_2|\geq1, \\
       1 & \text{ if } |t-u_2|\leq 1,\; \sigma_2 \in [\tfrac12-A,\tfrac12-\eta_0]\cup[\tfrac12+\eta_0,\tfrac32-\eta_0]     \end{cases}
\end{equation}
and
\begin{equation}
  \label{gammaden}
     |\Gamma(\tfrac{1}{2}+a_{i_1}+s_1+it) |^{-1} 
       \ll
        \begin{cases}
       |t+u_1|^{-\sigma_1} e^{\frac{\pi}{2}|t+u_1|}&  \text{ if }  |t+u_1|\geq1, \\
       1 & \text{ if } |t+u_1| \ll 1, \sigma_1 \in [-A-\tfrac12,A-\tfrac12].
      \end{cases}
\end{equation}
We provide full details for Cases (i), (v), (vi), and (ix) and note that the 
remaining cases are treated similarly. \\
\noindent Case (i): $(i,j)=(1,1)$.   
It may be checked that the conditions $\sqrt{t} \le |u_1| \le \frac{t}{2}$ and $ \sqrt{t} \le |u_2| \le \frac{t}{2}$ imply
$t+u_1, t-u_2 \in [t-\sqrt{t}, \tfrac{3t}{2}]$. 
Thus
\begin{align*}
      \Big|\frac{\Gamma(\frac{1}{2}-b_{i_2}-s_2+it)}{\Gamma(\frac{1}{2}+a_{i_1}+s_1+it)} \Big| 
      & \asymp  \frac{ 1 }{(t-u_2)^{\s_2}(t+u_1)^{\s_1}  }
      \frac{e^{-\frac{\pi}{2}(t-u_2)  } }{ e^{-\frac{\pi}{2}(t+u_1)} }
     =  \frac{ e^{\frac{\pi}{2}(u_1+u_2)} }{(t-u_2)^{\s_2}(t+u_1)^{\s_1}  }    \ll\frac{(u_1^2+u_2^2) e^{\frac{\pi}{2}(u_1+u_2)} }{t^{1+\s_1+\s_2}}.
\end{align*}
\noindent Case (v): $(i,j)=(2,2)$.  In this case, we have  $\frac{t}{2} \le |u_1|, |u_2| \le t-1$.  
This implies that $t+u_1 \geq 1$ and $t-u_2 \geq1$.  Thus it follows that 
\begin{align*}
      \Big|\frac{\Gamma(\frac{1}{2}-b_{i_2}-s_2+it)}{\Gamma(\frac{1}{2}+a_{i_1}+s_1+it)} \Big| 
      & \asymp  \frac{ 1 }{(t-u_2)^{\s_2}(t+u_1)^{\s_1}  }
      \frac{e^{-\frac{\pi}{2}(t-u_2)  } }{ e^{-\frac{\pi}{2}(t+u_1)} }
     =  \frac{ e^{\frac{\pi}{2}(u_1+u_2)} }{(t-u_2)^{\s_2}(t+u_1)^{\s_1}  }    \\
     &\ll  \frac{(u_1^2+u_2^2)}{t^2} 
     \frac{ e^{\frac{\pi}{2}(u_1+u_2)} }{(t-u_2)^{\s_2}(t+u_1)^{\s_1}}, 
\end{align*}
since $u_1^2 +u_2^2 \asymp t^2$. 
By considering cases, we find that  
\begin{equation*}
     \frac{ 1 }{(t-u_2)^{\s_2}(t+u_1)^{\s_1}  }
     \ll \begin{cases}
      t^{-\s_1}   & \text{ if } u_1, u_2 >0, \\
     t^{-\s_1-\s_2} & \text{ if } u_1>0, u_2 <0, \\
     1 & \text{ if } u_1 <0, u_2 >0, \\
     t^{-\s_2} & \text{ if } u_1, u_2 <0
     \end{cases}
\end{equation*}
and thus 
\begin{equation*}
      \Big|\frac{\Gamma(\frac{1}{2}-b_{i_2}-s_2+it)}{\Gamma(\frac{1}{2}+a_{i_1}+s_1+it)} \Big| 
      \ll \begin{cases}
      \frac{(u_1^2+u_2^2)}{t^{2+\s_1}} 
      e^{\frac{\pi}{2}(u_1+u_2)}     & \text{ if } u_1, u_2 >0, \\
      \frac{(u_1^2+u_2^2)}{t^{2+\s_1+\s_2}} 
      e^{\frac{\pi}{2}(u_1+u_2)}    & \text{ if } u_1>0, u_2 <0, \\
      \frac{(u_1^2+u_2^2)}{t^{2}} 
      e^{\frac{\pi}{2}(u_1+u_2)}   & \text{ if } u_1 <0, u_2 >0, \\
     \frac{(u_1^2+u_2^2)}{t^{2+\s_2}} 
      e^{\frac{\pi}{2}(u_1+u_2)}   & \text{ if } u_1, u_2 <0
      \end{cases}
\end{equation*}
Observe that the assumptions $\s_1+\s_2 \le 1$ and $\sigma_i\geq0$ imply $t^{-2-\s_i}, t^{-2} \le t^{-1-\s_1-\s_2}$ for $i=1,2$. 
Inserting these bounds implies \eqref{asymptoticbound}.

\noindent Case (vi): $(i,j)=(2,3)$. In this case, $\tfrac{t}{2} \le |u_1| \le t-1$ and $t-1 \le |u_2| \le t+1$.  
We demonstrate the proof in the case  $u_1<0$ and $u_2>0$ in which case $1<t+u_1<\frac{t}{2}$ and $-1<t+u_2<1$.  
It follows from \eqref{gammanum} and \eqref{gammaden} that 
\begin{align*}
      \Big|\frac{\Gamma(\frac{1}{2}-b_{i_2}-s_2+it)}{\Gamma(\frac{1}{2}+a_{i_1}+s_1+it)} \Big| 
      &
     \ll  \frac{ e^{\frac{\pi}{2}(t+u_1)} }{(t+u_1)^{\s_1} }    \ll\frac{(u_1^2+u_2^2) e^{\frac{\pi}{2}(u_1+u_2)} }{t^{2}}  \ll\frac{(u_1^2+u_2^2) e^{\frac{\pi}{2}(u_1+u_2)} }{t^{1+\sigma_1+\sigma_2}} ,
\end{align*}
where the last step holds 
provided that $\sigma_1+\sigma_2\leq 1$. The other cases are similar.


\noindent Case (ix): $(i,j)=(3,3)$.
In this case, we have  $t-1 \le |u_1|, |u_2| \le t+1$.   Splitting into cases, depending on the sign of $u_i$ we find that 
\begin{equation}
\begin{split}
  \label{cases1}
 & \text{if }u_1, u_2 > 0, \text{ then } t+u_1 \in [2t-1,2t+1], t-u_2 \in [-1,1] ; \\
  & \text{if }u_1 <0, u_2 > 0, \text{ then } t+u_1, t-u_2 \in [-1,1] ; \\
  & \text{if }u_1 >0, u_2 < 0, \text{ then } t+u_1, t-u_2  \in [2t-1,2t+1] ; \\
  & \text{if }u_1, u_2 < 0, \text{ then } t+u_1 \in [-1,1], t-u_2 \in [2t-1,2t+1]. 
\end{split}
\end{equation}
It follows from the above cases \eqref{cases1} and  \eqref{gammanum} and \eqref{gammaden} that 
\begin{align*}
      \Big|\frac{\Gamma(\frac{1}{2}-b_{i_2}-s_2+it)}{\Gamma(\frac{1}{2}+a_{i_1}+s_1+it)} \Big|
 \ll
 \begin{cases}
 \frac{e^{\frac{\pi}{2}\left(t+u_1\right)}}{(t+u_1)^{\sigma_1}} & \text{ if } u_1,u_2 >0, \\
 1 & \text{ if } u_1 <0, u_2 >0, \\
  \frac{ e^{\frac{\pi}{2}(u_1+u_2)} }{(t-u_2)^{\s_2}(t+u_1)^{\s_1}  }  & 
  \text{ if } u_1 >0, u_2 <0, \\
  \frac{e^{\frac{\pi}{2}\left(-t+u_2\right)}}{(t-u_2)^{\sigma_2}} 
& \text{ if } u_1, u_2 <0.
 \end{cases}
\end{align*}
In this case we have $u_1^2+u_2^2 \asymp t^2$.  Using this along with 
 the size conditions  \eqref{cases1} on $t+u_1$ and $t-u_2$ in the various cases, and the assumptions
$\s_1 +\s_2 \le 1$, $\s_1 \geq0$, and $\s_2 \geq0$, it follows that  \eqref{asymptoticbound} holds.

This completes the proof of \eqref{asymptoticbound}. 
To complete the proof of lemma we must establish \eqref{ratio-gamma-0} in the case that  $\sqrt{t}\leq |u_1|\leq t+1$ and $ |u_2|\leq\sqrt{t}$ or $\sqrt{t}\leq |u_2|\leq t+1$ and $ |u_1|\leq\sqrt{t}$.  The first case is dealt with by considering 
three subcases, depending on whether $|u_1|$ lies in the interval $[\sqrt{t},\tfrac{t}{2}], [\tfrac{t}{2},t-1],$ or $[t-1, t+1]$. 
The second case,  $\sqrt{t}\leq |u_2|\leq t+1$ and $ |u_1|\leq\sqrt{t}$ is treated in the same way and in each case we 
obtain the desired bound
\[
   \Big| \frac{\Gamma(\frac{1}{2}-b_{i_2}-s_2+it)}{\Gamma(\frac{1}{2}+a_{i_1}+s_1+it)} \Big| 
    \ll  \frac{u_1^2+u_2^2}{t^{1+\sigma_1+\sigma_2}}   e^{\frac{\pi}{2} (u_1+u_2)}.
\]
This completes the proof of part (i) of the lemma. 
\end{proof}

\begin{proof}[Proof of Lemma \ref{Stirling} (ii)]

We use the notation from the previous part and consider the following cases:
(i) $  |u_1| \ge t+1$ and  $  |u_2| \ge t+1$, (ii) $|u_1| \ge t+1$ and  $  |u_2| \le t+1$, 
and (iii) $  |u_1| \le t+1$ and  $  |u_2| \ge t+1$. We provide a proof in the first case
and just observe that cases (ii) and (iii) can be treated similarly. 
Note that we have the following inequalities
\begin{equation}
\begin{split}
     \label{cases2}
 & \text{if }u_1, u_2 > 0, \text{ then } t+u_1\geq 2t+1  \text{ and } t-u_2\leq -1; \\
  & \text{if }u_1 <0, u_2 > 0, \text{ then } t+u_1\leq -1 \text{ and } t-u_2\leq -1; \\
  & \text{if }u_1 >0, u_2 < 0, \text{ then } t+u_1\geq 2t+1\text{ and } t-u_2\geq2t+1; \\
  & \text{if }u_1, u_2 < 0, \text{ then } t+u_1\leq-1 \text{ and } t-u_2\geq2t+1.
\end{split}
\end{equation}
Using these facts, it follows from \eqref{gammanum} and \eqref{gammaden} that 
\begin{equation}
\begin{split}
   \label{Gammafracbd}
      \Big|\frac{\Gamma(\frac{1}{2}-b_{i_2}-s_2+it)}{\Gamma(\frac{1}{2}+a_{i_1}+s_1+it)} \Big| 
      \ll
      \begin{cases}
     \frac{ e^{\frac{\pi}{2}(2t+u_1-u_2)} }{|t-u_2|^{\s_2}|t+u_1|^{\s_1}  }   & \text{ if }  u_1,u_2 >0, \\
      \frac{ e^{-\frac{\pi}{2}(u_1+u_2)  } }{ |t-u_2|^{\sigma_2}|t+u_1|^{\sigma_1}  }
      & \text{ if }  u_1 <0,u_2 >0, \\
       \frac{ e^{\frac{\pi}{2}(u_1+u_2)  } }{ |t-u_2|^{\sigma_2}|t+u_1|^{\sigma_1}  } & \text{ if } u_1 >0, u_2 <0, \\
        \frac{ e^{\frac{\pi}{2}(-2t+u_2-u_1)  } }{ |t-u_2|^{\sigma_2}|t+u_1|^{\sigma_1}  }
      &  \text{ if } u_1, u_2 <0.
      \end{cases}
\end{split}
\end{equation}
By the conditions \eqref{cases2} it may be checked that each numerator on the right hand side of \eqref{Gammafracbd}
is bounded by $e^{\frac{\pi}{2}|u_1+u_2|}$. Since $|u_1|^2+|u_2|^2\geq t^2$ and $|t+u_1|, |t-u_2|\geq1$ in all cases,
we get
 \[ \Big|\frac{\Gamma(\frac{1}{2}-b_{i_2}-s_2+it)}{\Gamma(\frac{1}{2}+a_{i_1}+s_1+it)} \Big|\ll \frac{|u_1|^2+|u_2|^2}{t^2}e^{\frac{\pi}{2}|u_1+u_2|}, \] as desired.
\end{proof}

\end{document}